\documentclass{amsart}

\usepackage{latexsym, amssymb, amsmath, amscd}
\usepackage{graphics, psfrag}
\usepackage[all]{xy}

\newcommand{\Si}{\Sigma}
\newcommand{\Ga}{\Gamma}
\newcommand{\Up}{\Upsilon}

\newcommand{\ep}{\epsilon}
\newcommand{\la}{\lambda}
\newcommand{\si}{\sigma}

\newcommand{\bC}{\mathbb{C}}
\newcommand{\bE}{\mathbb{E}}

\newcommand{\bL}{\mathbb{L}}
\newcommand{\bP}{\mathbb{P}}
\newcommand{\bQ}{\mathbb{Q}}
\newcommand{\bR}{\mathbb{R}}
\newcommand{\bZ}{\mathbb{Z}}

\newcommand{\cB}{\mathcal{B}}
\newcommand{\cC}{\mathcal{C}}
\newcommand{\cD}{\mathcal{D}}
\newcommand{\cE}{\mathcal{E}}
\newcommand{\cF}{\mathcal{F}}
\newcommand{\cI}{\mathcal{I}}
\newcommand{\cK}{\mathcal{K}}
\newcommand{\cIX}{\mathcal{IX}} 
\newcommand{\cL}{\mathcal{L}}
\newcommand{\cM}{\mathcal{M}}

\newcommand{\cO}{\mathcal{O}}

\newcommand{\cT}{\mathcal{T}}
\newcommand{\cU}{\mathcal{U}}
\newcommand{\cV}{\mathcal{V}}
\newcommand{\cX}{\mathcal{X}}

\newcommand{\age}{\mathrm{age}}
\newcommand{\Aut}{\mathrm{Aut}}

\newcommand{\ch}{\mathrm{ch}}
\newcommand{\Conj}{\mathrm{Conj}}
\newcommand{\Cone}{\mathrm{Cone}}
\newcommand{\ev}{\mathrm{ev}}
\newcommand{\Ext}{\mathrm{Ext}}
\newcommand{\Hom}{\mathrm{Hom}}
\newcommand{\id}{\mathrm{id}}
\newcommand{\Ker}{\mathrm{Ker}}
\newcommand{\Ntor}{N_{\mathrm{tor}} }
\newcommand{\Ob}{\mathrm{Ob}}
\newcommand{\orb}{\mathrm{orb}}

\newcommand{\pr}{\mathrm{pr}}
\newcommand{\pt}{\mathrm{pt}}
\newcommand{\rank}{\mathrm{rank}}

\newcommand{\rig}{\mathrm{rig}}
\newcommand{\Spec}{\mathrm{Spec}}

\newcommand{\td}{\mathrm{td}}
\newcommand{\tw}{\mathrm{tw}}
\newcommand{\val}{\mathrm{val}}
\newcommand{\vir}{ \mathrm{vir} }

\newcommand{\bSi}{ { \mathbf{\Si} }  }

\newcommand{\be}{ {\mathbf{e}} }
\newcommand{\bh}{ {\mathbf{h}} }

\newcommand{\uu}{ {\mathbf{u}} }
\newcommand{\bv}{ {\mathbf{v}} }
\newcommand{\bw}{ {\mathbf{w}} }
\newcommand{\bx}{ {\mathbf{x}} }

\newcommand{\bfV}{\mathbf{V}}
\newcommand{\bfO}{\mathbf{O}}

\newcommand{\fl}{\mathfrak{l}}
\newcommand{\fp}{\mathfrak{p}}
\newcommand{\fx}{\mathfrak{x}}
\newcommand{\fy}{\mathfrak{y}}

\newcommand{\vGa}{ {\vec{\Ga }} }

\newcommand{\vmu}{ {\vec{\mu }} }

\newcommand{\vc}{\vec{c}}
\newcommand{\vd}{\vec{d}}
\newcommand{\vf}{\vec{f}}
\newcommand{\vg}{\vec{g}}
\newcommand{\vi}{ {\vec{i}} }
\newcommand{\vk}{\vec{k}}

\newcommand{\vs}{\vec{s}}

\newcommand{\tsi}{ {\widetilde{\si}} }
\newcommand{\tSi}{ {\widetilde{\Si}} }
\newcommand{\tchi}{ {\widetilde{\chi} }}

\newcommand{\tM}{\widetilde{M}}
\newcommand{\tN}{\widetilde{N}}

\newcommand{\tT}{ {\widetilde{T}} }

\newcommand{\tcC}{ {\widetilde{\cC}} }

\newcommand{\tb }{\widetilde{b}}

\newcommand{\te }{\widetilde{e}}
\newcommand{\tf }{\widetilde{f}}

\newcommand{\tit}{\widetilde{t}}
\newcommand{\tm }{\widetilde{m}}

\newcommand{\tp }{\widetilde{p}}

\newcommand{\tu }{\widetilde{u}}

\newcommand{\baN}{\bar{N}}
\newcommand{\bab}{\bar{b}}
\newcommand{\bIX}{\overline{\cI}\cX}

\newcommand{\hG}{\widehat{G}}

\newcommand{\Mbar}{\overline{\cM}}

\newcommand{\MgX}{\Mbar_{g,n}(X,\beta)}
\newcommand{\GgX}{G_{g,n}(X,\beta)}

\newcommand{\MgcX}{\Mbar_{g,n}(\cX,\beta)}
\newcommand{\MgXi}{\Mbar_{g,\vi}(\cX,\beta)}
\newcommand{\GgXi}{G_{g,\vi}(\cX,\beta)}

\newcommand{\MBG}{\Mbar_{g,n}(\cB G)}
\newcommand{\MBGc}{\Mbar_{g,\vc}(\cB G)}

\newcommand{\edge}{ {\tiny\begin{array}{c}e\in E(\Ga) \\(e,v),(e,v')\in F(\Ga)\end{array}} }
\newcommand{\vone}{ {v\in V^1(\Ga),(e,v)\in F(\Ga)}  }

\newcommand{\half}{ {\frac{1}{2}} }
\newcommand{\onerd}{ {\frac{1}{3}} }
\newcommand{\tword}{ {\frac{2}{3}} }
\newcommand{\fourd}{ {\frac{4}{3}} }
\newcommand{\lra}{\longrightarrow}

\newtheorem{theorem}[equation]{Theorem}

\newtheorem{example}[equation]{Example}
\newtheorem{definition}[equation]{Definition}
\newtheorem{remark}[equation]{Remark}

\newtheorem{lemma}[equation]{Lemma}
\newtheorem{corollary}[equation]{Corollary}
\newtheorem{assumption}[equation]{Assumption}

\begin{document}
\setcounter{page}{1}
%
%
\long\def\replace#1{#1}

\title[Localization in Gromov-Witten Theory]{Localization in Gromov-Witten Theory
and Orbifold Gromov-Witten Theory}
    
\author{Chiu-Chu Melissa Liu}
\address{Chiu-Chu Melissa Liu, Department of Mathematics, Columbia University,
2990 Broadway, New York, NY 10027}
\email{ccliu@math.columbia.edu}

\subjclass[2000]{Primary 14N35; Secondary 14H10}
\keywords{Gromov-Witten invariants, toric varieties, virtual localization, orbifold Gromov-Witten invariants, toric DM stacks}

\begin{abstract}
In this expository article, we explain how to use localization to compute Gromov-Witten invariants of 
smooth toric varieties and orbifold Gromov-Witten invariants of smooth toric Deligne-Mumford stacks.
\end{abstract}  

\maketitle

\tableofcontents

\section{Introduction}

Let $X$ be a smooth projective variety over $\bC$. Naively, 
Gromov-Witten invariants count parametrized algebraic curves of $X$;
more precisely, they are intersection numbers on moduli spaces
of stable maps to $X$.
Let $\Mbar_{g,n}(X,\beta)$ be the Kontsevich's moduli space of $n$-pointed, genus
$g$, degree $\beta$ stable maps $f:(C,x_1,\ldots, x_n)\to X$, where $\beta =f_*[C] \in H_2(X;\bZ)$. 
It is a proper Deligne-Mumford stack with a perfect obstruction theory of virtual dimension 
\begin{equation}\label{eqn:intro-dvir}
d^\vir= \int_\beta c_1(T_X) + (\dim X-3)(1-g) +n,
\end{equation}
where $\int$ stands for the pairing between the (rational) homology and cohomology.
Given $i\in \{1,\ldots, n\}$, there is an evaluation map $\ev_i: \Mbar_{g,n}(X,\beta)\to X$ which sends
a moduli point $[f:(C,x_1,\ldots, x_n)\to X ]\in \Mbar_{g,n}(X,\beta)$
to $f(x_i)\in X$, and there is a line bundle $\bL_i$ over $\Mbar_{g,n}(X,\beta)$ whose
fiber at the moduli point $[f:(C,x_1,\ldots, x_n) \to X]$ is the cotangent line
$T^*_{x_i}C$ at the $i$-th marked point $x_i$.  Gromov-Witten invariants of $X$ are  defined to be
\begin{equation}
\langle \tau_{a_1}(\gamma_1), \ldots, \tau_{a_n} (\gamma_n)\rangle^X_{g,\beta}
:= \int_{[\Mbar_{g,n}(X,\beta)]^\vir} \prod_{i=1}^n (\ev_i^*\gamma_i \psi_i^{a_i} ) \in \bQ
\end{equation}
where  $\gamma_1,\ldots, \gamma_n\in H^*(X;\bQ)$,  $\psi_i = c_1(\bL_i) \in H^2(\Mbar_{g,n}(X,\beta);\bQ)$, 
and
$$
[\Mbar_{g,n}(X,\beta)]^\vir \in H_{2 d^\vir}(X;\bQ)
$$ 
is the virtual fundamental class (Li-Tian \cite{LiTi1}, Behrend-Fantechi \cite{BeFa}).

When $X$ is a toric variety, the torus action on $X$ induces torus
actions  on moduli spaces of stable maps to $X$. 
By virtual localization (Graber-Pandharipande \cite{GrPa}, see also Behrend \cite{Be2} and
Kresch \cite{Kr}),   
\begin{equation}\label{eqn:intro-virtual}
\int_{[\Mbar_{g,n}(X,\beta)]^\vir} \prod_{i=1}^n (\ev_i^*\gamma_i \psi_i^{a_i} ) 
= \sum_{F} \int_{[F]^\vir} \frac{i_F^*  \prod_{i=1}^n (\ev_i^*\gamma^T_i  (\psi^T_i)^{a_i} )}{e_T(N^\vir_F)}, 
\end{equation}
where 
\begin{itemize}
\item $T$ is the torus acting on $X$ and on $\Mbar_{g,n}(X,\beta)$, 
\item  the sum on the right hand side of \eqref{eqn:intro-virtual}
is over connected components of the set of $T$-fixed points  
in $\Mbar_{g,n}(X,\beta)$, 
\item $\gamma_i^T \in H_T^*(X;\bQ)$ is a $T$-equivariant lift of $\gamma_i$,
\item $\psi^T_i \in H_T^2(\Mbar_{g,n}(X,\beta);\bQ)$ is a $T$-equivariant lift of $\psi_i$.
\item $i_F^*: H_T^*(\Mbar_{g,n}(X,\beta);\bQ) \to H_T^*(F;\bQ)$ is induced by
the inclusion map $i_F: F\to \Mbar_{g,n}(X,\beta)$, 
\item $e_T(N^\vir_F)$ is the $T$-equivariant
Euler class of the virtual normal bundle $N^\vir_F$ of $F$ in $\Mbar_{g,n}(X,\beta)$.
\end{itemize} 
Up to a finite morphism, each connected component $F$ is a product of moduli spaces
of stable curves (with marked points).  $F$ is a proper smooth DM stack, 
and $[F]^\vir$ is the usual fundamental class $[F]\in H_*(F;\bQ)$. 
The right hand side of \eqref{eqn:intro-virtual} can be expressed in terms of Hodge 
integrals, which are intersection numbers on moduli spaces of stable curves. 
The terminology ``virtual localization'' was introduced in \cite{GrPa} and the term ``Hodge integral'' 
was introduced in \cite{FaPa} precisely to study the virtual localization formula in \cite{GrPa}.
Algorithms of computing Hodge integrals are known; a brief review of the relevant results will be 
given in Section \ref{sec:hodge}.  
This gives an algorithm of evaluating Gromov-Witten invariants for
any smooth projective toric varieties, in all genera and all degrees.
Indeed, this algorithm was first described by Kontsevich for
genus zero Gromov-Witten invariants of $\bP^r$ in 1994 \cite{Ko2}, before 
the construction of virtual fundamental class and the proof of virtual 
localization. The moduli spaces $\Mbar_{0,n}(\bP^r,d)$ of genus zero stable maps 
to $\bP^r$ are proper {\em smooth} DM stacks, so there exists a fundamental class 
$[\Mbar_{0,n}(\bP^r,d)]\in H_*(\Mbar_{0,n}(\bP^r,d);\bQ)$, 
and one may apply the classical Atiyah-Bott localization formula \cite{AtBo}
in this case. H. Spielberg derived a formula of genus 0 Gromov-Witten 
invariants of smooth toric varieties in his thesis \cite{Sp}.

For a noncompact smooth toric variety $X$, Gromov-Witten invariants
are usually not defined, but one may use the right hand side 
of \eqref{eqn:intro-virtual} to define $T$-equivariant Gromov-Witten invariants 
of $X$. They are elements in the fractional field of $H^*(BT;\bQ)$, the rational
equivariant cohomology ring of the classifying space $BT$ of $T$. 

Chen-Ruan developed Gromov-Witten theory for symplectic orbifolds \cite{CR2}.  
The algebraic counterpart, the orbifold Gromov-Witten theory for
smooth Deligne-Mumford (DM) stacks, was developed by 
Abramovich-Graber-Vistoli \cite{AGV1, AGV2}. 
Orbifold Gromov-Witten invariants of a smooth DM stack $\cX$ are defined 
as intersection numbers on moduli spaces of twisted stable maps to $\cX$.  
When $\cX$ is a smooth toric DM stack, the torus action
on $\cX$ induces torus actions on moduli spaces of
twisted stable maps to $\cX$. By virtual localization, 
orbifold Gromov-Witten invariants of a smooth toric DM stack 
can be expressed in terms of Hurwitz-Hodge integrals, 
which are intersections numbers of moduli spaces
of twisted stable maps to $\cB G =[\pt/G]$, the classifying
space of a finite group $G$.  Algorithms for computing Hurwitz-Hodge 
integrals are known; a brief review of the relevant results will be 
given in Section \ref{sec:hurwitz-hodge}. 

The goal of this article is to provide details of the localization calculations described above.
In Section \ref{sec:intersection},  we review equivariant intersection theory and localization.
In Section \ref{sec:GWreview}, we give a brief review of Gromov-Witten theory. 
In Section \ref{sec:toric-varieties}, we give a brief review of smooth toric varieties, and introduce
toric graphs. In Section \ref{sec:toricGW}, we use virtual localization
to derive a formula for Gromov-Witten invariants
of smooth toric varieties in terms of Hodge integrals. 
Most of Section \ref{sec:toricGW} is straightforward generalization of the
 $\bP^r$ case discussed in \cite{Ko2} (genus 0) and 
\cite[Section 4]{GrPa}, \cite[Section 4]{Be2} (higher genus); 
see also \cite[Chapter 27]{HKKPTVVZ}. 
Smooth DM stacks, orbifold Gromov-Witten theory, and smooth toric DM stacks
are reviewed in Section \ref{sec:DMstacks}, Section \ref{sec:orbGWreview},
and Section \ref{sec:toric-stacks}, respectively.
In Section \ref{sec:toricGW-orb}, we use virtual localization
to derive a formula of orbifold Gromov-Witten invariants
of smooth toric DM stacks in terms of abelian Hurwitz-Hodge integrals.
Our main reference of Section \ref{sec:toricGW-orb} is P. Johnson's thesis \cite{Jo},
which contains detailed localization computations for 1-dimensional toric
DM stacks. D. Ross's recent preprint \cite{Ro} contains localization computations 
for 3-dimensional Calabi-Yau toric DM stacks.   

\subsection*{Acknowledgments}
I wish to thank Dan Abramovich, Lev Borisov, Dan Edidin, Ezra Getzler, Tom Graber, Paul Johnson,
Etienne Mann, Zhengyu Zong for helpful communications, and the referee for his or her comments.
Special thanks go to Tom Graber and Paul Johnson for their help with
orbifold Gromov-Witten theory and virtual localization.

\section{Equivariant Intersection Theory and Localization} 
\label{sec:intersection}

In this section, we review equivariant intersection
theory and localization. In Section \ref{sec:HG} -- Section \ref{sec:RR}
we discuss equivariant cohomology of topological spaces
and localization on smooth manifolds. In Section \ref{sec:AG} we give
a brief summary of equivariant intersection theory
on schemes and Deligne-Mumford stacks in terms of equivariant
Chow groups and equivariant operational Chow cohomology groups
\cite{EdGr1}. We state the virtual localization formula in Section \ref{sec:VL}.

In this paper, we consider cohomology groups, Chow groups,
and operational Chow groups with rational coefficients.
We write $H^*(\bullet)$, $A_*(\bullet)$, $A^*(\bullet)$
instead of $H^*(\bullet;\bQ)$, $A_*(\bullet;\bQ)$, $A^*(\bullet;\bQ)$. 

\subsection{Equivariant cohomology}\label{sec:HG}
Let $G$ be a Lie group, and let $EG$ be a contractible topological
space on which $G$ acts freely on the {\em right}.
The quotient $BG= EG/G$ is a classifying 
space of principal $G$-bundles, and  the natural projection
$EG \to BG$ is a universal principal $G$-bundle;
$EG$ and $BG$ are defined up to homotopy equivalences.

A $G$-space is a topological space together with
a continuous {\em left} $G$-action. Given a $G$-space $X$,
define a {\em right} $G$-action on $EG\times X$ by 
\begin{equation}\label{eqn:free}
(p,x) \cdot g = (p\cdot  g, g^{-1} \cdot x).
\end{equation}
The {\em homotopy orbit space} $X_G$ is defined
to be the quotient of $EG\times X$ by the  {\em free}
$G$-action \eqref{eqn:free}. 
The $G$-equivariant  cohomology of
$X$ is defined to be the ordinary cohomology
of the homotopy orbit space $X_G$:
$$
H_G^*(X):= H^*(X_G).
$$
In particular, the $G$-equivariant cohomology
of a point $\pt$ is the ordinary cohomology of
the classifying space $BG$:
$$
H_G^*(\pt)=H^*(BG)
$$

\begin{example}[Classifying space of $\bC^*$-bundles]
The Lie group $\bC^*$ acts on $\bC^\infty -\{0\}$ on the right by
$$
v\cdot \lambda  = \lambda  v,\quad \lambda\in \bC^*,\quad v\in \bC^\infty-\{0\}.
$$
$\bC^\infty-\{0\}$ is contractible, and the $\bC^*$-action on $\bC^\infty-\{0\}$ is free.
Therefore (up to homotopy equivalence)
$$
E\bC^*=\bC^\infty-\{0\},\quad
B\bC^*=(\bC^\infty-\{0\})/\bC^*= \bP^\infty ,
$$
where $\bP^\infty$ is the infinite dimensional complex projective space. 
Let $\cO_{\bP^\infty}(-1)$ be the tautological line bundle over
$\bP^\infty$, and let $u$ be the first Chern class of $\cO_{\bP^\infty}(-1)$:
$$
u:= c_1(\cO_{\bP^\infty}(-1))\in H^2(\bP^\infty) = H^2_{\bC^*}(\pt).
$$
Then $H^*_{\bC^*}(\pt)= H^*(B\bC^*)= \bQ[u]$.
\end{example}

In this paper we will consider action by an algebraic torus
$T=(\bC^*)^l$.  Let $\pi_i: BT= (B\bC^*)^l\to  B\bC^*$ be the projection to the $i$-th
factor, and let $u_i=\pi_i^*u\in H^2(BT)$. Then
$$
H^*_T(\pt)=H^*(BT)= \bQ[u_1,\ldots,u_l].
$$

\begin{example}\label{Pr}
Let $\tT=(\bC^*)^{r+1}$ act on the $r$-dimensional 
complex projective space $\bP^r$ by 
$$
(\tilde{t}_0,\ldots,\tilde{t}_r) \cdot[z_0,\ldots, z_r]
= [\tilde{t}_0 z_0,\ldots, \tilde{t}_r z_r],\quad
(\tilde{t}_0,\ldots,\tilde{t}_r)\in \tT,\quad [z_0,\ldots, z_r]\in \bP^r.
$$
For $i=0,\ldots,r$, let $p_i:B\tT=(B\bC^*)^{r+1}\to B\bC^*$
be the projection to the $i$-th factor.
Then $\bP^r_\tT$ can be identified with the total space of the $\bP^r$-bundle 
$$
\bP\bigl(\oplus_{i=0}^r p_i^*(\cO_{\bP^\infty}(-1))\bigr) \to BT.
$$

In general, let $E\to X$ be a rank $(r+1)$ complex vector bundle over a 
topological space $X$, and let $\pi: \bP(E)\to  X$ be 
the projectivization of $E$, which is an $\bP^r$-bundle over $X$.
Then the cohomology $H^*(\bP(E))$ of the
total space $\bP(E)$ is an $H^*(X)$-algebra
generated by $H$ with a single relation
$$
H^{r+1} + c_1(E)H^r +\cdots +  c_{r+1}(E) =0,
$$
where $c_i(E)$ is the $i$-th Chern class of $E$, and $H$ is of degree 2.   

In our case $E = \oplus_{i=0}^r p_i^*\cO_{\bP^\infty}(-1)$, so 
the total Chern class of $E$ is given by
$$
c(E) = \prod_{i=0}^r (1+ \tu_i),\quad \tu_i=p_i^*(c_1(\cO_{\bP^\infty}(-1)).
$$
We have
$$
H^*_{\tT}(\bP^r)= 
H^*(\bP^r_{\tT}) \cong \bQ[H,\tu_0,\ldots, \tu_r]/\langle 
\prod_{i=0}^r (H +\tu_i) \rangle,
$$
where $\bQ[H,\tu_0,\ldots,\tu_r]$ is the ring of polynomials
in $H, \tu_0,\ldots,\tu_r$ with coefficients in $\bQ$, 
and $\langle \prod_{i=0}^r (H+\tu_i)\rangle$ is 
the principal ideal generated by $\prod_{i=0}^r (H+\tu_i)$.
\end{example}

The trivial fiber bundle $EG\times X\to EG$ with base $EG$, fiber
$X$ descends to a fiber bundle  $X_G\to BG$ with base $BG$,
fiber $X$. The inclusion of a fiber,  $i_X: X\to X_G$,
induces a ring homomorphism $i_X^*: H_G^*(X)=H^*(X_G)\to H^*(X)$.

\subsection{Equivariant vector bundles and equivariant characteristic classes}
\label{sec:cG}

Let $G$ be a Lie group.
A continuous map $f:X\to Y$ between $G$-spaces
is called {\em $G$-equivariant} if 
$f(g\cdot x) = g\cdot f(x)$ for all $g\in G$ and $x\in X$.

Let $p: V\to X$ be a (real or complex) vector bundle over a
$G$-space $X$. We say $p:V\to X$ is a $G$-equivariant
vector bundle over $X$ if the following properties
hold.
\begin{itemize}
\item $V$ is a $G$-space.
\item $p$ is $G$-equivariant. 
\item For every $g\in G$, define $\tilde{\phi}_g: V\to V$ by $v\mapsto g\cdot v$,
and $\phi_g: X\to X$ by $x\mapsto g\cdot x$. Then 
$\tilde{\phi}_g$ is a vector bundle map covering $\phi_g$:
$$
\begin{CD}
V  @>{\tilde{\phi}_g}>>  V\\
@V{p}VV    @V{p}VV\\
X  @>{\phi_g}>>   X
\end{CD}
$$
\end{itemize}

\begin{example}
When $X$ is a point, a complex vector bundle
over $X$ is a complex vector space $V$, and 
a $G$-equivariant vector bundle over $X$ is
a representation $\rho: G\to GL(V)$.
\end{example}

Let $\pi:V\to X$ be a $G$-equivariant vector bundle
over a $G$-space $X$. Then $V_G$ is a vector bundle over $X_G$.  
Let $c$ be a characteristic
class of vector bundles (for example, Chern classes
$c_k$ and Chern characters $\ch_k$
for complex vector bundles, or the Euler class $e$ for
oriented real vector bundles). We define the corresponding
$G$-equivariant class $c^G$ by 
$$
c^G(V):= c(V_G)\in H^*(X_G)= H^*_G(X).
$$
If $V$ is a $G$-equivariant complex vector bundle over $X$
then we call $c_k^G(V)\in H^{2k}_G(X)$ (resp. $\ch_k^G(V)\in H^{2k}_G(X)$)
the $G$-equivariant $k$-th Chern class (resp. the $G$-equivariant
$k$-th Chern character) of $V$.  If $V$ is a $G$-equivariant
oriented real bundle of rank $r$ over $X$ then we call $e^G(V)\in H^r_G(X)$ the $G$-equivariant
Euler class of $V$.

\begin{example}\label{Ca}
For any $a\in \bZ$, let
$\bC_a$ be the 1-dimensional representation of $\bC^*$ 
with character $t\mapsto t^a$. Then $\bC_a$ can be viewed
as a $\bC^*$-equivariant vector bundle over a point. We have
$$
(\bC_a)_{\bC^*} = \{ (u,v)\in (\bC^\infty-\{0\})\times \bC\}
/ (u,v) \sim (t u, t^{-a} v)
\cong \cO_{\bP^\infty}(-a)
$$
$$
c_1^{\bC^*}(\bC_a) = au\in H^2_{\bC^*} (\pt)=\bZ u.
$$
\end{example}

\subsection{Push-forward}\label{sec:push-forward}

Let $X$, $Y$ be compact oriented manifolds
of dimension $r$, $s$, respectively, and let
$[X]\in H_r(X)$, $[Y]\in H_s(Y)$ be the fundamental
classes. A continuous map $f:X\to Y$ induces a group homomorphism
$$
f_*: H^k(X)\to H^{k+s-r}(Y)
$$
characterized by 
$$
(f_*\alpha)\cap [Y] = f_*(\alpha\cap [X]) \in H_{r-k}(Y).
$$
In particular, if $s\geq r$ then
$f_*1\in H^{s-r}(Y)$ is the Poincar\'{e} dual of
$f_*[X]\in H_r(Y)$. The push-forward map $f_*:H^*(X)\to H^*(Y)$ is a homomorphism
of $H^*(Y)$-modules:
$$
f_*(\alpha\cup f^*\beta)= (f_*\alpha)\cup \beta,\quad
\alpha\in H^*(X),\ \beta\in H^*(Y).
$$
If $g:Y\to Z$ is a continuous map and $Z$ is a compact oriented manifold
then $g_*\circ f_* = (g\circ f)_*:H^*(X)\to H^*(Z)$.

\begin{example}
\begin{enumerate}
\item  Let $V$ be a rank $q$ oriented real vector bundle
over $Y$, and let $X$ be the transversal intersection
of a section $s:Y\to V$ and the zero section.
Let $f:X\to Y$ be the inclusion. Then 
$f_*1=e(V)\in H^q(Y)$.

\item Let $p_X:X\to  \pt$  be the constant map to a point.
Then $p_{X*}:H^*(X)\to H^*(\pt)\cong\bQ$ can be identified with
$\int_X$.
\end{enumerate}
\end{example}

Suppose that a Lie group $G$ acts on $X$ and on $Y$, and
let $[X]^G\in H_r^G(X)$, $[Y]^G\in H_s^G(X)$ be $G$-equivariant
fundamental class, where $H_*^G(X)$ is the $G$-equivariant
homology groups with rational coefficients,  constructed
from $G$-invariant cycles in $X$. A $G$-equivariant map
$f:X\to Y$ induces a group homomorphism $f_*: H_k^G(X)\to H_k^G(Y)$.
It also induces
$$
f_*: H^k_G(X) \to H^{k+s-r}_G(Y)
$$
characterized by
$$
f_*(\alpha^G)\cap [Y]^G= f_*(\alpha^G\cap [X]^G)\in H^G_{k-r}(X).
$$
In particular, if $s\geq r$ then $f_*1\in H^{s-r}_G(Y)$ is the
equivariant Poincar\'{e} dual of $f_*[X]^G\in H_r^G(Y)$.

We have the following commutative diagram:
$$
\begin{CD}
H^k_G(X) @>{f_*}>> H^{k+s-r}_G(Y)\\
@ V{i_X^*}VV @VV{i_Y^*}V\\
H^k(X) @>{f_*}>> H^{k+s-r}(Y)
\end{CD}
$$
where $i_X^*$, $i_Y^*$  are defined as in the last paragraph
of Section \ref{sec:HG}. If $s\geq r$ then
$f_*1\in H^{s-r}(Y)$ is the equivariant Poincare dual
of $f_*[X]^G\in H_r^G(Y)$.

The push-forward map $f_*:H^*_G(X)\to H^*_G(Y)$ is a homomorphism
of $H^*_G(Y)$-modules:
$$
f_*(\alpha^G\cup f^*\beta^G)= (f_*\alpha^G)\cup \beta^G,\quad
\alpha^G\in H^*_G(X),\ \beta^G\in H^*_G(Y).
$$

If $G$ acts on another compact oriented manifold $Z$, and
$g:Y\to Z$ is a $G$-equivariant map,
then $g_*\circ f_* = (g\circ f)_*:H_k^G(X)\to H_k^G(Z)$,
$H^k_G(X)\to H^{k+\dim Z-\dim X}_G(Z)$.

\begin{example}
\begin{enumerate}
\item  Let $V$ be a $G$-equivariant rank $q$ oriented real vector bundle
over $Y$, and let $X$ be the transversal intersection
of a $G$-equivariant section $s:Y\to V$ and the zero section.
Let $f:X\to Y$ be the inclusion. Then 
$f_*1=e^G(V)\in H^q_G(Y)$.

\item Let $p_X:X\to  \pt$  be the constant map to a point.
Then $(p_X)_*: H^*_G(X)\to H^*_G(\pt)=H^*(BG)$ is denoted
by $\int_{[X]^G}$.
\end{enumerate}
\end{example}

\subsection{Localization}\label{sec:localization}
Suppose that $T=(\bC^*)^l$ acts on a compact oriented manifold $M$, and suppose
that each connected component of the $T$ fixed points
set $M^T\subset M$ is a compact orientable submanifold of $M$. 
Let $F_1,\ldots, F_N$ be the connected components of $M^T$. Then
$(F_j)_T = F_j\times BT$, so  
$$
H_T^*(F_j) = H^*(F_j\times BT)\cong H^*(F_j)\otimes_\bQ H^*(BT)=H^*(F_j)\otimes_\bQ R_T
$$
where $R_T=H^*(BT)=\bQ[u_1,\ldots,u_l]$.
Let $Q_T=\bQ(u_1,\ldots,u_l)$ be the 
fractional field of $R_T$.
The equivariant Euler class $e^T(N_j)$ of the
normal bundle $N_j$ of $F_j$ in $M$ is invertible
in $H^*(F_j)\otimes_\bQ Q_T$. The inclusion
$i_j:F_j\to M$ induces a homomorphism $(i_j)_*:H_T^*(F_j)\to H_T^*(M)$
of $R_T$-modules and can
be extended to 
$$
(i_j)_*: H_T^*(F_j)\otimes_{R_T}Q_T \to  H_T^*(M)\otimes_{R_T}Q_T.
$$

\begin{theorem}[{Atiyah-Bott localization formula  \cite{AtBo}}] 
\item If $\alpha^T\in H_T^*(X)$ then
\begin{equation}\label{eqn:localization}
\alpha^T =\sum_{j=1}^N (i_j)_*\frac{i_j^*\alpha^T}{e^T(N_j)}.
\end{equation}
\end{theorem}

\begin{corollary}[{integration formula \cite[Equation (3.8)]{AtBo}}] If $\alpha\in H_T^*(X)$ then
\begin{equation}\label{eqn:integral}
\int_{[X]^T} \alpha^T =\sum_{j=1}^N \int_{[F_j]^T}\frac{i_j^*\alpha^T}{e^T(N_j)}.
\end{equation}
\end{corollary}

Each term of the right hand side is a rational function 
in $u_1,\ldots,u_l$, while the left hand side is a {\em polynomial} 
in  $u_1,\ldots,u_l$. If $\alpha\in H^k_T(X)$ then
$\int_{[X]^T}\alpha^T \in H^{k-\dim X}_T(\pt)$. In particular,
\begin{itemize}
\item If $k=\dim X$ then $\int_{[X]^T}\alpha^T\in \bQ$.
\item If $k<\dim X$ or if $k-\dim X$ is odd,  then $\int_{[X]^T}\alpha^T=0$.
\end{itemize}
$H^{2m}_T(\pt)$ is the space of polynomials in $u_1,\ldots, u_l$ with 
$\bQ$ coefficients, homogeneous of degree $m$.

We also have $i_{j_1}^* i_{j_2*} =0$ if $j_1\neq j_2$. Therefore the
inclusion $i: M^T=\cup_{j=1}^N F_j \to M$ induces an isomorphism
$$
i_*: H_T^*(M^T)\otimes_{R_T}Q_T
=\bigoplus_{j=1}^N H^*(F_j)\otimes_\bQ Q_T
\to H^*_T(M)\otimes_{R_T}Q_T.
$$

\begin{example}
Let $\tT=(\bC^*)^{r+1}$ act on $\bP^r$ as in Example \ref{Pr}.
Then the fixed points set consists of $(r+1)$ isolated points.
$$
(\bP^r)^{\tT}=\{ p_0=[1,0,\ldots,0],\quad p_1=[0,1,0,\ldots,0],\quad p_r=[0,\ldots, 0,1] \}.
$$
Let $D_j$ be the $\tT$-invariant divisor defined by 
$x_j=0$. Then $x_j$ is a $\tT$-equivariant section of 
the $\tT$-equivariant line bundle $\cO_{\bP^r}(D_j)$.
$\{x_k\mid k\neq j\}$ defines a $T$-equivariant
section $s_j$ of the rank $r$ vector bundle 
$\oplus_{k\neq j}\cO_{\bP^r}(D_j)$. The section
$s_j$  intersects the zero section transversally
at a single point $p_j$. Let $i_j:p_j\to \bP^r$ be the inclusion, 
and let $h_j=(c_1)_{\tT}(\cO(D_j))\in H^2_{\tT}(\bP^r)$. Then
$$
(i_j)_*1=\prod_{k\neq j}h_k\in H^{2r}_{\tT}(\bP^r).
$$
We have
$$
i_j^* h_k = c_1^{\tT}(\cO_{\bP^r}(D_k)_{p_j})=\tu_k-\tu_j \in H_{\tT}^2(p_j).
$$
$D_0\cap D_1\cap \cdots \cap D_r$ is empty, so 
$$
h_0h_1\cdots h_r=0.
$$

For a fixed $k\in\{1,\ldots,r\}$,  $(i_j)^*(h_k-h_0)= \tu_k-\tu_0$
for all $j\in \{0,\ldots, r\}$.  By localization, 
$h_k-h_0=\tu_k-\tu_0$. Define
$$
H= h_0 - \tu_0 = h_1 -\tu_1 =\cdots = h_r - \tu_r.
$$
Then
$$
(i_j)_*1 =\prod_{k\neq j}(H +\tu_k),\quad i_j^* H = -\tu_j,\quad 
\prod_{j=0}^l(H+\tu_j)=0,
$$
where the last identity agrees with the relation derived in 
Example \ref{Pr}.
\end{example}

\begin{definition}[equivariant integration on noncompact spaces]\label{df:residue}
Suppose that $X$ is a {\em noncompact} oriented manifold, but
$X^T$ is a finite union of compact, orientable submanifolds
$F_1,\ldots,F_N$. We define
$\int_{[X]^T}: H^*_T(M)\to \bQ(u_1,\ldots,u_l)$ by 
$$
\int_{[X]^T} \alpha^T =\sum_{j=1}^N \int_{[F_j]^T}\frac{i_j^*\alpha^T}{e^T(N_j)}.
$$
\end{definition}

In the above Definition \ref{df:residue}, $\int_{[X]^T}\alpha^T$ can be nonzero even if 
$\alpha^T\in H^k_T(X)$ and $k<\dim M$. The following is a simple example.
\begin{example}
Let $T=(\bC^*)^2$ act on $\bC^2$ by
$(t_1,t_2)\cdot (z_1, z_2)=(t_1 z_1, t_2 z_2)$ for
$(t_1,t_2)\in T$, $(z_1,z_2)\in \bC^2$. Then
$$
\int_{[\bC^2]^T }1 =\frac{1}{e^T(T_0\bC^2)}=\frac{1}{u_1u_2}.
$$
\end{example}

\subsection{Equivariant Riemann-Roch}\label{sec:RR}
Let $X$ be a compact complex manifold with a holomorphic $T$-action. 
The constant map $X\to \pt$ also induces an additive map between
equivariant $K$-theories:
$$
\pi_!: K_T(X)\to K_T(\pt),\quad
\cE\mapsto \sum_i (-1)^i H^i(X,\cE)
$$ 
where $\cE$ is a $T$-equivariant holomorphic vector bundle
over $X$, and $H^i(X,\cE)$ are the sheaf cohomology groups, which
are representations of $T$.

A representation of $T$ is determined by its $T$-equivariant Chern 
character $\ch^T$. We can compute $\ch^T (\pi_!\cE)$ by Grothendieck-Riemann-Roch
(GRR) theorem and the Atiyah-Bott localization formula. Applying
GRR to  the fibration $\pi: X_T \to BT$, we have
$$
\ch^T (\pi_! \cE)  = \int_{[X]^T} \ch^T(\cE)\td^T(TX)
$$
where $\td^T(TX)$ is the $T$-equivariant Todd  class of
the tangent bundle $TX$ of $X$. 
By the integration formula \eqref{eqn:integral},
$$
 \int_{[X]^T } \ch^T(\cE)\td^T(TX)
=\sum_{j=1}^N \int_{ [F_j]^T }\frac{i_j^*\left( \ch^T(\cE) \td^T(TX)\right)}{e^T(N_j)}. 
$$

We now specialize to the case where $F_j$ are isolated points.
We write $p_1,\ldots, p_N$ instead of $F_1,\ldots, F_N$.
Let $r=\dim_\bC X$, and let
$$
x_{j,1},\ldots, x_{j,r} \in 
H^2_T(\pt) = \bigoplus_{i=1}^l \bQ u_i 
$$
be the weights of the $T$-action on the tangent
space $T_{p_j}X$ of $X$ at $p_j$.   
Then
$$
i_j^* \td^T(TX) =\prod_{k=1}^r \frac{x_{j,k}}{1-e^{-x_{j,k}}},\quad
e^T(N_j)= e^T(T_{p_j}X) = \prod_{k=1}^r x_{j,k}.
$$
Let $m=\rank_\bC \cE$, and let
$$
y_{j,1},\ldots, y_{j,m} \in 
H^2_T(\pt) 
$$
be the weights of the $T$-action on the fiber $\cE_{p_j}$ 
of $\cE$ at $p_j$. Then
$$
i_j^* \ch^T(\cE) = \sum_{l=1}^m e^{y_{j,l} }.
$$
Therefore
\begin{equation}\label{eqn:GRR}
\ch^T(\pi_!\cE) =\sum_{j=1}^N \frac{\sum_{l=1}^m e^{y_{j,l}} }
{\prod_{k=1}^r (1-e^{-x_{j,k}}) }.
\end{equation}

\begin{example}\label{Pone}
Suppose that $T=(\bC^*)^l$ acts on  
on $\bP^1$, and let $L\to \bP^1$ be a $T$-equivariant line bundle.
The weights of the $T$-actions on 
$T_0\bP^1$, $T_\infty \bP^1$, $L_0$, $L_\infty$
are given by $u$, $-u$, $w$, $w-au$, respectively,
where $u,w\in H^2_T(\pt;\bQ)$ and $a\in \bZ$ is the degree of $L$. 
Then
\begin{eqnarray*}
&& \ch^T(H^0(\bP^1,L)-H^1(\bP^1,L)) = \int_{[\bP^1]^T}\ch^T(L)\td^T(T\bP^1)\\
&=&\frac{e^w}{1-e^{-u}} + \frac{e^{w-au}}{1-e^u}
=\begin{cases}
\sum_{i=0}^a e^{w-iu}, & a\geq 0\\
- \sum_{i=1}^{-a-1} e^{w+iu} & a< 0
\end{cases}
\end{eqnarray*}
Indeed,
$$
H^0(\bP^1,L)=\begin{cases}
\sum_{i=0}^a e^{w-iu}, & a\geq 0,\\
0, & a<0. \end{cases}\quad\quad
H^1(\bP^1,L)=\begin{cases}
0, & a\geq 0,\\
\sum_{i=1}^{-a-1} e^{w+iu}, & a<0.
\end{cases}
$$
\end{example}

\subsection{Basic intersection theory in algebraic geometry} \label{sec:intersectionAG}

We refer to \cite{Fu1} for intersection theory
on schemes, and to \cite{Vi} for intersection theory on Deligne-Mumford stacks.

Given a scheme or a Deligne-Mumford
stack $M$ over $\bC$, let $A_*(M)=\oplus_k A_k(M)$
be the Chow groups of $M$ with rational coefficients,
and let $A^*(M) =\oplus_k A^k(M)$ be the operational
Chow cohomology groups (see \cite{Fu1}) with rational coefficients. 
There is a cap product
$$
A^k(M)\times A_l(M)\to A_{l-k}(M),\quad (\alpha,\beta)\mapsto \alpha\cap \beta,
$$
and a group homomorphism $\deg:A_0(M)\to \bQ$. If $M$ is a scheme and
$p\in M$ is a smooth point then $\deg[p]=1$; if $M$ is a Deligne-Mumford
stack and $p\in M$ is a smooth point with automorphism group $\Aut(p)$ (which is
a finite group) then  $\deg[p]=1/|\Aut(p)|$. (In this paper, $|S|$ denotes
the cardinality of a finite set $S$.)  We extend
$\deg$ to $A_*(X)$ by sending $A_k(X)$ to zero for $k\neq 0$.

If $M$ is a proper smooth scheme  or a proper smooth Deligne-Mumford stack of dimension $r$, 
then there is a fundamental class
$$
[M]\in A_r(M).
$$
We define $\int_M:  A^*(M)\to \bQ$ by
$$
\int_{M}\alpha =\deg(\alpha\cap[M])\in \bQ.
$$

\subsection{Equivariant intersection theory in algebraic geometry}
\label{sec:AG}

We have discussed equivariant intersection theory on topological
spaces, and localization of equivariant cohomology on manifolds.
We now discuss equivariant intersection theory on schemes
and Deligne-Mumford stacks, and virtual localization. This is
similar to the discussion in Section \ref{sec:HG} -- \ref{sec:RR}, so 
we will just give a brief summary in the case $G=T=(\bC^*)^l$. 
We refer to \cite{EdGr1} for equivariant intersection theory on schemes and algebraic spaces.

Suppose that $T=(\bC^*)^l$ acts on a scheme or a Deligne-Mumford stack $M$ over $\bC$. 
The $T$-equivariant  operational Chow cohomology of $M$ is defined to be 
the ordinary operational Chow cohomology of the quotient stack $[M/T]$:
$$
A_T^*(M) := A^*([M/T]).
$$
In particular, 
$$
A_T^*(\pt)=A^*([\pt/T])=\bQ[u_1,\ldots,u_l].
$$
The $T$-equivariant Chow groups $A_*^T(M)$ is constructed from
$T$-invariant cycles in $M$. We refer to \cite{EdGr1} for
the construction.

A $T$-equivariant vector bundle $V\to M$ corresponds
to a vector bundle $[V/T]\to [M/T]$.  Define
the  $T$-equivariant Chern classes and Chern characters
of $E$ by 
$$
c_k^T(V):= c_k([V/T]) \in A_T^k(M),\quad
\ch_k^T(V):= \ch_k([V/T])\in A_T^k(M).
$$

Now suppose that $M$ is a proper Deligne-Mumford
stack with a $T$-equivariant perfect obstruction theory of virtual
dimension $m$. In particular,  locally there exists a two term complex
of $T$-equivariant vector bundles $E\to F$ over $M$, where
$\rank F -\rank E=m$,  such that
we have an exact sequence of $T$-equivariant sheaves:
$$
0\to \cT^1\to F^\vee\to E^\vee \to \cT^2\to 0.
$$
The perfect obstruction theory defines a $T$-equivariant virtual fundamental class
$$
[M]^{\vir,T}\in A_m^T(M)
$$
which defines
$$
\int_{[M]^{\vir,T}}: A^k_T(M)\to A^{k-m}_T(\pt).
$$
In particular, $\int_{[M]^{\vir,T}}$ sends $A^k_T(M)$ to $0$ if 
$k<m$.

\subsection{Virtual localization} \label{sec:VL}

Let $M^T$ denote the substack of $T$-fixed points in $M$. 
Let $F_1,\ldots, F_N$ be the connected components of
$M^T$. We assume that each $F_j$ is a proper Deligne-Mumford
substack. Given any $\xi\in M^T$, let
$T^1$ and $T^2$ be the tangent and obstruction
spaces at $\xi$. Then $T$ acts on $T^1$ and $T^2$. 
Let $T^{i,f}\subset T^i$ be the maximal subspace
where $T$ acts trivially. Then
$T^i=T^{i,f}\oplus T^{i,m}$. We call
$T^{i,f}$ and $T^{i,m}$ the fixed
and moving parts of $T^i$, respectively. Then
$T^{i,f}$ defines a perfect obstruction theory
on each $F_j$ and a virtual fundamental class
$[F_j]^{\vir,T} \in A_*^T(F_j)$. The virtual
normal bundle $N^\vir_j$ of $F_j$ in $M$ is
$N^\vir_j= T^{1,m}_j - T^{2,m}_j$. 
The $T$-equivariant Euler class
$e_T(N^\vir_j) \in A_T^*(F_j)$ is invertible in
$$
A^*_T (F_j)\otimes_{R_T} Q_T = A^*(F_j)\otimes_{\bQ} Q_T
$$
where $R_T=\bQ[u_1,\ldots,u_l]$, $Q_T=\bQ(u_1,\ldots,u_l)$.
Let $i_j: F_j\to M$ be the inclusion. Assuming the existence
of a $T$-equivariant embedding from $M$ into a {\em smooth} Deligne-Mumford stack,
Graber and Pandharipande proved the following localization formula \cite{GrPa}
(see also K. Behrend \cite{Be2}, A. Kresch \cite{Kr}):

\begin{theorem}[virtual localization]
\begin{equation}
[X]^\vir_T =\sum_{j=1}^N (i_j)_* \frac{[F_j]^{\vir,T} }{e^T(N^\vir_j)}.
\end{equation}
\end{theorem}

\begin{corollary}[intergration formula] If $\alpha^T\in A_T^*(M)$ then 
\begin{equation}\label{eqn:not-proper}
\int_{[M]^{\vir,T} }\alpha^T =\sum_{j=1}^N \int_{[F_j]^{\vir,T} }\frac{i_j^*\alpha^T}{e^T(N^\vir_j)}.
\end{equation}
\end{corollary}

\begin{definition}[equivariant integration on non-proper Deligne-Mumford stack]  
Suppose that $X$ is a {\em non-proper} Deligne-Mumford stack
with a perfect obstruction theory, and
$X^T$ is a finite union of {\em proper} Deligne-Mumford stacks
$F_1,\ldots,F_N$.  We define
$\int_{[X]^{\vir,T} }: A^*_T(M)\to \bQ(u_1,\ldots,u_l)$ by 
$$
\int_{[X]^{\vir,T} } \alpha^T =\sum_{j=1}^N \int_{[F_j]^{\vir,T} }\frac{i_j^*\alpha^T}{e^T(N_j)}.
$$
\end{definition}

\section{Gromov-Witten Theory}
\label{sec:GWreview}

In this section, we give a brief review
of Gromov-Witten theory. We work over $\bC$.

\subsection{Moduli of stable curves and Hodge integrals} 
\label{sec:hodge}

An $n$-pointed, genus $g$ prestable curve
is a connected algebraic curve $C$ of
arithmetic genus $g$ together with
$n$ ordered marked points $x_1,\ldots, x_n\in C$,  
where $C$ has at most nodal singularities,
and $x_1,\ldots,x_n$ are distinct smooth points.
An $n$-pointed, genus $g$ prestable
curve $(C,x_1,\ldots,x_n)$ is {\em stable} if its automorphism
group is finite, or equivalently,
$$
\Hom_{\cO_C}(\Omega_C(x_1+\cdots+x_n), \cO_C) =0.
$$

Let $\Mbar_{g,n}$ be the moduli space of $n$-pointed, genus $g$ stable curves, 
where $n, g$ are nonnegative integers. We assume that
$2g-2+n>0$, so that $\Mbar_{g,n}$ is nonempty.
Then $\Mbar_{g,n}$ is a proper smooth Deligne-Mumford
stack of dimension $3g-3+n$ \cite{DeMu, KnMu, Kn2, Kn3}. The tangent space
of $\Mbar_{g,n}$ at a moduli point
$[(C,x_1,\ldots,x_n)]\in \Mbar_{g,n}$ is given by
$$
\Ext^1_{\cO_C}(\Omega_C(x_1+\cdots+x_n), \cO_C).
$$
Since $\Mbar_{g,n}$ is a proper Deligne-Mumford stack, we may define
$$
\int_{\Mbar_{g,n}}:A^*(\Mbar_{g,n})\to \bQ.
$$

We now introduce some classes in $A^*(\Mbar_{g,n})$.
There is a forgetful morphism $\pi:\Mbar_{g,n+1}\to \Mbar_{g,n}$ 
given by forgetting the $(n+1)$-th marked point (and contracting
the unstable irreducible component if there is one):
$$
[(C,x_1,\ldots,x_n,x_{n+1})]\mapsto
[(C^{st}, x_1,\ldots, x_n)]
$$
where $(C^{st}, x_1,\ldots,x_n)$ is the stabilization
of the prestable curve $(C,x_1,\ldots,x_n)$.
$\pi:\Mbar_{g,n+1}\to \Mbar_{g,n}$ can be identified
with the universal curve over $\Mbar_{g,n}$. 
\begin{itemize}
\item ($\lambda$ classes)
Let $\omega_\pi$ be the relative dualizing sheaf
of $\pi:\Mbar_{g,n+1}\to \Mbar_{g,n}$. 
The Hodge bundle $\bE=\pi_*\omega_\pi$ 
is a rank $g$ vector bundle over $\Mbar_{g,n}$ whose fiber over the moduli point
$[(C,x_1,\ldots,x_n)]\in \Mbar_{g,n}$
is $H^0(C,\omega_C)$, the space of sections of the
dualizing sheaf $\omega_C$ of the curve $C$.  
The $\la$ classes are defined by
$$
\la_j=c_j(\bE)\in A^j(\Mbar_{g,n}).
$$
\item ($\psi$ classes)
The $i$-th marked point $x_i$ gives rise a section
$s_i:\Mbar_{g,n}\to \Mbar_{g,n+1}$ of the universal curve.
Let $\bL_i=s_i^*\omega_\pi$ be the line bundle
over $\Mbar_{g,n}$ whose fiber over the moduli
point $[(C,x_1,\ldots,x_n)]\in \Mbar_{g,n}$ is the
cotangent line $T_{x_i}^*C$ of $C$ at $x_i$.
The $\psi$ classes are defined by
$$
\psi_i=c_1(\bL_i)\in A^1(\Mbar_{g,n}). 
$$
\end{itemize}

{\em Hodge integrals} are top intersection numbers of $\lambda$ classes and
$\psi$ classes:
\begin{equation}\label{eqn:hodge}
\int_{\Mbar_{g,n}}\psi_1^{a_1}\cdots \psi_n^{a_n} 
\la_1^{k_1}  \cdots \la_g^{k_g}  \in \bQ.
\end{equation}
By definition, \eqref{eqn:hodge} is zero unless
$$
a_1+\cdots + a_n + k_1 + 2k_2 + \cdots + g k_g = 3g-3+n.
$$

Using Mumford's Grothendieck-Riemann-Roch calculations in \cite{Mu},  
Faber proved, in \cite{Fa},  that general Hodge integrals can
be uniquely reconstructed from the $\psi$ integrals 
(also known as {\em descendant integrals}):
\begin{equation}\label{eqn:descendant}
\int_{\Mbar_{g,n}}\psi_1^{a_1}\cdots \psi_n^{a_n}. 
\end{equation}
The descendant integrals can be computed recursively by Witten's conjecture
which asserts that the $\psi$ integrals \eqref{eqn:descendant}
satisfy a system of differential equations known as the KdV equations \cite{Wi}.
The KdV equations and the string equation determine all the $\psi$ integrals
\eqref{eqn:descendant} from the initial value $\int_{\Mbar_{0,3}} 1=1$.
For example, from the initial value $\int_{\Mbar_{0,3}}1=1$ and the string equation,
one can derive the following formula of genus $0$ descendant integrals:
 \begin{equation}\label{eqn:genus-zero-psi}
\int_{\Mbar_{0,n}}\psi_1^{a_1}\cdots\psi_n^{a_n}=\frac{(n-3)!}{a_1!\cdots a_n!}
\end{equation}
where $a_1+\cdots+a_n=n-3$ \cite[Section 3.3.2]{Ko2}.

The Witten's conjecture was first proved by Kontsevich in \cite{Ko1}. By now, 
Witten's conjecture has been reproved many times (Okounkov-Pandharipande \cite{OP1},
Mirzakhani \cite{Mi},  Kim-Liu \cite{KiL}, Kazarian-Lando \cite{KaL},  
Chen-Li-Liu \cite{CLL}, Kazarian \cite{Ka}, Mulase-Zhang \cite{MuZ}, etc.).

\subsection{Moduli of stable maps} \label{sec:stable-maps}

Let $X$ be a nonsingular projective or quasi-projective variety over $\bC$,
and let $\beta \in H_2(X;\bZ)$. An $n$-pointed, genus $g$,  
degree $\beta$ prestable map to $X$ is 
a morphism $f:(C,x_1,\ldots,x_n)\to X$, where
$(C,x_1,\ldots,x_n)$ is an $n$-pointed, genus $g$
prestable curve, and $f_*[C]= \beta$. Two prestable maps
$$
f:(C,x_1,\ldots,x_n)\to X,\quad f':(C', x'_1,\ldots, x'_n)\to X
$$
are isomorphic if there exists an isomorphism
$\phi:(C,x_1,\ldots,x_n)\to (C',x'_1,\ldots,x'_n)$
of $n$-pointed prestable curves such that $f=f'\circ \phi$. A prestable map
$f:(C,x_1,\ldots,x_n)\to X$ is {\em stable} if its
automorphism group is finite. The notion of stable maps
was introduced by Kontsevich \cite{Ko2}.

The moduli space $\Mbar_{g,n}(X,\beta)$ 
of $n$-pointed, genus $g$, degree $\beta$ stable maps to $X$
is a Deligne-Mumford stack which is proper
when $X$ is projective \cite{BeMa}. 

\subsection{Obstruction theory and virtual fundamental classes} \label{sec:GWdeform}

The tangent space $T^1$ and the obstruction space $T^2$ at
a moduli point $[f:(C, x_1,\ldots, x_n)\to X] \in \Mbar_{g,n}(X,\beta)$ fit
in the {\em tangent-obstruction exact sequence}:
\begin{equation}\label{eqn:tangent-obstruction}
\begin{aligned}
0 \to& \Ext^0_{\cO_C}(\Omega_C(x_1+\cdots + x_n) , \cO_C)\to H^0(C,f^*T_X) \to T^1  \\
  \to& \Ext^1_{\cO_C}(\Omega_C(x_1+\cdots + x_n), \cO_C)\to H^1(C,f^*T_X)\to T^2 \to 0
\end{aligned}
\end{equation}
where
\begin{itemize}
\item $\Ext^0_{\cO_C}(\Omega_C(x_1+\cdots+x_n),\cO_C)$ is the space of infinitesimal automorphisms of the domain $(C, x_1,\ldots, x_n)$, 
\item $\Ext^1_{\cO_C}(\Omega_C(x_1+\cdots + x_n), \cO_C)$ is the space of infinitesimal 
deformations of the domain $(C, x_1, \ldots, x_n)$, 
\item $H^0(C,f^*T_X)$ is the space of
infinitesimal deformations of the map $f$, and 
\item $H^1(C,f^*T_X)$ is
the space of obstructions to deforming the map $f$.
\end{itemize}
$T^1$ and $T^2$ form sheaves $\cT^1$ and $\cT^2$ on the
moduli space $\MgX$.

Let $X$ be a nonsingular projective variety. 
We say $X$ is {\em convex} if $H^1(C,f^*TX)=0$ for all genus $0$ stable maps $f$.
Projective spaces $\bP^n$, or more generally, generalized flag varieties
$G/P$, are examples of convex varieties. When $X$ is convex and $g=0$,
the obstruction sheaf $\cT^2=0$, and the moduli space $\Mbar_{0,n}(X,\beta)$ is a {\em smooth} 
Deligne-Mumford stack.  

In general, $\MgX$ is a {\em singular} Deligne-Mumford stack equipped 
with a {\em perfect obstruction theory}: 
there is a two term complex of locally free sheaves $E \to F$  on $\Mbar_{g,n}(X,\beta)$  
such that
$$
0\to \cT^1 \to F^\vee \to E^\vee \to \cT^2 \to 0
$$
is an exact sequence of sheaves.
(See \cite{BeFa} for the complete definition of a perfect 
obstruction theory.)  The {\em virtual dimension} $d^\vir$ of 
$\Mbar_{g,n}(X,\beta)$ is the rank of the virtual tangent bundle 
$T^\vir = F^\vee - E^\vee$.
\begin{equation}\label{eqn:virtual-dim}
d^\vir = \int_\beta c_1(T_X) +(\dim X-3)(1-g) +n
\end{equation}

Suppose that $\MgX$ is {\em proper}. (Recall that if $X$
is projective then $\MgX$ is proper for any $g$, $n$, $\beta$.)
Then there is a {\em virtual fundamental class}
$$
[\MgX]^\vir\in A_{d^\vir}(\MgX).
$$ 
The virtual fundamental class has been constructed 
by Li-Tian \cite{LiTi1},  Behrend-Fantechi \cite{BeFa} 
in algebraic Gromov-Witten theory. The virtual
fundamental class allows us to define 
$$
\int_{[\MgX]^\vir}: A^*(\MgX) \longrightarrow \bQ,\quad
\alpha \mapsto  \deg(\alpha\cap[\MgX]^\vir).
$$

\subsection{Gromov-Witten invariants}\label{sec:GWinvariants}
Let $X$ be a nonsingular projective variety. 
Gromov-Witten invariants are rational numbers defined by 
applying
$$
\int_{[\MgX]^\vir}: A^*(\MgX)\to \bQ
$$
to certain classes in $A^*(\MgX)$.

Let $\ev_i:\Mbar_{g,n}(X,\beta)\to X$ be the evaluation
at the $i$-th marked point: $\ev_i$ sends
$[f:(C,x_1,\ldots,x_n)\to X]\in \Mbar_{g,n}(X,\beta)$ to
$f(x_i)\in X$. Given $\gamma_1,\ldots,\gamma_n \in A^*(X)$, define
\begin{equation}\label{eqn:primaryGW}
\langle \gamma_1,\ldots, \gamma_n\rangle^X_{g,\beta}
= \int_{ [\MgX]^\vir } \ev_1^*\gamma_1 \cup \cdots \cup \ev_n^*\gamma_n 
\in \bQ.
\end{equation}
These are known as the {\em primary} Gromov-Witten invariants of $X$.
More generally, we may also view $[\MgX]^\vir$ as a class in $H_{2d}(\MgX)$. 
Then \eqref{eqn:primaryGW} is defined for ordinary cohomology classes $\gamma_1,\ldots,\gamma_n\in H^*(X)$,
including odd cohomology classes which do not come from $A^*(\MgX)$.

Let $\pi:\Mbar_{g,n+1}(X,\beta)\to \Mbar_{g,n}(X,\beta)$ be the universal
curve. For $i=1,\ldots,n$, let $s_i:\Mbar_{g,n}(X,\beta)\to \Mbar_{g,n+1}(X,\beta)$, 
be the section which corresponds to the $i$-th marked point.
Let $\omega_\pi\to \Mbar_{g,n+1}(X,\beta)$ be the relative dualizing
sheaf of $\pi$, and let $\bL_i=s_i^*\omega_\pi$ be the line bundle
over $\Mbar_{g,n}(X,\beta)$ whose fiber at the
moduli point $[f:(C,x_1,\ldots,x_n)\to X]\in \Mbar_{g,n}(X,\beta)$
is the cotangent line $T^*_{x_i}C$ at the $i$-th marked point
$x_i$. The $\psi$-classes are defined to be
$$
\psi_i:=c_1(\bL_i)\in A^1(\MgX),\quad i=1,\ldots,n.
$$
We use the same notation $\psi_i$ to denote the
corresponding classes in the ordinary cohomology group  $H^2(\MgX)$.

The {\em descendant} Gromov-Witten invariants are defined by
\begin{equation}\label{eqn:descendantGW}
\langle \tau_{a_1}(\gamma_1)\cdots \tau_{a_n}(\gamma_n)\rangle_{g,\beta}^X
:= \int_{ [\MgX]^\vir } \ev_1^*\gamma_1\cup \psi_1^{a_1} \cup \cdots 
\cup \ev_n^*\gamma_n \cup \psi_n^{a_n} 
\in \bQ.
\end{equation}
Suppose that $\gamma_i\in H^{d_i}(X)$. Then \eqref{eqn:descendantGW} is zero unless
\begin{equation}\label{eqn:nonzero-condition}
\sum_{i=1}^n d_i + 2\sum_{i=1}^n a_i =2\Bigl(\int_\beta c_1(TX) + (\dim X-3)(1-g)+n\Bigr).
\end{equation}

\begin{remark} Note that
$$
\psi_i \neq \pi^* \psi_i.
$$
where the $\psi_i$ on the left  hand side is
an element in  $H^2(\Mbar_{g,n+1}(X,\beta))$, whereas the $\psi_i$ on the
right hand side is an element in $H^2(\Mbar_{g,n}(X,\beta))$.
Indeed, let $D_i \subset \Mbar_{g,n+1}(X,\beta)$
be the divisor associated to the section $s_i$. Then
$\psi_i = \pi^*\psi_i + [D_i]$ \cite[Section 2b]{Wi}.
\end{remark}

\section{Toric Varieties} \label{sec:toric-varieties}

In this section, we review geometry and topology of
nonsingular toric varieties. We refer to \cite{Fu2} for the theory 
of toric varieties.  We also introduce the toric graph of 
a nonsingular toric variety that satisfies some mild assumptions.
In Section \ref{sec:toricGW}, we will see that the toric graph contains 
all the information needed for computing Gromov-Witten
invariants and equivariant Gromov-Witten invariants of the toric variety.

\subsection{Basic notation}
Let $X$ be a smooth toric variety of dimension $r$. Then 
$X$ contains the algebraic torus $T=(\bC^*)^r$ as a dense open
subset, and the action of $T$ on itself extends to $X$. Let
$N=\Hom(\bC^*, T) \cong \bZ^r$ be the lattice of 1-parameter
subgroups of $T$, and let $M=\Hom(T,\bC^*)$ be the lattice of irreducible
characters of $T$. Then $M=\Hom(N,\bZ)$ is the dual lattice of $N$.
Let $N_\bR = N\otimes_\bZ \bR$ and $M_\bR = M\otimes_\bZ \bR$, so 
that they are dual real vector spaces of dimension $r$. 

The toric variety $X$ is defined by a fan $\Si\subset N_\bR$.
Let $\Si(k)$ be the set of $k$-dimensional cones in $\Si$.
A $k$-dimensional cone $\si\in \Si(k)$ corresponds
to an $(r-k)$-dimensional orbit closure $V(\si)$ of
the $T$-action on $X$.  We make the following assumption:

\begin{assumption}\label{assume}
\begin{itemize}
\item $\Si(r)$ is nonempty,
so that $X$ contains at least one fixed point.
\item Each $(r-1)$ dimensional cone $\tau\in \Si(r-1)$ is contained 
in at least one top dimensional cone $\si\in \Si(r)$.
\end{itemize}
\end{assumption}

We introduce some notation:
\begin{itemize}
\item Let $\{e_1,\ldots,e_r\}$ be a $\bZ$-basis of $N$, and let
$\{u_1,\ldots, u_r\}$ be the dual $\bZ$-basis of
$M=\Hom(N,\bZ)$: $\langle u_i,e_i\rangle =\delta_{ij}$.
\item Given linearly independent vectors $w_1,\ldots,w_k\in N$, define 
$$
\Cone (\{w_1,\ldots,w_k\})=\{t_1w_1+\cdots t_k w_k\mid t_1,\ldots,t_k\in \bR_{\geq 0}\}.
$$
We define $\Cone(\emptyset)=\{0\}$.
\end{itemize}

\begin{example}[$\bP^r$]\label{Pr-fan}
$N=\oplus_{i=1}^r \bZ  e_i$. Let
$$
v_i=e_i,\quad 1\leq i\leq r,\quad v_0=-e_1-\cdots-e_r.
$$
The projective space $\bP^r$ is a nonsingular projective toric variety 
of dimension $r$, defined by the fan 
$$
\Si=\{ \Cone(S)\mid S\subset \{ v_0,\ldots,v_r\}, |S| \leq r\}.
$$
\end{example}

\begin{example}[$\cO_{\bP^1}(-1)\oplus \cO_{\bP^1}(-1)$]\label{conifold-fan}
$N=\bZ e_1\oplus \bZ e_2\oplus \bZ e_3$. Define
$$
v_1=e_1,\quad v_2 =e_2,\quad v_3=e_3,\quad
v_4=e_1+e_2-e_3.
$$
Given $1\leq i_1<\cdots<i_k\leq 4$, define
$$
\si_{i_1\cdots i_k}= \Cone(\{v_{i_1},\ldots, v_{i_k}\}).
$$
The total space of $\cO_{\bP^1}(-1)\oplus \cO_{\bP^1}(-1)$ is a nonsingular
quasi-projective toric variety of dimension $3$, defined
by the fan
$$
\Si=\{ \{0\}, \si_1,\si_2,\si_3,\si_4, \si_{12},\si_{13},\si_{23},\si_{24},\si_{34},
\si_{123},\si_{234}\}.
$$
\end{example}

\subsection{One-skeleton}
The set of $T$ fixed points in $X$ is given by
$$
\{ p_\si:= V(\si):\si\in \Si(r)\}.
$$
The set of $1$-dimensional $T$ orbit closures in $X$ is given by
$$
\{ \ell_\tau:= V(\tau): \tau\in \Si(r-1)\}.
$$
Under our assumption, each $\ell_\tau$ is either an affine line $\bC$  
or a projective line $\bP^1$. We define
$$
\Si(r-1)_c = \{ \tau\in \Si(r-1)_c: \ell_\tau \cong \bP^1\}.
$$
Note that $\Si(r-1)_c =\Si(r-1)$ if $X$ is proper. We define
the 1-skeleton of $X$ to be the union of 1-dimensional
orbit closures:
\begin{equation}\label{eqn:one-skeleton}
X^1:=\bigcup_{\tau\in \Si(r-1)} \ell_\tau.
\end{equation}

We define the set of flags in $\Si$ to be 
\begin{eqnarray*}
F(\Si) & =& \{ (\tau,\si)\in \Si(r-1)\times \Si(r)\mid \tau\subset \si\}\\
&=& \{(\tau,\si)\in \Si(r-1)\times \Si(r)\mid p_\si \in\ell_\tau \}.
\end{eqnarray*}

\begin{example}[$\bP^r$]\label{Pr-flag}
We use the notation in Example \ref{Pr-fan}. Define
\begin{eqnarray*} 
&& \si_i= \Cone\{v_j\mid j\neq i \}, \quad i=0,\dots,r,\\
&& \tau_{ij}= \si_i\cap \si_j\in \Si(r-1),\quad 0\leq i<j\leq r.
\end{eqnarray*}
Then
\begin{eqnarray*}
&& \Si(r)=\{\si_i\mid i=0,\ldots,r\}\\
&& \Si(r-1)=\Si(r-1)_c=\{ \tau_{ij}\mid 0\leq i<j\leq r\}\\
&& F(\Si)=\{(\tau_{ij},\si_i)\mid 0\leq i<j\leq n\}
\cup\{(\tau_{ij},\si_j)\mid 0\leq i<j\leq r\}.
\end{eqnarray*}
\end{example}

\begin{example}[$\cO_{\bP^1}(-1)\oplus \cO_{\bP^1}(-1)$]\label{conifold-flag}
We use the notation in Example \ref{conifold-fan}.
$$
\Si(3)= \{ \si_{123},\si_{234}\},\quad
\Si(2)= \{ \si_{12},\si_{13},\si_{23},\si_{24},\si_{34}\},\quad
\Si(2)_c = \{ \si_{23}\}
$$
$$
F(\Si)=\{(\si_{12},\si_{123}),(\si_{13},\si_{123}), (\si_{23},\si_{123})
,(\si_{23},\si_{234}), (\si_{24},\si_{234}), (\si_{34},\si_{234})\}
$$
\end{example}

\subsection{Toric graph}\label{sec:toric-graph}
The sets $\Si(r)$, $\Si(r-1)$ and $F(\Si)$ define
a connected graph $\Up$.
Each top dimensional cone $\si\in \Si(r)$ corresponds to
a vertex $\bv(\si)$ in $\Up$. Each $(r-1)$ dimensional
cone $\tau\in \Si(r-1)$ corresponds to an edge 
$\be(\tau)$ in $\Up$; $\be(\tau)$ is a ray if $\ell_\tau\cong\bC$,
and is a line segment if $\ell_\tau\cong \bP^1$. 
The vertex $\bv(\si)$ is contained in
the edge $\be(\tau)$  if and only if the fixed point $p_\si$ is contained
in the (affine or projective) line $\ell_\tau$.

Given any top dimensional cone $\si\in \Si(r)$, define
the following subset of $\Si(r-1)$: 
$$
E_\si=\{ \tau\in \Si(r-1)\mid \tau\subset \si\}
=\{\tau\in \Si(r-1)\mid p_\si\in \ell_\tau\}. 
$$
Then $|E_\si| =n$. Therefore $\Up$ is an $r$-valent graph.

Given a flag $(\tau,\sigma)\in F(\Si)$, let 
$\bw(\tau,\sigma) \in M=\Hom(T,\bC^*)$
be the weight of $T$-action on $T_{p_\si} \ell_\tau$, the tangent
line to $\ell_\tau$ at the fixed point $p_\si$, namely,
$$
\bw(\tau,\sigma) := c_1^T(T_{p_\si}\ell_\tau) \in H^2_T(p_\si;\bZ)\cong M. 
$$
This gives rise to a map $\bw:F(\Si)\to M$ satisfying the following
properties.
\begin{enumerate}
\item Given any $\si\in \Si(r)$, the
set $\{\bw(\tau,\si)\mid \tau\in E_\si \}$ form
a $\bZ$-basis of $M$. These are the weights
of the tangent space $T_{p_\si}X$ to $X$ at the fixed point
$p_\si$.
\item Any $\tau\in \Si(r-1)_c$ is contained in two top dimensional
cones $\si,\si'\in \Si(r)$. 
\begin{enumerate}
\item $\bw(\tau,\sigma)+\bw(\tau,\sigma')=0$. 
\item Let $E_\si= \{ \tau_1,\ldots, \tau_r\}$, where
$\tau_r=\tau$.  For any $\tau_i\in E_\si$ there exists a unique
$\tau_i'\in E_{\si'}$ and $a_i\in \bZ$ such that
$$
\bw(\tau_i',\si') =\bw(\tau_i,\sigma)- a_i \bw(\tau,\si).
$$
In particular, $\tau'_r=\tau_r=\tau$ and $a_r=2$. 
\end{enumerate}

\end{enumerate}
Let $\tau$ be as in (2). The normal
bundle of $\ell_\tau \cong \bP^1$ in $X$ is given by
$$
N_{\ell_\tau/X}\cong L_1\oplus\cdots \oplus  L_{n-1}
$$
where $L_i$ is a degree $a_i$ $T$-equivariant line bundle
over $\ell_\tau$ such that the weights of the
$T$-actions on the fibers $(L_i)_{p_\si}$ and $(L_i)_{p_{\si'} }$
are $\bw(\tau_i,\si)$ and $\bw(\tau_i',\sigma')$, respectively.

\begin{example}[$\bP^r$]
In notation in Example \ref{Pr-flag},
\begin{eqnarray*}
&& \bw(\si_{0j},\si_0) = -\bw(\si_{0j},\si_j)= u_j,\quad j=1,\ldots,r,\\
&& \bw(\si_{ij},\si_i) = -\bw(\si_{ij},\si_j) = u_j-u_i,\quad 1\leq i<j\leq r.
\end{eqnarray*}
\end{example}

\begin{example}[$\cO_{\bP^1}(-1)\oplus \cO_{\bP^1}(-1)$]
In notation in Example \ref{conifold-fan} and Example \ref{conifold-flag},
\begin{eqnarray*}
&& \bw(\si_{23},\si_{123}) = u_1,\quad
\bw(\si_{13},\si_{123})=u_2,\quad
\bw(\si_{12},\si_{123})=u_3,\\
&& \bw(\si_{23}, \si_{234})= -u_1,\quad
\bw(\si_{24},\si_{234}) = u_1+u_3,\quad
\bw(\si_{34},\si_{234}) = u_1+u_2.
\end{eqnarray*}
\end{example}

We now give another interpretations of the weight
$\bw(\tau,\si)$ associated to a flag $(\tau,\sigma)\in F(\Ga)$. 
There is a unique $\rho\in \Si(1)$ such that
$\rho\subset \sigma$ and $\rho\not\subset \tau$.
$D_\rho:=V(\rho)$ is a $T$-invariant divisor which
intersects the $T$-invariant (affine or projective) line 
$\ell_\tau$ transversally at the $T$-fixed point $p_\si$. Then
$\bw(\tau,\si)$ is the weight of the $T$-action
on $\cO(D_\rho)_{p_\si}$, i.e.
$$
\bw(\tau,\si)=c_1^T(\cO(D_\rho)_{p_\si}).
$$

The formal completion $\hat{X}$ of $X$ along $X^1$, together with the $T$-action,
can be reconstructed from the graph $\Up$ and $\bw:F(\Si)\to M$. 
We call $(\Up,\bw)$ the {\em  toric graph} defined by $\Si$. 

\subsection{Induced  torus action}
Suppose that there is a group homomorphism
$\phi:T'\to T$ from another
torus $T'\cong (\bC^*)^s$ to $T\cong (\bC^*)^r$.
Then $T'$ acts on $X$ by 
$$
t'\cdot x = \phi(t')\cdot x,\quad t'\in T,\ x\in X.
$$

The group homomorphism $\phi:T'\to T$ induces group homomorphisms
\begin{eqnarray*}
&&\phi_*:  N'=\Hom(\bC^*, T')\longrightarrow  \Hom(\bC^*,T)\\
&&\phi^*:  M = \Hom(T,\bC^*)\longrightarrow \Hom(T',\bC^*)
\end{eqnarray*}

An important example is the big torus $\tT=(\bC^*)^s$ coming 
from the geometric quotient, where $s = |\Si(1)| \geq r$.
Let $I_\Si$ be the ideal of $\bC[z_1,\ldots, z_s]$ generated by
$$
\{ \prod_{\rho_i\not \subset \sigma} z_i :\sigma\in \Si\},
$$
and let $Z(I_\Si)$ be the closed subscheme of $\bC^s$ defined by $I_\Si$. Then
$$
X =(\bC^s-Z(I_\Si))/(\bC^*)^{s-r}.
$$
Let $\Si(1)=\{\rho_1,\ldots,\rho_s\}$. For
each $\rho_\alpha$ there exists a unique primitive vector $v_\alpha\in N$ such that
$\rho_\alpha \cap N = \bZ_{\geq 0}v_\alpha $. 
The group homomorphism
$$
\phi_*: \tN=\bigoplus_{\alpha=1}^s \bZ \te_\alpha \longrightarrow
N=\bigoplus_{i=1}^r\bZ e_i,\quad
\te_\alpha\mapsto v_\alpha 
$$
induces group homomorphisms
$$
\phi: \tT=(\bC^*)^s\longrightarrow T=(\bC^*)^r,\quad
(\tilde{t}_1,\ldots,\tilde{t}_s)
\mapsto (\prod_{\alpha=1}^s \tilde{t}_\alpha^{\langle u_1,v_\alpha\rangle},\ldots,
\prod_{\alpha=1}^s \tilde{t}_\alpha^{\langle u_r, v_\alpha\rangle})
$$
and
$$
\phi^*: M=\bigoplus_{i=1}^r\bZ u_i\longrightarrow \tM=\bigoplus_{\alpha=1}^s \bZ \tu_\alpha,\quad
u_i\mapsto \sum_{\alpha=1}^s\langle u_i,v_\alpha \rangle \tu_\alpha. 
$$

\begin{example}[$\bP^r$] $\bP^r=(\bC^{r+1}-\{0\})/\bC^*$.
The group homomorphism
$$
\phi_*: \tN=\bigoplus_{i=0}^r \bZ \te_i\longrightarrow N=\bigoplus_{i=1}^r \bZ e_i,\quad \te_i\mapsto v_i,
$$
induces group homomorphisms
$$
\phi:\tT= (\bC^*)^{r+1}\to T= (\bC^*)^r,\quad
(\tilde{t}_0,\ldots,\tilde{t}_r)\mapsto 
(\tilde{t}_1 \tilde{t}_0^{-1},\quad, \tilde{t}_r \tilde{t}_0^{-1}). 
$$
and
$$
\phi^*: M=\bigoplus_{i=1}^r \bZ u_i \longrightarrow 
\tM=\bigoplus_{i=0}^r \bZ \tu_i,\quad u_i\mapsto \tu_i-\tu_0,
\quad i=1,\ldots,r.
$$
We have
$$
\phi^*\circ \bw(\tau_{ij},\si_i)=-\bw(\tau_{ij},\si_j)= \tu_j-\tu_i,\quad
0\leq i<j\leq r.
$$
\end{example}

\subsection{Cohomology and equivariant cohomology}
Let $X$ be a smooth toric variety of dimension $r$ defined by a fan $\Si$.
Let $\Si(1)=\{\rho_1,\ldots,\rho_s\}$, and let
$v_\alpha \in N$ be the unique primitive vector such
that $\rho_\alpha\cap N=\bZ_{\geq 0}v_\alpha$. Let
$D_\alpha=V(\rho_\alpha)$.

Given  $\sigma\in \Si(k)$, the scheme theoretic intersection
of toric subvarieties $D_\alpha$ and $V(\si)$ is given by
\begin{equation}\label{eqn:intersect}
D_\alpha\cap V(\si)=\begin{cases}
V(\gamma) & \textup{if $\si$ and $v_\alpha$ span the cone $\gamma \in \Si(k+1)$},\\
\emptyset & \textup{if $\si$ and $v_\alpha$ do not span a cone in $\Si$}.
\end{cases}
\end{equation}

Now assume that $X$ is projective, so that $V(\si)$ is projective for
all $\si\in \Si$. Given a $k$-dimensional cone $\si\in \Si(k)$,
let $[V(\si)]\in H^{2k}(X)$ be the Poincar\'{e} dual of the homology class represented
by $V(\si)$, and let $[V(\si)]^T\in H^{2k}_T(X)$ be the equivariant Poincar\'{e} dual of the 
$T$-equivariant homology class represented by the $T$-invariant subvariety 
$V(\si)$. Then $A^k(X)=H^{2k}(X)$ is generated, as a $\bQ$-vector space, by  
$\{[V(\si)]\mid \si\in \Si(k)\}$, and  $A^k_T(X)=H^{2k}_T(X)$ is generated by
$\{[V(\si)]^T\mid \si\in \Si(k)\}$. We have
\begin{enumerate}
\item[(i)] If $v_{i_1},\ldots, v_{i_k}$ do not span a cone of $\Si$ then
\begin{eqnarray*}
&& [D_{i_1}]\cup \cdots \cup [D_{i_k}]=0\in H^{2k}(X),\\
&& [D_{i_1}]^T\cup \cdots \cup [D_{i_k}]^T=0\in H^{2k}_T(X).
\end{eqnarray*}
\item[(ii)]  For any $u\in M\subset H^2_T(X)$,  
$$
\sum_{\alpha=1}^s \langle u,v_\alpha\rangle [D_\alpha]=0\in H^2(X),
\quad \sum_{\alpha=1}^s\langle u,v_\alpha\rangle [D_\alpha]^T = u\in H^2_T(X).
$$
\end{enumerate}
The above (i) follows from \eqref{eqn:intersect}. To see (ii), 
let $\chi^u:T\to \bC^*$ be the character which corresponds
to $u\in M$. Then $\chi^u$  is a rational function 
on $X$ which defines a $T$-invariant principal divisor 
$\sum_{\alpha=1}^s\langle u,v_\alpha\rangle [D_\alpha]^T$. 
Relations (i) and (ii) are essentially all the relations in 
$H^*(X)$ or $H^*_T(X)$.
\begin{definition}
\begin{enumerate}
\item Let $I$ be the ideal in $\bQ[X_1,\ldots,X_s]$ generated by 
the monomials 
$\{ X_{i_1}\cdots X_{i_k}\mid v_{i_1},\ldots,v_{i_k}\textup{ do not generate a cone in $\Si$}\}$.
\item Let $J$ be the ideal in $\bQ[X_1,\ldots,X_s]$ generated by
$\{ \sum_{\alpha=1}^s \langle u,v_\alpha\rangle X_\alpha \mid u\in M\}$.
\item Let $I'$ be the ideal in $R_T[X_1,\ldots,X_s]= \bQ[X_1,\ldots,X_s, u_1,\ldots,u_r]$
generated by the monomials
$\{ X_{i_1}\cdots X_{i_k}\mid v_{i_1,\ldots,i_k}\textup{ do not generate a cone in $\Si$}\}$.
\item Let $J'$ be the ideal in $R_T[X_1,\ldots,X_s]= \bQ[X_1,\ldots,X_s, u_1,\ldots,u_r]$
generated by $\{ \sum_{\alpha=1}^s \langle u,v_\alpha\rangle X_\alpha-u \mid u\in M\}$.
\end{enumerate}
\end{definition}

With all the above definitions, the cohomology and equivariant
cohomology rings of $X$ can be describe explicitly as follows. (See for example
\cite[Section 5.2]{Fu2}, \cite[Lecture 14]{Fu3}.) 
\begin{theorem}
\begin{eqnarray*}
H^*(X)&\cong& \bQ[X_1,\ldots,X_s]/ (I+J).\\
H^*_T(X)&\cong & \bQ[X_1,\ldots,X_s,u_1,\ldots,u_r]/ (I'+J')\cong
\bQ[X_1,\ldots,X_s]/I.
\end{eqnarray*}
The isomorphism is given by $X_\alpha\mapsto [D_\alpha]$ or $[D_\alpha]^T$.
\end{theorem}
The ring $\bQ[X_1,\ldots,X_s]/I$ is known as the Stanley-Reisner ring.
The ring homomorphism 
$$
i_X^*:H^*_T(X)=\bQ[X_1,\ldots,X_s, u_1,\ldots,u_r]/(I'+J')\to 
H^*(X)=\bQ[X_1,\ldots,X_s]/(I+J)
$$
is surjective. The kernel is the ideal generated by
$u_1,\ldots, u_r$.  We say $\gamma^T\in H_T^*(X)$ is 
a $T$-equivariant lift of $\gamma\in H^*(X)$ if
$i_X^*(\gamma^T)=\gamma$.

\begin{example}[$\bP^r$]
\begin{eqnarray*}
H^*(\bP^r)&\cong & \bQ[X_0,\ldots, X_r]/\langle X_0\cdots X_r, X_1-X_0,\ldots, X_r-X_0\rangle
\cong\bQ[X]/\langle X^{r+1}\rangle.\\
H^*_T(\bP^r)&\cong & \bQ[X_0,\ldots,X_r, u_1,\ldots,u_r]/\langle
X_0\cdots X_r, X_1-X_0-u_1,\ldots, X_r-X_0-u_r\rangle\\
&=& \bQ[X,u_1,\ldots,u_r]/\langle X(X+u_1)\cdots (X+u_r)\rangle.
\end{eqnarray*}
\end{example}

\section{Gromov-Witten Invariants of Smooth Toric Varieties}\label{sec:toricGW}

Let $X$ be a nonsingular toric variety of dimension $r$.
Then $T=(\bC^*)^r$ acts on $X$, and acts on $\MgX$ by
$$
t\cdot [f:(C, x_1,\ldots,x_n)\to X]\mapsto [t\cdot f:(C,x_1,\ldots,x_n)\to X]
$$
where $(t\cdot f)(z)= t\cdot f(z)$, $z\in \bC$. The evaluation maps
$\ev_i:\MgX\to X$ are $T$-equivariant and induce
$\ev_i^*:A^*_T(X)\to A^*_T(\MgX)$.

\subsection{Equivariant Gromov-Witten invariants}
Suppose that $\MgX$ is proper, so that there are virtual fundamental classes
$$
[\MgX]^\vir \in A_{d^\vir}(\MgX),\quad [\MgX]^{\vir,T} \in A_{d^\vir}^T(\MgX),
$$
where
$$
d^\vir= \int_\beta c_1(TX)+(r-3)(1-g)+n.
$$
Given $\gamma_i\in A^{d_i}(X)=H^{2d_i}(X)$ and $a_i\in \bZ_{\geq 0}$, define
$\langle \tau_{a_i}(\gamma_1)\cdots \tau_{a_n}(\gamma_n)\rangle^X_{g,\beta}$ 
as in Section \ref{sec:GWinvariants}:
\begin{equation}\label{eqn:nonequivariantGW}
\langle\tau_{a_1}(\gamma_1)\cdots \tau_{a_n}(\gamma_n)\rangle^X_{g,\beta}
=\int_{[\MgX]^\vir}\prod_{i=1}^n \left(\ev_i^*\gamma_i\cup \psi_i^{a_i}\right) \in \bQ.
\end{equation}
By definition, \eqref{eqn:nonequivariantGW} is zero unless
$\sum_{i=1}^n d_i = d^\vir$. In this case,
\begin{equation}\label{eqn:alphaT}
\langle \tau_{a_1}(\gamma_1)\cdots \tau_{a_N}(\gamma_n)\rangle^X_{g,\beta}
=\int_{[\MgX]^{\vir,T}}\prod_{i=1}^n \left(\ev_i^*\gamma_i^T\cup (\psi_i^T)^{a_i}\right)
\end{equation}
where $\gamma_i^T\in A^{d_i}_T(X)$ is any $T$-equivariant lift of $\gamma_i\in A^{d_i}(X)$,
and $\psi_i^T\in A^1_T(\MgX)$ is any $T$-equivariant lift of $\psi_i\in A^1(\MgX)$.

In this section, we fix a choice of $\psi_i^T$ as follows. 
A stable map $f:(C,x_1,\ldots,x_n)\to X$ induces $\bC$-linear maps
$T_{x_i}C\to T_{f(x_i)}X$ for $i=1,\ldots,n$. This gives rise to
$\bL_i^\vee \to \ev_i^*TX$. The $T$-action on $X$ induces a $T$-action on $TX$, so that
$TX$ is a $T$-equivariant vector bundle over $X$, and $\ev_i^*TX$ is a $T$-equivariant
vector bundle over $\MgX$. Let $T$ act on $\bL_i$ such that
$\bL_i^\vee\to \ev_i^*TX$ is $T$-equivariant, and define
$$
\psi_i^T=c_1^T(\bL_i)\in A^1_T(\MgX), \quad i=1,\ldots,n.
$$
Then $\psi_i^T$ is a $T$-equivariant lift of $\psi_i=c_1(\bL_i)\in A^1(\MgX)$.

Given $\gamma_i^T\in A_T^{d_i}(X)$, 
we define equivariant Gromov-Witten invariants
\begin{equation}\label{eqn:equivariantGW}
\begin{aligned}
\langle\tau_{a_1}(\gamma_1^T),\cdots,\tau_{a_n}(\gamma_n^T)\rangle_{g,\beta}^{X_T}
& :=\int_{[\MgX]^{\vir,T}} \prod_{i=1}^n \left(\ev_i^*\gamma_i^T (\psi_i^T)^{a_i}\right)\\
&\in \bQ[u_1,\ldots,u_l](\sum_{i=1}^n d_i-d^\vir).
\end{aligned}
\end{equation}
where $\bQ[u_1,\ldots,u_l](k)$ is the space of degree $k$ homogeneous polynomials in 
$u_1,\ldots,u_l$ with rational coefficients. In particular, 
$$
\langle\tau_{a_1}(\gamma_1^T),\cdots,\tau_{a_n}(\gamma_n^T)\rangle_{g,\beta}^{X_T}
=\begin{cases}
0, & \sum_{i=1}^n d_i <d^\vir,\\
\langle\tau_{a_1}(\gamma_1),\cdots,\tau_{a_n}(\gamma_n)\rangle_{g,\beta}^X\in \bQ, &
\sum_{i=1}^n d_i = d^\vir.
\end{cases}
$$
where $\gamma_i=i_X^*\gamma_i^T\in A^{d_i}(X)$. (Recall that
$i_X: X\to X_T$ is the inclusion of a fiber of $X_T\to BT$.)

In this section, we will compute the 
equivariant Gromov-Witten invariants \eqref{eqn:equivariantGW}
by localization. Section \ref{sec:graph-notation}
-- Section \ref{sec:each-graph} below are mostly straightforward generalizations 
of the $\bP^r$ case discussed in \cite{Ko2}  (genus $0$),
and \cite[Section 4]{GrPa}, \cite[Section 4]{Be2}
(higher genus). H. Spielberg derived a formula of genus 0 Gromov-Witten 
invariants of smooth toric varieties in his thesis \cite{Sp}. 
See also \cite[Chapter 27]{HKKPTVVZ}.

Let $\MgX^T\subset \MgX$ be the substack of $T$ fixed points, 
and let $i:\MgX^T\to \MgX$ be the inclusion.
Let $N^\vir$ be the virtual normal bundle of substack 
$\MgX^T$ in $\MgX$; in general, $N^\vir$ has different
ranks on different connected components of $\MgX^T$.  
By virtual localization,
\begin{equation}
\int_{[\MgX]^{\vir,T}} \prod_{i=1}^n\left(\ev_i^*\gamma_i^T\cup (\psi_i^T)^{a_i}\right) 
=\int_{[\MgX^T]^{\vir,T}}\frac{i^*\prod_{i=1}^n\left(\ev_i^*\gamma_i^T\cup (\psi_i^T)^{a_i}\right)}{e^T(N^\vir)}.
\end{equation}

Indeed, we will see that $\MgX^T$ is proper even when $\MgX$ is not. When $\MgX$ is not proper, we
{\em define} 
\begin{equation}\label{eqn:residue}
\langle \tau_{a_1}(\gamma_1^T),\ldots,\tau_{a_n}(\gamma_n^T)\rangle^X_{g,\beta}
=\int_{[\MgX^T]^{\vir,T}}\frac{i^*\prod_{i=1}^n\left(\ev_i^*\gamma_i^T\cup (\psi_i^T)^{a_i}\right)}{e^T(N^\vir)}\in  Q_T.
\end{equation}
When $\MgX$ is not proper, the right hand side of \eqref{eqn:residue} is a rational function
(instead of a polynomial) in $u_1,\ldots,u_r$. It can be nonzero when
$\sum d_i< d^\vir$, and does not have a nonequivariant limit (obtained by setting $u_i=0$) 
in general.

\subsection{Torus fixed points and graph notation}\label{sec:graph-notation}
In this subsection, we describe the $T$-fixed points in 
$\MgX$. Following Kontsevich \cite{Ko2}, given 
a stable map $f:(C,x_1,\ldots,x_n)\to X$ such that
$$
[f:(C,x_1,\ldots,x_n)\to X] \in \MgX^T,
$$
we will associate a decorated graph  $\vGa$.

We first give a formal definition.
\begin{definition}\label{df:GgX}
A decorated graph $\vGa=(\Ga, \vf, \vd, \vg, \vs)$
for $n$-pointed, genus $g$, degree $\beta$ 
stable maps to $X$  consists of the following data.

\begin{enumerate}
\item $\Ga$ is a compact, connected 1 dimensional CW complex. 
We denote the set of vertices (resp. edges) in $\Ga$ 
by $V(\Ga)$ (resp. $E(\Ga)$). Let
$$
F(\Gamma)=\{ (e,v)\in E(\Ga)\times V(\Ga)\mid v\in e\}
$$
be the set of flags in $\Gamma$. 
\item The {\em label map} $\vf: V(\Ga)\cup E(\Ga)\to \Si(r)\cup \Si(r-1)_c$
sends a vertex $v\in V(\Ga)$ to 
a top dimensional cone $\si_v \in \Si(r)$, and 
sends an edge $e\in E(\Ga)$ to
an $(r-1)$-dimensional cone $\tau_e \in \Si(r-1)_c$. 
Moreover, $\vf$ defines a map from the graph $\Ga$
to the graph $\Up$: if $(e,v)$ is a flag
in $\Gamma$ then $(\be(\tau_e), \bv(\si_v))$
is a flag in $\Up$, or equivalently,
$(\tau_e,\si_v)\in F(\Si)$.

\item The {\em degree map} $\vd:E(\Ga)\to \bZ_{>0}$
sends an edge $e\in E(\Ga)$ to a positive
integer $d_e$.

\item The {\em genus map} $\vg:V(\Ga)\to \bZ_{\geq 0}$
sends a vertex $v\in V(\Ga)$ to a nonnegative
integer $g_v$.

\item The {\em marking map} $\vs: \{1,2,\ldots,n\}\to V(\Ga)$
is defined if $n>0$.

\end{enumerate} 

The above maps satisfy the following two constraints:
\begin{enumerate}
\item[(i)] (topology of the domain)
$\displaystyle{\sum_{v\in V(\Ga)} g_v +  |E(\Ga)| - |V(\Ga)| +1 = g}$.
\item[(ii)] (topology of the map)
$\displaystyle{ \sum_{e\in E(\Ga)} d_e[\ell_{\tau_e}] =\beta}$.
\end{enumerate}

Let $\GgX$ be the set of all decorated graphs
$\vGa=(\Ga,\vf, \vd,\vg,\vs)$ satisfying the above
constraints.
\end{definition}

We now describe the geometry and combinatorics of
a stable map $f:(C,x_1,\ldots,x_n)\to X$ which represents
a $T$ fixed point in $\MgX$. 

For any $t\in T$, there exists 
an automorphism $\phi_t:(C,x_1,\ldots,x_n)$ such that
$t\cdot f(z)= f\circ\phi_t(z)$ for any $z\in C$. 
Let $C'$ be an irreducible component of $C$, and let
$f'=f|_{C'}:C'\to X$. There are two possibilities:
\begin{enumerate}
\item[Case 1:] $f'$ is a constant map, and 
$f(C')=\{ p_\si\}$, where $p_\si$ is a fixed point
in $X$ associated to some $\si\in \Si(r)$ 
\item[Case 2:] $C'\cong \bP^1$ and $f(C') =\ell_\tau$, where
$\ell_\tau$ is a $T$-invariant $\bP^1$ in $X$
associated to some $\tau\in \Si(r-1)_c$.
\end{enumerate}

We define a decorated graph $\vGa$ associated
to $f:(C,x_1,\ldots,x_n)\to X$ as follows.
\begin{enumerate}
\item (Vertices) We assign a vertex $v$ to each connected component
$C_v$ of $f^{-1}(X^T)$.
\begin{enumerate}
\item (label) $f(C_v)=\{p_\si\}$
for some top dimensional cone $\si\in \Si(r)$;
we define $\vf(v)=\si_v=\si$. 

\item (genus) $C_v$ is a curve or a point. If $C_v$ is a curve
then we define $\vg(v)=g_v$ to be the arithmetic
genus of $C_v$; if $C_v$ is a point then we define
$\vg(v)=g_v=0$. 
\item (marking) For $i=1,\ldots,n$, define
$\vs(i)=v$ if $x_i\in C_v$.
\end{enumerate}

\item (Edges)
For any $\tau\in \Si(r-1)$, let 
$O_\tau \cong \bC^*$ be the 1-dimensional
orbit whose closure is $\ell_\tau$. 
Then
$$
X^1\setminus X^T = \bigsqcup_{\tau\in \Si(r-1)} O_\tau
$$
where the right hand side is a disjoint union of connected components.
We assign an edge $e$ to each connected component
$O_e\cong \bC^*$ of $f^{-1}(X^1\setminus X^T)$. 
\begin{enumerate}
\item (label) Let
$C_e\cong\bP^1$ be the closure of $O_e$. Then
$f(C_e)=\ell_\tau$ for some $\tau$ in $\Si(r-1)_c$; we define
$\vf(e)=\tau_e=\tau$.
\item (degree) We define $\vd(e)=d_e$ to be the degree
of the map $f|_{C_e}: C_e\cong \bP^1\to \ell_\tau\cong \bP^1$. 
\end{enumerate}

\item (Flags) The set of flags in the graph $\Ga$ is defined by
$$
F(\Ga)=\{(e,v)\in E(\Ga)\times V(\Ga)\mid C_e\cap C_v\neq \emptyset\}.
$$
\end{enumerate}
The above (1), (2), (3) define a decorated graph $\vGa=(\Ga, \vf,\vd,\vg,\vs)$ satisfying
the constraints (i) and (ii) in Definition \ref{df:GgX}. 
Therefore $\vGa\in \GgX$.  This gives a map
from $\MgX^T$ to the discrete set $\GgX$. 
Let $\cF_\vGa\subset \MgX^T$ denote the preimage of $\vGa$.
Then
$$
\MgX^T=\bigsqcup_{\vGa\in \GgX}\cF_\vGa
$$
where the right hand side is a disjoint union of connected components.
We next describe the fixed locus $\cF_\vGa$ associated to
each decorated graph $\vGa\in \GgX$. For later convenience,
we introduce some definitions.

\begin{definition}\label{df:unstable}
Given a vertex $v\in V(\Ga)$, we define
$$
E_v=\{e\in  E(\Ga)\mid (e,v)\in F(\Ga)\}, 
$$
the set of edges emanating from $v$, and define $S_v=\vs^{-1}(v)\subset \{1,\ldots,n\}$. 
The valency of $v$ is given by $\val(v)= |E_v|$. Let $n_v=|S_v|$
be the number of marked points contained in $C_v$.
We say a vertex is {\em stable} if
$2g_v-2 + \val(v) + n_v>0$. Let $V^S(\Ga)$ be the set of
stable vertices in $V(\Ga)$. There are three types of 
unstable vertices:
\begin{eqnarray*}
V^1(\Ga)&=&\{ v\in V(\Ga)\mid g_v=0, \val(v)=1, n_v=0\},\\
V^{1,1}(\Ga)&=& \{ v\in V(\Ga)\mid g_v=0, \val(v)=n_v=1\},\\
V^2(\Ga)&=& \{ v\in V(\Ga)\mid g_v=0, \val(v)=2, n_v=0\}.
\end{eqnarray*}
Then $V(\Ga)$ is the disjoint union
of $V^1(\Ga)$, $V^{1,1}(\Ga)$, $V^2(\Ga)$, and $V^S(\Ga)$.

The set of stable flags is defined to be
$$
F^S(\Gamma) = \{(e,v)\in F(\Ga)\mid  v\in V^S(\Ga)\}.
$$
\end{definition}

Given a decorated graph $\vGa=(\Ga,\vf,\vd,\vg,\vs)$, the
curves $C_e$ and the maps $f|_{C_e}:C_e\to \ell_{\tau_e}\subset X$
are determined by $\vGa$.  If $v\notin V^S(\Ga)$ then
$C_v$ is a point. If $v\in V^S(\Ga)$ then $C_v$ is a curve,
and $y(e,v):= C_e\cap C_v$ is a node of $C$ for $e\in E_v$.
$$
\bigl(C_v, \{ y(e,v): e\in E_v\} \cup \{ x_i\mid i\in S_v\} \bigr)
$$ 
is a $(\val(v)+n_v)$-pointed, genus $g_v$ curve, which represents
a point in $\Mbar_{g_v,\val(v)+n_v}$. We call this moduli space
$\Mbar_{g_v,E_v\cup S_v}$ instead of $\Mbar_{g_v, \val(v)+n_v}$
because we would like to label the marked points
on $C_v$ by $E_v\cup S_v$ instead of $\{ 1,2, \ldots, \val(v)+n_v\}$.
Then
$$
\cM_{\vGa}=\prod_{v\in V^S(\Ga)}\Mbar_{g_v, E_v\cup S_v}.
$$
The automorphism group $A_{\vGa}$
for any point $[f:(C,x_1,\ldots,x_n)\to X]\in \cF_{\vGa}$
fits in the following short exact sequence of groups: 
$$
1\to \prod_{e\in E(\Ga)} \bZ_{d_e} \to A_{\vGa}\to \Aut(\vGa)\to 1
$$
where $\bZ_{d_e}$ is the automorphism group
of the degree $d_e$ morphism
$$
f|_{C_e}:C_e\cong \bP^1\to \ell_{\tau_e}\cong \bP^1,
$$
and $\Aut(\vGa)$ is the automorphism group of
the decorated graph $\vGa=(\Ga,\vf,\vd,\vg,\vs)$.
There is a morphism $i_{\vGa}:\cM_{\vGa}\to \MgX$
whose image is the fixed locus $\cF_{\vGa}$ associated
to $\vGa\in \GgX$. The morphism $i_{\vGa}$ induces
an isomorphism $[\cM_{\vGa}/A_{\vGa}]\cong \cF_{\vGa}$.

\subsection{Virtual tangent and normal bundles}
Given a decorated graph $\vGa\in \GgX$ and
a stable map $f:(C,x_1,\ldots,x_n)\to X$
which represents a point in the fixed locus
$\cF_{\vGa}$ associated to $\vGa$, let
\begin{eqnarray*}
&& B_1 =  \Hom(\Omega_C(x_1+\cdots+x_n),\cO_C),\quad B_2 =   H^0(C,f^*TX)\\
&& B_4 = \Ext^1(\Omega_C(x_1+\cdots+ x_n),\cO_C),\quad B_5= H^1(C,f^*TX)
\end{eqnarray*}
$T$ acts on $B_1, B_2, B_4, B_5$. 
Let $B_i^m$ and $B_i^f$ be the moving and fixed parts
of $B_i$.  We have the following exact sequences:
\begin{equation}
0\to B_1^f\to B_2^f\to T^{1,f}\to B_4^f\to B_5^f\to T^{2,f}\to 0
\end{equation}
\begin{equation}
 0\to B_1^m\to B_2^m\to T^{1,m}\to B_4^m\to B_5^m\to T^{2,m}\to 0
\end{equation}

The irreducible components of $C$ are
$$
\{ C_v\mid v\in V^S(\Ga)\}\cup\{C_e\mid e\in  E(\Ga) \}.
$$
The nodes of $C$ are
$$
\{ y_v=C_v \mid v\in V^2(\Ga) \} \cup  \{ y(e,v) \mid (e,v)\in F^S(\Ga)\}
$$

\subsubsection{Automorphisms of the domain} \label{sec:aut}
Given any $(e,v)\in F(\Ga)$, let $y(e,v)=C_e\cap C_v$, and
define
$$
w_{(e,v)}:=e^T(T_{y(e,v)}C_e)=\frac{\bw(\tau_e,\si_v)}{d_e} \in 
H_T^2(y(e,v))= M\otimes_\bZ \bQ.
$$
We have
\begin{eqnarray*}
B_1^f&=&\bigoplus_\edge \Hom(\Omega_{C_e}(y(e,v)+y(e,v')),\cO_{C_e})\\
&=& \bigoplus_\edge H^0(C_e, TC_e(-y(e,v)-y(e,v'))\\
B_1^m&=& \bigoplus_\vone T_{y(e,v)}C_e 
\end{eqnarray*}

\subsubsection{Deformations of the domain} \label{sec:deform}
Given any $v\in V^S(\Ga)$, define
a divisor $\bx_v$ of $C_v$ by
$$
\bx_v=\sum_{i\in S_v} x_i + \sum_{e\in E_v} y(e,v).
$$
Then
\begin{eqnarray*}
B_4^f&=& \bigoplus_{v\in V^S(\Ga)} \Ext^1(\Omega_{C_v}(\bx_v),\cO_C) = \bigoplus_{v\in V^S(\Ga)}
T\Mbar_{g_v, E_v\cup S_v}\\
B_4^m &= & \bigoplus_{v\in V^2(\Ga), E_v=\{e,e'\} }
T_{y_v}C_e\otimes T_{y_v} C_{e'} \oplus \bigoplus_{(e,v)\in F^S(\Ga)} 
T_{y(e,v)}C_v\otimes T_{y(e,v)} C_e
\end{eqnarray*}
where
\begin{eqnarray*}
e^T (T_{y_v}C_e \otimes T_{y_v} C_{e'})&=& w_{(e,v)}+w_{(e',v)},\quad v\in V^2(\Ga)\\
e^T (T_{y(e,v)}C_v \otimes T_{y(e,v)} C_e) &=& w_{(e,v)}-\psi_{(e,v)},\quad v\in V^S(\Ga)
\end{eqnarray*}

\subsubsection{Unifying stable and unstable vertices}
From the discussion in Section \ref{sec:aut} and Section \ref{sec:deform},
\begin{equation} \label{eqn:Bonefour}
\begin{aligned}
\frac{e^T(B_1^m)}{e^T(B_4^m)}=& 
\prod_\vone w_{(e,v)} 
\prod_{v\in V^2(\Ga), E_v=\{e,e'\} }
\frac{1}{w_{(e,v)}+ w_{(e',v)} }\\
&\cdot \prod_{v\in V^S(\Ga)}\frac{1}{\prod_{e\in E_v}(w_{(e,v)}-\psi_{(e,v)}) }.
\end{aligned}
\end{equation}

Recall that
$$
\cM_\vGa =\prod_{v\in V^S(\Ga)} \Mbar_{g_v, E_v\cup S_v}. 
$$

To unify the stable and unstable vertices, we use the following
convention for the empty sets $\Mbar_{0,1}$ and $\Mbar_{0,2}$.
Let $w_1, w_2$ be formal variables.

\begin{enumerate}
\item[(i)]  $\Mbar_{0,1}$ is a $-2$ dimensional space, and
\begin{equation}\label{eqn:one}
\int_{\Mbar_{0,1}}\frac{1}{w_1-\psi_1}=w_1.
\end{equation}
\item[(ii)] $\Mbar_{0,2}$ is a $-1$ dimensional space, and
\begin{equation}\label{eqn:two}
\int_{\Mbar_{0,2}}\frac{1}{(w_1-\psi_1)(w_2-\psi_2)}= \frac{1}{w_1+w_2}
\end{equation}
\begin{equation}\label{eqn:one-one}
\int_{\Mbar_{0,2}}\frac{1}{w_1-\psi_1} =1.
\end{equation}
\item[(iii)] $\displaystyle{ \cM_{\vGa}=\prod_{v\in V(\Ga)} \Mbar_{g_v, E_v\cup S_v} }$.
\end{enumerate}

With the above conventions (i), (ii), (iii), we may rewrite \eqref{eqn:Bonefour} as
\begin{equation}
\frac{e^T(B_1^m)}{e^T(B_4^m)}=
\prod_{v\in V(\Ga)}\frac{1}{\prod_{e\in E_v}(w_{(e,v)}-\psi_{(e,v)}) }.
\end{equation}

The following lemma shows that the conventions (i) and (ii) are consistent with
the stable case $\Mbar_{0,n}$, $n\geq 3$.
\begin{lemma}\label{lemma:psi}
For any positive integer $n$ and formal variables $w_1,\ldots,w_n$, we have
\begin{enumerate}
\item[(a)]$\displaystyle{ \int_{\Mbar_{0,n}}\frac{1}{\prod_{i=1}^n(w_i-\psi_i)}
=\frac{1}{w_1\cdots w_n}(\frac{1}{w_1}+\cdots \frac{1}{w_n})^{n-3} }$.
\item[(b)]$\displaystyle{\int_{\Mbar_{0,n}}\frac{1}{w_1-\psi_1} = w_1^{2-n} }$.
\end{enumerate}
\end{lemma}
\begin{proof}
(a)  The cases $n=1$ and $n=2$ follow from the definitions
\eqref{eqn:one} and \eqref{eqn:two}, respectively.
For $n\geq 3$, we have
\begin{eqnarray*}
&& \int_{\Mbar_{0,n}}\frac{1}{\prod_{i=1}^n (w_i-\psi_i)}
=\frac{1}{w_1\cdots w_n}\int_{\Mbar_{0,n}} \frac{1}{\prod_{i=1}^n(1-\frac{\psi_i}{w_i})}\\
&& =\frac{1}{w_1\cdots w_n}\sum_{a_1+\cdots+ a_n=n-3} w_1^{-a_1}\cdots w_n^{-a_n}
\int_{\Mbar_{0,n}}\psi_1^{a_1}\cdots \psi_n^{a_n}
\end{eqnarray*}
where 
$$
\int_{\Mbar_{0,n}}\psi_1^{a_1}\cdots \psi_n^{a_n} =\frac{(n-3)!}{a_1!\cdots a_n!}.
$$
So 
$$
 \int_{\Mbar_{0,n}}\frac{1}{\prod_{i=1}^n (w_i-\psi_i)}=\frac{1}{w_1\cdots w_n}
(\frac{1}{w_1}+\cdots \frac{1}{w_n})^{n-3}.
$$

(b) The cases $n=1$ and $n=2$ follow from the 
definitions \eqref{eqn:one} and \eqref{eqn:one-one}, respectively. 
For $n\geq 3$, we have
$$
\int_{\Mbar_{0,n}}\frac{1}{w_1-\psi_1}
=\frac{1}{w_1}\int_{\Mbar_{0,n}} \frac{1}{1-\frac{\psi_1}{w_1}}
=\frac{1}{w_1} w_1^{3-n} = w_1^{2-n}
$$
\end{proof}

\subsubsection{Deformation of the map}
Consider the normalization sequence
\begin{equation}\label{eqn:normalize}
\begin{aligned}
0 &\to \cO_C\to \bigoplus_{v\in V^S(\Ga)} \cO_{C_v} \oplus \bigoplus_{e\in E(\Ga)} \cO_{C_e}  \\
 & \to \bigoplus_{v\in V^2(\Ga)} \cO_{y_v} \oplus \bigoplus_{(e,v)\in F^S(\Ga)} \cO_{ y(e,v) }\to 0.
\end{aligned}
\end{equation}
We twist the above short exact sequence of sheaves
by $f^*TX$. The resulting short exact sequence gives
rise a long exact sequence of cohomology groups
\begin{eqnarray*}
0&\to& B_2 \to \bigoplus_{v\in V^S(\Ga)} H^0(C_v)\oplus
\bigoplus_{e\in E(\Ga)}H^0(C_e)  \\
&\to& \bigoplus_{v\in V^2(\Ga)} T_{f(y_v)}X 
\oplus \bigoplus_{(e,v)\in F^S(\Ga)} T_{f(y(e,v))}X \\ 
&\to& B_5 \to \bigoplus_{v\in V^S(\Ga)} H^1(C_v)\oplus
\bigoplus_{e\in E(\Ga)}H^1(C_e) \to 0.
\end{eqnarray*}
where
\begin{eqnarray*}
H^i(C_v) &=& H^i(C_v, (f|_{C_v})^*TX) \cong H^i(C_v, \cO_{C_v})\otimes
T_{p_{\si_v}}X,\\
H^i(C_e) &=& H^i(C_e, (f|_{C_e})^*TX)
\end{eqnarray*}
for $i=0,1$. We have
\begin{eqnarray*}
H^0(C_v)&=& T_{p_{\si_v}}X\\
H^1(C_v)&=& H^0(C_v,\omega_{C_v})^\vee\otimes T_{p_{\si_v}}X.
\end{eqnarray*}

\begin{lemma} Let $\si\in \Si(r)$, so that
$p_\si$ is a $T$ fixed point in $X$. Define
\begin{eqnarray*}
\bw(\si)&=& e^T(T_{p_\si}X) \in H^{2r}_T(\pt)\\
\bh(\si,g)&=& \frac{e^T(\bE^\vee\otimes T_{p_\si}X) }{e^T(T_{p_\si}X)}
\in H^{2r(g-1)}_T(\Mbar_{g,n}). 
\end{eqnarray*}
Then
\begin{equation}\label{eqn:wsi}
\bw(\si)=\prod_{(\tau,\si)\in F(\Si)} \bw(\tau,\si).
\end{equation}
\begin{equation}\label{eqn:hsig}
\bh(\si,g)=\prod_{(\tau,\si)\in F(\Si)} \frac{\Lambda_g^\vee(\bw(\tau,\si))}{\bw(\tau,\si)}
\end{equation}
where $\displaystyle{\Lambda^\vee_g(u)=\sum_{i=0}^g (-1)^i \lambda_i u^{g-i}}$.
\end{lemma}
\begin{proof}  $\displaystyle{T_{p_\si}X =\bigoplus_{(\tau,\si)\in F(\Si)} T_{p_\si} \ell_\tau}$, 
where $e^T(T_{p_\si}\ell_\tau)= \bw(\tau,\si)$. So
\begin{eqnarray*}
e^T(T_{p_\si})&=& \prod_{(\tau,\si)\in F(\Si)} \bw(\tau,\si),\\
\frac{e^T(\bE^\vee\otimes T_{p_\si}\ell_\tau)}{e^T(T_{p_\si}\ell_\tau)}
&=& \prod_{(\tau,\si)\in F(\Si)}\frac{e^T(\bE^\vee\otimes T_{p_\si}\ell_\tau)}{\bw(\tau,\si)},
\end{eqnarray*}
where 
$$
e^T(\bE^\vee\otimes T_{p_\si}\ell_\tau)=
\sum_{i=0}^g  (-1)^i c_i(\bE)c_1^T(T_{p_\si}\ell_\tau)^{g-i}= \sum_{i=0}^g (-1)^i\lambda_i \bw(\tau,\si)^{g-i}.
$$
\end{proof}

The map $B_1\to B_2$  sends
$H^0(C_e, TC_e(-y(e,v)-y(e',v)))$ isomorphically
to $H^0(C_e, (f|_{C_e})^*T\ell_{\tau_e})^f$, the 
fixed part of $H^0(C_e, (f|_{C_e})^*T\ell_{\tau_e})$.

\begin{lemma}
Given $d\in \bZ_{>0}$ and $\tau\in \Si(r-1)_c$, define
$\si, \si',\tau_i, \tau_i', a_i$ as in Section \ref{sec:toric-graph}, and
let $f_d:\bP^1\to \ell_\tau\cong \bP^1$ be the unique
degree $d$ map totally ramified over 
the two $T$ fixed point $p_\si$ and $p_{\si'}$ in $\ell_\tau$. Define
$$
\bh(\tau,d)=\frac{e^T(H^1(\bP^1, f_d^* TX )^m )}
{e^T (H^0(\bP^1, f_d^*TX)^m) }.
$$
Then
\begin{equation}\label{eqn:etaud}
\bh(\tau,d)=\frac{(-1)^d d^{2d}}{ (d!)^2 \bw(\tau,\si)^{2d}}
\prod_{i=1}^{r-1} b(\frac{\bw(\tau,\si)}{d}, \bw(\tau_i,\si), da_i)
\end{equation}
where
\begin{equation}\label{eqn:b}
b(u,w,a)=\begin{cases}
\prod_{j=0}^a(w-ju)^{-1}, & a\in \bZ, a\geq 0,\\
\prod_{j=1}^{-a-1} (w+ju), & a\in \bZ, a<0.
\end{cases}
\end{equation}
\end{lemma}
\begin{proof}
We use the notation in Section \ref{sec:toric-graph}.
We have
$$
N_{\ell_\tau/X}=L_1\oplus \cdots \oplus L_{r-1}.
$$
The weights of $T$-actions on $(L_i)_{p_\si}$
and $(L_i)_{p_\si}$ are $\bw(\tau_i,\si)$
and $\bw(\tau_i,\si)-a_i\bw(\tau,\si)$, respectively.
The weights of $T$-actions on
$T_0\bP^1$, $T_\infty \bP^1$, $(f_d^*L_i)_0$, $(f_d^*L_i)_\infty$
are $u:=\frac{\bw(\tau,\si)}{d}$, $-u$, 
$w_i:= \bw(\tau_i,\si)$, $w_i- da_i u$, respectively.
By Example \ref{Pone},
$$
\ch_T(H^0(\bP^1, f_d^*L_i)-H^1(\bP^1, f_d^*L_i))
=\begin{cases}
\sum_{j=0}^{da_i} e^{w_i-ju}, & a_i\geq 0,\\
\sum_{j=1}^{-da_i-1}e^{w_i+ju}, & a_i<0.
\end{cases}
$$

Note that $w_i+ju$ is nonzero for any $j\in \bZ$ since
$w_i$ and $u$ are linearly independent for
$i=1,\ldots,n-1$. So
$$
\frac{e^T\left(H^1(\bP^1,f_d^*L_i)\right)}{e^T\left(H^0(\bP^1,f_d^*L_i)\right)}
=\frac{e^T\left(H^1(\bP^1,f_d^*L_i)^m\right)}{e^T\left(H^0(\bP^1,f_d^*L_i)^m\right)}
=b(u, w_i, da_i)
$$
where $b(u,w,a)$ is defined by \eqref{eqn:b}.
By Example \ref{Pone},
$$
\ch_T(H^0(\bP^1, f_d^*T\ell_\tau)-H^1(\bP^1, f_d^*T\ell_\tau))
=\sum_{j=0}^{2d} e^{du-ju}= 1+\sum_{j=1}^d (e^{j\bw(\tau,\si)/d}+ e^{-j\bw(\tau,\si)/d}).
$$
So
$$
\frac{e^T(H^1(\bP^1,f_d^*T\ell_\tau)^m)}{e^T(H^0(\bP^1,f_d^*T\ell_\tau)^m)}
=\prod_{j=1}^d \frac{-d^2}{j^2 \bw(\tau,\si)^2 }=\frac{(-1)^d d^{2d}}{(d!)^2 \bw(\tau,\si)^{2d}}.
$$
Therefore,
\begin{eqnarray*}
 \frac{e^T(H^1(\bP^1, f_d^* TX )^m )}
{e^T (H^0(\bP^1, f_d^*TX)^m) }
&=&\frac{e^T(H^1(\bP^1,f_d^*T\ell_\tau)^m)}{e^T(H^0(\bP^1,f_d^*T\ell_\tau)^m)}
\cdot\prod_{i=1}^{r-1}\frac{e^T(H^1(\bP^1,f_d^*L_i)^m)}{e^T(H^0(\bP^1,f_d^*L_i)^m)}\\
&=&\frac{(-1)^d d^{2d}}{(d!)^2 \bw(\tau,\si)^{2d}}\prod_{i=1}^{r-1}b(\frac{\bw(\tau,\si)}{d}, \bw(\tau_i,\si), da_i).
\end{eqnarray*}
\end{proof}
Finally, $f(y_v)=p_{\si_v}= f(y(e,v))$, and
$$
e^T(T_{p_{\si_v}}X) =\bw(\si_v).
$$
From the above discussion, we conclude that
\begin{eqnarray*}
\frac{e^T(B_5^m)}{e^T(B_2^m)} &=&
\prod_{v\in V^2(\Ga)} \bw(\si_v) \cdot \prod_{(e,v)\in F^S(\Ga)}\bw(\si_v)
\cdot \prod_{v\in V^S(\Gamma)}\bh(\si_v,g_v)\cdot\prod_{e\in E(\Ga)}\bh(\tau_e,d_e)\\
&=& \prod_{v\in V(\Ga)}\bigl( \bh(\si_v,g_v)\cdot \bw(\si_v)^{\val(v)} \bigr)
\cdot \prod_{e\in E(\Ga)}\bh(\tau_e,d_e)
\end{eqnarray*}
where $\bw(\si)$, $\bh(\si,g)$, and $\bh(\tau,d)$ are defined by
\eqref{eqn:wsi}, \eqref{eqn:hsig}, \eqref{eqn:etaud}, respectively.

\subsection{Contribution from each graph} \label{sec:each-graph}
\subsubsection{Virtual tangent bundle} We have $B_1^f=B_2^f$, $B_5^f=0$. So
$$
T^{1,f}=B_4^f =\bigoplus_{v\in V^S(\Ga)}T\Mbar_{g_v, E_v\cup S_v},\quad
T^{2,f}=0.
$$
We conclude that
$$
[\prod_{v\in V^S(\Ga)}\Mbar_{g_v, E_v\cup S_v}]^\vir
=\prod_{v\in V^S(\Ga)}[\Mbar_{g_v, E_v\cup S_v} ].
$$
\subsubsection{Virtual normal bundle} 
Let $N^\vir_{\vGa}$ be the pull back of the virtual normal bundle
of $\cF_{\vGa}$ in $\MgX$ under $i_\vGa:\cM_{\vGa}\to \cF_{\vGa}$. Then 
$$
\frac{1}{e^T(N^\vir_\vGa)} = \frac{e^T(B_1^m)e^T(B_5^m)}{e^T(B_2^m)e^T(B_4^m)}
=\prod_{v\in V(\Ga)}\frac{\bh(\si_v,g_v)\cdot \bw(\si_v)^{\val(v)} }{\prod_{e\in E_v}(w_{(e,v)}-\psi_{(e,v)})} 
\cdot \prod_{e\in E(\Ga)}\bh(\tau_e,d_e)
$$

\subsubsection{Integrand} Given $\si\in \Si(r)$, let
$$
i_\si^*: A^*_T(X)\to A^*_T(p_\si) =\bQ[u_1,\ldots,u_r]
$$
be induced by the inclusion $i_\si:p_\si \to X$. Then
\begin{equation}\label{eqn:integrand}
\begin{aligned}
&i_\vGa^*\prod_{i=1}^n\left(\ev_i^*\gamma_i^T \cup (\psi_i^T)^{a_i}\right)\\
=&\prod_{\tiny \begin{array}{c} v\in V^{1,1}(E)\\ S_v=\{i\}, E_v=\{e\}\end{array}} i^*_{\si_v}\gamma_i^T (-w_{(e,v)})^{a_i} 
\cdot \prod_{v\in V^S(\Ga)}\Bigl(\prod_{i\in S_v} i^*_{\sigma_v}\gamma_i^T\prod_{e\in E_v}\psi_{(e,v)}^{a_i}\Bigr)
\end{aligned}
\end{equation}
To unify the stable vertices in $V^S(\Ga)$ and the unstable vertices in $V^{1,1}(\Ga)$ , 
we use the following convention: for $a\in \bZ_{\geq 0}$,
\begin{equation}\label{eqn:one-one-a}
\int_{\Mbar_{0,2}}\frac{\psi_2^a}{w_1-\psi_1}=(-w_1)^a.
\end{equation}
In particular, \eqref{eqn:one-one} is obtained by setting $a=0$. With the 
convention \eqref{eqn:one-one-a}, we may rewrite \eqref{eqn:integrand} as
\begin{equation}
i_\vGa^*\prod_{i=1}^n\left(\ev_i^*\gamma_i^T \cup (\psi_i^T)^{a_i}\right)=
\prod_{v\in V(\Ga)}\Bigl(\prod_{i\in S_v} i^*_{\sigma_v}\gamma_i^T\prod_{e\in E_v}\psi_{(e,v)}^{a_i}\Bigr).
\end{equation}

The following lemma shows that the convention \eqref{eqn:one-one-a} is consistent with
the stable case $\Mbar_{0,n}$, $n\geq 3$.
\begin{lemma}  Let $n,a$ be integers, $n\geq 2$, $a\geq 0$. Then
$$
\int_{\Mbar_{0,n}}\frac{\psi_2^a}{w_1-\psi_1}=\begin{cases}
\displaystyle{ \frac{\prod_{i=0}^{a-1}(n-3-i)}{a!}w_1^{a+2-n},}  & n=2\textup{ or } 0\leq a\leq n-3,\\
0, &\textup{otherwise}.
\end{cases} 
$$
\end{lemma}
\begin{proof} The case $n=2$ follows from \eqref{eqn:one-one-a}.
For $n\geq 3$, 
\begin{eqnarray*}
&& \int_{\Mbar_{0,n}}\frac{\psi_2^a}{w_1-\psi_1}
=\frac{1}{w_1} \int_{\Mbar_{0,n}}\frac{\psi_2^a}{1-\frac{\psi_1}{w_1}}
=w_1^{a+2-n} \int_{\Mbar_{0,n}} \psi_1^{n-3-a}\psi_2^a\\
&& = w_1^{a+2-n}\frac{(n-3)!}{(n-3-a)! a_!}=\frac{\prod_{i=0}^{a-1}(n-3-i)}{a!} w_1^{a+2-n}.
\end{eqnarray*}
\end{proof}

\subsubsection{Integral}\label{sec:vGa-integral}
The contribution of 
$$
\int_{[\MgX^T]^{\vir,T}} \frac{i^*\prod_{i=1}^n (\ev_i^*\gamma_i^T\cup (\psi_i^T)^{a_i})}{e^T(N^\vir)}
$$
from the fixed locus $\cF_\vGa$ is given by
\begin{eqnarray*}
&& \frac{1}{|A_{\vGa}|}\prod_{e\in E(\Ga)}\bh(\tau_e,d_e) \prod_{v\in V(\Ga)}\Bigl(\bw(\si_v)^{\val(v)}\prod_{i\in S_v}i_{\si_v}^*\gamma_i^T\Bigr)\\
&& \quad\quad \cdot \prod_{v\in V(\Ga)}\int_{\Mbar_{g_v,E_v\cup S_v}}
\frac{\bh(\si_v,g_v)\cdot \prod_{e\in E_v}\psi^{a_i}_{(e,v)} }{\prod_{e\in E_v}(w_{(e,v)}-\psi_{(e,v)})} 
\end{eqnarray*}
where $|A_{\vGa}| =|\Aut(\vGa)|\cdot\prod_{e\in E(\Ga)}d_e$.

\subsection{Sum over graphs} 
Summing over the contribution from each graph $\vGa$ given
in Section \ref{sec:vGa-integral} above, we obtain the following
formula.
\begin{theorem}\label{main}
\begin{equation}\label{eqn:sum}
\begin{aligned}
& \langle \tau_{a_1}(\gamma_1^T)\cdots \tau_{a_n}(\gamma_n^T)\rangle^{X_T}_{g,\beta}\\
=& \sum_{\vGa\in G_{g,n}(X,\beta)} \frac{1}{|\Aut(\vGa)|}
\prod_{e\in E(\Ga)} \frac{\bh(\tau_e,d_e)}{d_e}
\prod_{v\in V(\Ga)}\Bigl(\bw(\si_v)^{\val(v)} \prod_{i\in S_v} i_{\si_v}^* \gamma_i^T \Bigr) \\
&\cdot \prod_{v\in V(\Ga)}
\int_{\Mbar_{g,E_v \cup S_v} }
\frac{\bh(\si_v,g_v)\prod_{i\in S_v} \psi_i^{a_i}}{\prod_{e \in E_v} (w_{(e,v)}-\psi_{(e,v)})}.
\end{aligned}
\end{equation}
where $\bh(\tau,d)$, $\bw(\si)$, $\bh(\si,g)$ are given by
\eqref{eqn:etaud}, \eqref{eqn:wsi}, \eqref{eqn:hsig}, respectively, and we have
the following convention for the $v\notin V^S(\Ga)$:
\begin{eqnarray*}
&&\int_{\Mbar_{0,1}}\frac{1}{w_1-\psi_2}= w_1,\quad
\int_{\Mbar_{0,2}}\frac{1}{(w_1-\psi_1)(w_2-\psi_2)}=\frac{1}{w_1+w_2},\\
&& \int_{\Mbar_{0,2}}\frac{\psi_2^a}{w_1-\psi_1}=(-w_1)^a,\quad a\in \bZ_{\geq 0}.
\end{eqnarray*}
\end{theorem}

Given $g\in \bZ_{\geq 0}$, $r$ weights $\vec{w}=\{ w_1,\ldots,w_r\}$,
$r$ partitions $\vec{\mu} =\{ \mu^1,\ldots, \mu^r\}$, and $a_1,\ldots,a_k\in \bZ$,
let $\ell(\mu^i)$ be the length of $\mu^i$, and let $\ell(\vmu)=\sum_{i=1}^r \ell(\mu^i)$.
We define 
$$
\langle \tau_{a_1},\ldots,\tau_{a_k}\rangle_{g,\vmu,\vec{w}}
=\int_{\Mbar_{g,\ell(\vmu)+k}}
\prod_{i=1}^r\Bigl(\frac{\Lambda_g^\vee(w_i) w_i^{\ell(\vmu)-1}
}{\prod_{j=1}^{\ell(\mu^i)}\frac{w_i}{\mu^i_j}-\psi^i_j)}\Bigr)
\prod_{b=1}^k\psi_b^{a_i}.
$$
Given $v\in V(\Ga)$, define
$\vec{w}(v)= \{ \bw(\tau,\si_v)\mid (\tau,\si_v)\in F(\Si)\}$.
Given $v\in V(\Ga)$, and $\tau\in E_{\si_v}$,
let $\mu^{v,\tau}$ be a (possibly empty) partition defined by
$\{ d_e\mid e\in E_v, \vf(e)=\tau\}$,  and define
$\vmu(v)=\{ \mu^{v,\tau}\mid (\tau,\si_v)\in F(\Si)\}$. 
Then \eqref{eqn:sum} can be rewritten as
\begin{equation}
\begin{aligned}
 &\langle \tau_{a_1}(\gamma_1^T)\cdots \tau_{a_n}(\gamma_n^T)\rangle^{X_T}_{g,\beta}\\
= & \sum_{\vGa\in G_{g,n}(X,\beta)} \frac{1}{ |\Aut(\vGa)| }
\prod_{e\in E(\Ga)} \frac{\bh(\tau_e,d_e)}{d_e}
\prod_{v\in V(\Ga)}\Bigl(\prod_{i\in S_v} i_{\si_v}^* \gamma_i  
\langle \prod_{i\in S_v}\tau_{a_i}\rangle_{g_v,\vmu(v),\vec{w}(v)}\Bigr).
\end{aligned}
\end{equation}

Recall that
$$
g= \sum_{v\in V(\Ga)}g_v + |E(\Ga)| - |V(\Ga)| +1
$$
so 
$$
2g-2 =\sum_{v\in V(\Ga)} (2g_v-2 +\val(v)).
$$

Given $\vGa=(\Ga,\vf,\vd,\vg,\vs)$, let $\vGa'=(\Ga,\vf,\vd,\vs)$ be the
decorated graph obtained by forgetting the genus map. Let
$G_n(X,\beta)=\{ \vGa'\mid \vGa\in \cup_{g\geq 0}\GgX \}$. Define
\begin{equation}\label{eqn:gene}
\langle\tau_{a_1}(\gamma_1^T),\cdots, \tau_{a_n}(\gamma_n^T)\mid u\rangle^{X_T}_\beta
= \sum_{g\geq 0} u^{2g-2}\langle\tau_{a_1}(\gamma_1^T),\cdots, \tau_{a_n}(\gamma_n^T)\rangle^{X_T}_{g,\beta}
\end{equation}
\begin{equation}
\langle \tau_{a_1},\ldots,\tau_{a_k}\mid u \rangle_{\vmu,\vec{w}}=
\sum_{g\geq 0}u^{2g-2+\ell(\vmu)}\langle \tau_{a_1},\ldots,\tau_{a_k}\rangle_{g,\vmu,\vec{w}}.
\end{equation}
Then we have the following formula for the generating  function  \eqref{eqn:gene}.
\begin{theorem}\label{generating}
\begin{equation}
\begin{aligned}
\langle \tau_{a_1}(\gamma_1^T)\cdots \tau_{a_n}(\gamma_n^T)\mid u\rangle^{X_T}_\beta 
= &  \sum_{\vGa'\in G_n(X,\beta)} \frac{1}{|\Aut(\vGa)|}
\prod_{e\in E(\Ga)} \frac{\bh(\tau_e,d_e)}{d_e} \\
& \cdot \prod_{v\in V(\Ga)}\Bigl(\prod_{i\in S_v} i_{\si_v}^* \gamma_i^T  
\langle \prod_{i\in S_v}\tau_{a_i}\mid u \rangle_{\vmu(v),\vec{w}(v)}\Bigr).
\end{aligned}
\end{equation}
\end{theorem}

\section{Smooth Deligne-Mumford Stacks} \label{sec:DMstacks}

We work over $\bC$. Let $\cX$ be a smooth Deligne-Mumford (DM) stack. 
Let $\pi:\cX\to X$ be the natural projection to the coarse moduli space $X$.

\subsection{The inertia stack and its rigidification}
The inertia stack $\cIX$ associated to $\cX$ is a smooth DM stack
such that the following diagram is Cartesian:
$$
\begin{CD}
\cIX @>>> \cX \\
@VVV @VV{\Delta}V \\
\cX  @>{\Delta}>> \cX\times \cX 
\end{CD}
$$
where $\Delta:\cX\to \cX\times \cX$ is the diagonal map.
An object in the category $\cIX$ is a pair
$(x,g)$, where $x$ is an object in the category
$\cX$ and $g\in \Aut_{\cX}(x)$:
$$
\Ob(\cIX) =\{ (x,g)\mid x\in \cX, g\in \Aut_{\cX}(x)\}.
$$
The morphisms between two objects in the category $\cIX$ are:
$$
\Hom_{\cIX}( (x_1,g_1), (x_2,g_2)) =\{ h\in \Hom_{\cX} (x_1, x_2) \mid h\circ g_1 = g_2 \circ h \}.
$$
In particular,
$$
\Aut_{\cIX}(x,g) =\{ h\in \Aut_{\cX}(x)\mid h\circ g = g\circ h\}.
$$
The rigidified inertia stack $\bIX$ satisfies
$$
\Ob(\bIX) = \Ob(\cIX),\quad
\Aut_{\bIX}(x,g) = \Aut_{\cIX}(x,g)/\langle g\rangle,
$$
where $\langle g\rangle$ is the subgroup of  $\Aut_{\cIX}(x,g)$ generated by $g$.

There is a more topological interpretation of the inertia stack $\cIX$. 
Let $L\cX =\mathrm{Map}(S^1,\cX)$ be the stack of loops in $\cX$.  
The rotation of $S^1$ induces an $S^1$-action on $L\cX$. 
The stack $(L\cX)^{S^1}$ of $S^1$ fixed loops can be identified with the inertial stack $\cIX$. 
An object in $\cIX$ is a morphism $[\mathrm{pt}/\bZ]\to \cX$ of stacks, which is determined
by $x\in \Ob(\cX)$ and the image of $1\in \bZ$ in $\Aut(x)$.

There is a natural projection $q:\cIX\to \cX$ 
which sends $(x,g)$ to $x$.  There is a natural
involution $\iota:\cIX\to \cIX$ which
sends $(x,g)$ to $(x,g^{-1})$. We assume that $\cX$ is connected. Let
$$
\cIX =\bigsqcup_{i\in I}\cX_i
$$
be disjoint union of connected components.
There is a distinguished connected component
$\cX_0$ whose objects are $(x,\id_x)$, where
$x\in \Ob(\cX)$, and $\id_x\in \Aut(x)$ is the identity
element. The involution $\iota$ restricts
to an isomorphism $\iota_i:\cX_i\to \cX_{\iota(i)}$. 
In particular, $\iota_0:\cX_0\to \cX_0$ is the
identity functor.

\begin{example}[classifying space] \label{ex:BG}
Let $G$ be a finite group. The stack $\cB G= [\pt/G]$
is a category which consists of one object $x$, and $\Hom(x,x)=G$. 
The objects of its inertia stack $\cI \cB G$ are
$$
\Ob(\cI \cB G) =\{ (x,g)\mid g\in G\}.
$$
The morphisms between two objects are
$$
\Hom((x,g_1), (x,g_2)) = \{ g\in G\mid g_2 g = g g_1\} =\{ g\in G\mid g_2= g g_1 g^{-1}\}. 
$$
Therefore
$$
\cI\cB G\cong [G/G]
$$
where $G$ acts on $G$ by conjugation. We have
$$
\cI \cB G=\bigsqcup_{c\in \Conj(G)} (\cB G)_c
$$
where $\Conj(G)$ is the set of conjugacy classes in $G$, and
$(\cB G)_c$ is the connected component associated to the conjugacy 
class $c\in \Conj(G)$.

In particular, when $G$ is abelian, $\Conj(G)=G$, and
$$
\cI \cB G =\bigsqcup_{g\in G} (\cB G)_g
$$
where $(\cB G)_g =[g/G]$.

\end{example}

Given a positive integer $r$, let $\mu_r$ denote the 
group of $r$-th roots of unity. It is a cyclic subgroup
of $\bC^*$ of order $r$, generated by
$$
\zeta_r := e^{2\pi \sqrt{-1}/r}.
$$

\begin{example}\label{ex:WPone}
Let $\bC^*$ acts on $\bC^2-\{0\}$ by
$$
\lambda\cdot (x,y)=(\lambda^2 x, \lambda^3 y),\quad\lambda\in \bC^*,\quad
(x,y)\in \bC^2-\{0\}.
$$
Let $\cX$ be the quotient stack:
$$
\cX=[ (\bC^2-\{0\})/\bC^*] =\bP[2,3].
$$
Then the coarse moduli space is $X=\bP^1$.

We have
$$
\cI \cX= \bigsqcup_{i=0}^3 \cX_i
$$
where 
\begin{eqnarray*}
\cX_0= \cX, && \Ob(\cX_0)= \{ ((x,y), 1)\mid (x,y)\in \bC^2-\{0\} \}, \\
\cX_1= \cB\mu_2,&& \Ob(\cX_1)=\{ ((1,0), -1) \},\\ 
\cX_2=\cB\mu_3, && \Ob(\cX_2)= \{ ((0,1), e^{2\pi\sqrt{-1}/3}) \},\\
\cX_3=\cB\mu_3, && \Ob(\cX_3)= \{ ((0,1), e^{4\pi\sqrt{-1}/3}) \}.
\end{eqnarray*}
We have
$$
\iota_0:\cX_0\to \cX_0,\quad
\iota_1:\cX_1,\to \cX_1,\quad
\iota_2:\cX_2\to \cX_3.
$$
\end{example}

\subsection{Age} \label{sec:age}
Given any object $(x,g)$ in $\cIX$, $g:T_x\cX\to T_x\cX$ is 
a linear isomorphism such that $g^r=\id$, where $r$ is the order
of $g$. The eigenvalues of $g:T_x\cX\to T_x\cX$ are $\zeta_r^{l_1},\ldots, \zeta_r^{l_n}$, where
$l_i\in \{0,1,\ldots, r-1\}$, $n=\dim_\bC \cX$.
Define 
$$
\age(x,g):= \frac{l_1+\cdots+l_n}{r}. 
$$
Then $\age: \cIX\to \bQ$ is constant on each connected
component $\cX_i$ of $\cIX$.  Define $\age(\cX_i)=\age(x,g)$ where
$(x,g)$ is any object in $\cX_i$.  Note that
$$
\age(\cX_i)+ \age(\cX_{\iota(i)})=\dim_\bC \cX-\dim_\bC \cX_i.
$$
\begin{example} \label{ex:WPone-age} 
Let $\cX_0$, $\cX_1$, $\cX_2$, $\cX_3$
be defined as in Example \ref{ex:WPone}. Then 
$$
\age(\cX_0)=0,\quad 
\age(\cX_1)=\frac{1}{2},\quad
\age(\cX_2)=\frac{1}{3},\quad
\age(\cX_3)=\frac{2}{3}.
$$
\end{example}

\subsection{The orbifold cohomology group and operational Chow group}
In \cite{CR1}, W. Chen and Y. Ruan introduced  the orbifold cohomology group 
of a complex orbifold. See \cite[Section 4.4]{AGV1} for a more algebraic version.

The rational Chen-Ruan orbifold cohomology group of $\cX$ is defined to be
$$
H^*_\orb(\cX): = \bigoplus_{a\in \bQ_{\geq 0}}H^a_\orb(\cX)
$$
where
$$
H^a_\orb(\cX)=\bigoplus_{i\in I} H^{a-2\age(\cX_i)}(\cX_i).
$$
The Chen-Ruan orbifold cohomology $H^*_\orb$ is denoted by $H^*_{\mathrm{CR}}$ in some 
papers, for example \cite{Jo}. 

The rational orbifold operational Chow group of $\cX$ is defined
to be
$$
A^*_\orb(\cX): = \bigoplus_{a\in \bQ_{\geq 0}}A^a_\orb(\cX)
$$
where
$$
A^a_\orb(\cX)=\bigoplus_{i\in I} A^{a-\age(\cX_i)}(\cX_i).
$$

Suppose that $\cX$ is proper, and let
$$
\int_\cX: A^*(\cX) \to \bQ
$$
be defined as in Section \ref{sec:intersectionAG}.  Similarly, we have
$$
\int_\cX:H^*(\cX)\to \bQ.
$$
The orbifold Poincar\'{e} pairing is defined by 
$$
(\alpha,\beta)_\orb := 
\begin{cases} 
\int_{\cX_i} \alpha \cup \iota_i^* \beta, & j=\iota(i),\\
0, & j\neq \iota(i),
\end{cases}
$$
where $\alpha\in H^*(\cX_i)$, $\beta\in H^*(\cX_j)$. 
 
\begin{example}
Let $\cX = \bP[2,3]$, and let $\cX_0, \cX_1, \cX_2, \cX_3$ be defined 
as in Example \ref{ex:WPone}. Let $H\in H^2(\cX)=A^1(\cX)$ be the 
pull back of the hyperplane class of $H^2(\bP^1)=A^1(\bP^1)$ under
the map $\cX = \bP[2,3] \to \bP^1$ to the coarse moduli space. 
We have
\begin{eqnarray*}
H^*_\orb(\cX) &=& H^0_\orb(\bP[2,3])\oplus H^\tword_\orb(\cX) \oplus H^1_\orb(\cX) \oplus H^\fourd_\orb(\cX) \oplus H^2_\orb(\cX), \\
A^*_\orb(\cX) &=& A^0_\orb(\cX)\oplus A^\onerd_\orb(\cX) \oplus A^\half_\orb(\cX) \oplus A^\tword_\orb(\cX) \oplus A^1_\orb(\cX), 
\end{eqnarray*} 
where
\begin{eqnarray*}
&& H^0_\orb(\cX) = A^0_\orb(\cX) = H^0(\cX_0) = A^0(\cX_0) = \bQ 1, \\
&& H^\tword_\orb(\cX) = A^\onerd_\orb(\cX) = H^0(\cX_2) = A^0(\cX_2) = \bQ 1_\onerd,\\
&& H^1_\orb(\cX) = A^\half_\orb(\cX) = H^0(\cX_1) = A^0(\cX_1) = \bQ 1_\half, \\
&& H^\fourd_\orb(\cX) = A^\tword_\orb(\cX) = H^0(\cX_3) = A^0(\cX_3) = \bQ 1_\tword,\\
&& H^2_\orb(\cX) = A^1_\orb(\cX) = H^2(\cX_0) = A^1(\cX_0) = \bQ H.
\end{eqnarray*}
\end{example}

\section{Orbifold Gromov-Witten Theory} \label{sec:orbGWreview}

In \cite{CR2}, Chen-Ruan developed Gromov-Witten
theory for symplectic orbifolds. The algebraic counterpart,
the Gromov-Witten theory for smooth DM stacks,
was developed by  Abramovich-Graber-Vistoli 
\cite{AGV1, AGV2}. In this section, we give a brief review
of algebraic orbifold Gromov-Witten theory, following \cite{AGV2}. 

\subsection{Twisted curves and their moduli} \label{sec:tw-curves}

An $n$-pointed, genus $g$  twisted curve is a connected proper one-dimensional
DM stack $\cC$ together with $n$ disjoint closed substacks $\fx_1,\ldots,\fx_n$ of $\cC$, such that
\begin{enumerate}
\item  $\cC$ is \'{e}tale locally a nodal curve; 
\item formally locally near a node, $\cC$ is isomorphic
to the quotient stack
$$
[\Spec(\bC[x,y]/(xy))/\mu_r],
$$ 
where the action of $\zeta\in \mu_r$ is given by $\zeta\cdot (x,y)= (\zeta x, \zeta^{-1}y)$;
\item each $\fx_i\subset \cC$ is contained in the smooth locus of $\cC$;
\item each stack $\fx_i$ is an  \'{e}tale gerbe over $\Spec\bC$ {\em with a section} (hence
trivialization); 
\item $\cC$ is a scheme outside the twisted points $\fx_1,\ldots, \fx_n$ and the singular locus;
\item the coarse moduli space $C$ is a nodal curve of arithmetic genus $g$.
\end{enumerate}
Let $\pi:\cC\to C$ be the projection to the coarse moduli space, and let
$x_i=\pi(\fx_i)$. Then $x_1,\ldots, x_n$ are distinct smooth points of
$C$, and $(C,x_1,\ldots, x_n)$ is an $n$-pointed, genus $g$ prestable curve.

Let $\cM^\tw_{g,n}$ be the moduli of $n$-pointed, genus $g$
twisted curves. Then $\cM^\tw_{g,n}$ is a smooth
algebraic stack, locally of finite type \cite{Ol}.

\subsection{Riemann-Roch theorem for twisted curves} \label{sec:orbRR}
Let $(\cC,\fx_1,\ldots, \fx_n)$ be an $n$-pointed, genus $g$ twisted curve, and
let $(C,x_1,\ldots, x_n)$ be the coarse curve, which is an $n$-pointed, genus
$g$ prestable curve.  Let $\cE \to \cX$ be a vector bundle over $\cX$. 
Then $\fx_i\cong \cB\mu_{r_i}$, and
$\zeta_{r_i}\in \mu_{r_i}$ acts on $\cE|_{\fx_i}$ with eigenvalues
$\zeta_{r_i}^{l_1},\ldots, \zeta_{r_i}^{l_N}$, where $l_i\in \{0,1,\ldots, r_i-1\}$ and
$N=\rank\cE$. Define
$$
\age_{x_i}(\cE) := \frac{l_1+\cdots + l_N}{r_i} \in \bQ. 
$$
The Riemann-Roch theorem for twisted curves says
\begin{equation}\label{eqn:twistedRR}
\chi(\cE)=\int_{\cC}c_1(\cE) + \rank(\cE)(1-g) -\sum_{i=1}^n \age_{x_i}(\cE).
\end{equation}

Given a real number $x$, let $\lfloor x\rfloor$ denote the largest
integer which is less or equal to $x$, and let 
$\langle x \rangle = x-\lfloor x \rfloor$.

\begin{example}
Let $\cC=\bP[2,3]$, $\fx_1=[0,1]$, $\fx_2=[1,0]$. Then
$(\cC, \fx_1,\fx_2)$ is a $2$-pointed, genus $0$ twisted curve.
The coarse moduli curve is $(C,x_1,x_2)=(\bP^1, [0,1], [1,0])$. 
Let $\cL_n = \cO_\cC(n \fx_2)$, where $n\in \bZ$. Then
$$
\int_{\cC}c_1(\cE)=\frac{n}{2},\quad \rank(\cL_n)=1,\quad
\age_{x_1}(\cL_n)=\langle \frac{n}{2}\rangle,\quad \age_{x_2}(\cL_n)=0,
$$
so 
$$
\chi(\cL_n)= 1+ \lfloor \frac{n}{2}\rfloor.
$$

Let $\pi:\cC=\bP[2,3]\to C=\bP^1$ be the projection to the coarse
moduli space. Then $\pi^*\cO_{\bP^1}(kx_2)= \cL_{2k}$. We have
$$
\chi(\cL_{2k})= k+1 = \chi(\cO_{\bP^1}(kx_2))
$$
as expected.

\end{example}

\subsection{Moduli of twisted stable maps}
Let $\cX$ be a proper smooth DM stack with a projective
coarse moduli space $X$, and let $\beta$ be an effective curve class in 
$X$.  An $n$-pointed, genus $g$, degree $\beta$ twisted
stable map to $\cX$ is a representable morphism $f:\cC\to \cX$, where
the domain $\cC$ is an $n$-pointed, genus $g$ twisted curve,
and the induced morphism $C\to X$ between the coarse moduli spaces
is an $n$-pointed, genus $g$, degree $\beta$ stable map to $X$.

Let $\MgcX$ be the moduli stack of $n$-pointed, genus $g$, degree $\beta$ twisted
stable maps to $\cX$. Then $\MgcX$ is a proper DM stack.

For $j=1,\ldots, n$, there are evaluation maps
$\ev_j:\MgcX\to \cIX$. Given $\vi=(i_1,\ldots,i_n)$, where
$i_j\in I$, define
$$
\MgXi :=\bigcap_{j=1}^n \ev_j^{-1}(\cX_{i_j}). 
$$
Then $\MgXi$ is a union of connected components
of $\MgcX$, and  $$
\MgcX=\bigsqcup_{\vi\in I^n}\MgXi.
$$
\begin{remark}
In the definition of twisted curves in Section \ref{sec:tw-curves}, if we replace (4) by
\begin{enumerate}
\item[(4)'] each stack $\fx_i$ is an  \'{e}tale gerbes over $\Spec\bC$;
\end{enumerate}
i.e. without a section, then the resulting moduli space is
$\cK_{g,n}(\cX,\beta)$ in \cite{AGV2}, and the evaluation maps take
values in the rigidified inertial stack $\overline{\cI}{\cX}$ instead of the initial
stack $\cI\cX$.
\end{remark}

\subsection{Obstruction theory and virtual fundamental classes}

The tangent space $T^1$ and the obstruction space $T^2$ at
a moduli point $[f:(\cC, \fx_1,\ldots, \fx_n)\to \cX] \in \MgcX$ fit
in the {\em tangent-obstruction exact sequence}:
\begin{equation}\label{eqn:tangent-obstruction}
\begin{aligned}
0 \to& \Ext^0_{\cO_\cC}(\Omega_\cC(\fx_1+\cdots + \fx_n) , \cO_\cC)\to H^0(\cC,f^*T_\cX) \to T^1  \\
  \to& \Ext^1_{\cO_\cC}(\Omega_\cC(\fx_1+\cdots + \fx_n), \cO_\cC)\to H^1(\cC,f^*T_\cX)\to T^2 \to 0
\end{aligned}
\end{equation}
where
\begin{itemize}
\item $\Ext^0_{\cO_\cC}(\Omega_\cC(\fx_1+\cdots+\fx_n),\cO_\cC)$ is the space of infinitesimal automorphisms of the domain $(\cC, \fx_1,\ldots, \fx_n)$, 
\item $\Ext^1_{\cO_\cC}(\Omega_\cC(\fx_1+\cdots + \fx_n), \cO_\cC)$ is the space of infinitesimal 
deformations of the domain $(\cC, \fx_1, \ldots, \fx_n)$, 
\item $H^0(\cC,f^*T_\cX)$ is the space of
infinitesimal deformations of the map $f$, and 
\item $H^1(\cC,f^*T_\cX)$ is
the space of obstructions to deforming the map $f$.
\end{itemize}
$T^1$ and $T^2$ form sheaves $\cT^1$ and $\cT^2$ on the
moduli space $\MgXi$. This defines a perfect obstruction theory
of virtual dimension
$$
d^\vir_\vi =\int_\beta c_1(T_\cX) + (\dim \cX-3)(1-g) +n -\sum_{j=1}^n\age(\cX_{i_j})
$$
on $\MgXi$, which defines a virtual fundamental class
$$
[\MgXi]^\vir\in A_{d^\vir_\vi}(\MgXi).
$$ 
The {\em weighted virtual fundamental class} is defined by 
$$
[\MgXi]^w := \Bigl(\prod_{j=1}^n r_{i_j}\Bigr) [\MgXi]^\vir.
$$

\subsection{Hurwitz-Hodge integrals}
\label{sec:hurwitz-hodge}

By Example \ref{ex:BG}, when $\cX=\cB G$ we have
$$
\cI \cB G =\bigsqcup_{c\in \Conj(G)} (\cB G)_c
$$
where $\Conj(G)$ is the set of conjugacy classes of $G$.
Give $\vc=(c_1,\ldots,c_n) \in \Conj(G)^n$, let 
$\MBGc = \Mbar_{g,\vc}(\cB G,\beta=0)$.
Then $\MBGc$ is a union of connected components
of $\MBG:=\Mbar_{g,n}(\cB G,0)$, and 
$$
\MBG=\bigsqcup_{\vc\in \Conj(G)^n}\MBGc. 
$$

We now fix a genus $g$ and $n$ conjugacy classes $\vc=(c_1,\ldots,c_n)\in \Conj(G)^n$.
Let $\pi: \cU\to \MBGc$ be the universal curve, and let $f: \cU\to \cB G$ be the universal
map. Let $\rho: G\to GL(V)$ be an irreducible representation of $G$, where $V$ is a finite
dimensional vector space over $\bC$. Then $\cE_\rho := [V/G]$ is
a vector bundle over $\cB G = [\pt/G]$. We have
$$
\pi_* f^* \cE_\rho =\begin{cases}
\cO_{\MBGc}, & \textup{if $\rho: G\to GL(1,\bC)$ is the trivial representation}, \\
0, & \textup{otherwise}.
\end{cases}
$$
The $\rho$-twisted Hurwitz-Hodge bundle $\bE_\rho$ can be defined as the 
dual of the vector bundle $R^1\pi_* f^* \cE_\rho$.   
If $\rho=1$ is the trivial representation, then
$\bE_1 = \ep^*\bE$, where $\ep:\Mbar_{g,\vc}(\cB G) \to \Mbar_{g,n}$, and
$\bE \to \Mbar_{g,n}$ is the Hodge bundle of $\Mbar_{g,n}$. So
$\rank \bE_1 =g$. If $\rho$ is a nontrivial irreducible representation, it follows
from the Riemann-Roch theorem for twisted curves (see Section \ref{sec:orbRR})  that
\begin{equation}
\rank \bE_\rho = \rank (\cE_\rho) (g-1) + \sum_{j=1}^n \age_{c_j}(\cE_\rho),
\end{equation}
where $\age_{c_j}(\cE_\rho)$ is given as follows.  Choose $g\in c_j$. Let $r>0$ be 
the order of $g$ in $G$,  let $N=\rank \cE_\rho =\dim V$.
If the eigenvalues of $\rho(g)\in GL(V)=GL(N,\bC)$ are $\zeta_r^{l_1},\ldots, \zeta_r^{l_N}$,
where $l_1,\ldots, l_N\in\{0,1,\ldots, r-1\}$,  then
$$
\age_{c_j}(\cE_\rho)=\frac{l_1+\cdots + l_N}{r}. 
$$
The definition is independent of choice of $g\in c_j$. The map
$\det\circ \rho: G\to GL(1,\bC)$ descends to a map $\det\circ \rho: \Conj(G)\to GL(1,\bC)$.
We have
$$
\prod_{j=1}^n \det\circ \rho(c_j) =1, 
$$
so
$$
\sum_{j=1}^n \age_{c_j}(\cE_\rho) \in \bZ. 
$$
Note that when $G$ is abelian, any irreducible representation of $G$ is 1-dimensional,
so $\rank(\cE_\rho)=1$ for any irreducible representation $\rho$ of $G$. 

\begin{itemize}
\item {\em Hodge classes.} Given an irreducible representation $\rho$ of $G$, define
$$
\lambda_i^\rho = c_i(\bE_\rho) \in A^i(\MBGc), \quad i=1,\ldots, \rank\bE_\rho. 
$$

\item {\em Descendant classes}.
There is a map $\epsilon: \MBGc\to \Mbar_{g,n}$. Define 
$$
\bar{\psi}_j = \epsilon^*\psi_j \in A^1(\MBGc),\quad j=1,\ldots,n.
$$
\end{itemize}

{\em Hurwitz-Hodge integrals} are top intersection numbers of
Hodge classes $\lambda_i^\rho$  and  descendant classes $\bar{\psi}_j$: 
\begin{equation}
\int_{\MBGc}\bar{\psi}_1^{a_1}\cdots \bar{\psi}_n^{a_n}
(\lambda_1^{\rho_1})^{k_1}\cdots (\lambda_g^{\rho_g})^{k_g}.
\end{equation}
In \cite{Zh}, J. Zhou described an algorithm of computing Hurwitz-Hodge integrals, as
follows. By Tseng's orbifold quantum Riemann-Roch theorem \cite{Ts}, Hurwitz-Hodge integrals 
can be reconstructed from descendant integrals on $\MBGc$:
\begin{equation}
\int_{\MBGc}\bar{\psi}_1^{a_1}\cdots \bar{\psi}_n^{a_n}. 
\end{equation}
Jarvis-Kimura relate the descendant integrals on $\MBGc$ to those on $\Mbar_{g,n}$ \cite{JK}. 
We now state their result. Given $g\in \bZ_{\geq 0}$ and $\vc=(c_1,\ldots,c_n)\in \Conj(G)^n$, let
$$
V^G_{g,\vc}:=\{(a_1,b_1,\ldots, a_g, b_g, e_1,\ldots, e_n)\in G^{2g+n}\mid \prod_{i=1}^g[a_i,b_i]=\prod_{j=1}^n e_j, e_j\in c_j\}.
$$ 
Then $\Mbar_{g,\vc}(\cB G)$ is nonempty iff $V^G_{g,\vc}$ is nonempty.
\begin{theorem}[{Jarvis-Kimura \cite[Proposition 3.4]{JK}}]\label{thm:orb-psi}
Suppose that $2g-2+n>0$ and $V^G_{g,\vc}$ is nonempty. Then
$$
\int_{\Mbar_{g,\vc}(\cB G)}\bar{\psi}_1^{a_1} \cdots \bar{\psi}_n^{a_n} =\frac{|V^G_{g,\vc}|}{|G|}\int_{\Mbar_{g,n}} \psi_1^{a_1}\cdots \psi_n^{a_n}.
$$
\end{theorem}

When $G$ is abelian, each $c_i$ is an element in $G$. 
$V^G_{g,\vc}$ is nonempty iff $c_1 \cdots c_n=1$, and in this case $V^G_{g,\vc} = G^{2g}$.
\begin{corollary}\label{cor:abelian-psi} 
Let $G$ be a finite abelian group.
Suppose that $2g-2+n>0$, and $\vc=(c_1,\ldots, c_n)\in G^n$, where $c_1\cdots c_n=1$. Then
$$
\int_{\Mbar_{g,\vc}(\cB G)}\bar{\psi}_1^{a_1} \cdots \bar{\psi}_n^{a_n} 
= |G|^{2g-1} \int_{\Mbar_{g,n}} \psi_1^{a_1}\cdots \psi_n^{a_n}.
$$
\end{corollary}

\subsection{Orbifold Gromov-Witten invariants} \label{sec:orbGW}
There is a morphism $\epsilon:\MgXi \to \MgX$. Define $\bar{\psi}_i = \epsilon^* \psi_i$.
Let
$$
\gamma_j\in A^{d_j}(\cX_{i_j}) \subset A^{d_j+\age(\cX_{i_j})}_\orb(\cX).
$$
Define orbifold Gromov-Witten invariants
\begin{equation}\label{eqn:orbifoldGWprimary}
\langle \bar{\tau}_{a_1}\gamma_1,\ldots, \bar{\tau}_{a_n}\gamma_n\rangle_{g,\beta}^\cX
:=\int_{[\MgXi]^w} \prod_{j=1}^n \ev_j^* \gamma_j \bar{\psi}_j^{a_j} 
\end{equation}
which is zero unless
$$
\sum_{j=1}^n (d_j+ \age(\cX_{i_j}) + a_j) =\int_\beta c_1(T_\cX)+(1-g)(\dim \cX-3)+n.
$$
More generally, let
$$
\gamma_j\in H^{d_j}(\cX_{i_j}) \subset H^{d_j+2\age(\cX_{i_j})}_\orb(\cX),
$$
and define orbifold Gromov-Witten invariants \eqref{eqn:orbifoldGWprimary}.
Then it is zero unless
$$
\sum_{j=1}^n (d_j+ 2\age(\cX_{i_j})+2a_j) =2\left( \int_\beta c_1(T_\cX)+(1-g)(\dim \cX-3)+n\right).
$$

\section{Toric Deligne-Mumford Stacks} \label{sec:toric-stacks}

In \cite{BCS}, Borisov, Chen, and Smith defined toric DM stacks
in terms of stacky fans. Toric DM stacks are smooth DM
stacks, and their coarse moduli spaces are simplicial toric varieties.
A toric DM stack is called  a {\em toric orbifold} if its generic stabilizer is trivial.
Later, more geometric definitions of toric orbifolds and
toric DM stacks are given by Iwanari \cite{Iw1, Iw2}
and by Fantechi-Mann-Nironi \cite{FMN}, respectively.

\subsection{Stacky fans} \label{sec:stacky}

In this subsection, we recall the definition of stacky fans. 
Let $N$ be a finitely generated abelian group, and let $N_\bR=N\otimes_\bZ\bR$.
We have a short exact sequence of abelian groups:
$$
1\to \Ntor\to N\to \baN=N/\Ntor \to 1,
$$
where $\Ntor$ is the subgroup of torsion elements in 
$N$. Then $\Ntor$ is a finite abelian group, and $\baN\cong \bZ^r$,
where $r=\dim_\bR N_\bR$. 
The natural projection $N\to \baN$ is denoted by $b\mapsto \bab$.

Let $\Si$ be a simplicial fan in $N_\bR$ (see \cite{Fu2}), and let 
$\Si(1)=\{\rho_1,\ldots,\rho_s\}$ be
the set of 1-dimensional cones in the fan $\Si$.
We assume that $\rho_1,\ldots, \rho_s$ span $N_\bR$, and 
fix $b_i\in N$ such that $\rho_i=\bR_{\geq 0}\bab_i$. 
A \emph{stacky fan} $\bSi$ is defined as the data
$(N,\Si,\beta)$, where  
$\beta:\tN:=\oplus_{i=1}^s \bZ \tb_i\cong \bZ^s\to N$
is a group homomorphism defined by  $\tb_i\mapsto b_i$.
By assumption, the cokernel of $\beta$ is finite.

We introduce some notation. 
\begin{enumerate}
\item $M=\Hom(N,\bZ)=\Hom(\baN,\bZ)\cong (\bZ^r)^*$. 
\item $\tM=\Hom(\tN,\bZ)\cong (\bZ^s)^*$. 
\item Let $\Si(d)$ be the set of $d$-dimensional
cones in $\Si$. Given $\si\in \Si(d)$, let 
$N_\si\subset N$ be the subgroup generated by
$\{ b_i\mid \rho_i\subset \si\}$, 
and let $\baN_\si$ be the rank $d$ sublattice of $\baN$
generated by $\{\bab_i\mid \rho_i\subset \si\}$. Let
$M_\si=\Hom(\baN_\si,\bZ)$ be the dual lattice of $\baN_\si$.
\end{enumerate}

Given $\si\in \Si(d)$, the surjective group homomorphism $N_\si\to \baN_\si$ induces
an injective group homomorphism $\Hom(\baN_\si,\bZ)\to \Hom(N_\si,\bZ)$
which is indeed an isomorphism. So $\Hom(N_\si,\bZ)\cong M_\si\cong \bZ^d$.

\subsection{The Gale dual} \label{sec:gale}
The finite abelian group $N_{\mathrm{tor}}$ is of 
the form $\oplus_{j=1}^l\bZ_{a_j}$.
We choose a projective resolution of $N$:
$$
0\to \bZ^l\stackrel{Q}{\to }\bZ^{r+l}\to N\to 0.
$$
Choose a map $B:\tN\to \bZ^{r+l}$ lifting
$\beta:\tN \to N$. Let $\pr_1:\tN\oplus \bZ^l\to \tN$
and $\pr_2:\tN\oplus\bZ^l\to \bZ^l$ be projections
to the first and second factors, respectively. We have the
following commutative diagram:
$$
\xymatrix{
    & &
\tN\oplus \bZ^l \ar[r]^{\pr_1}\ar[d]_{B\oplus Q}\ar[dl]_{\pr_2} & \tN \ar[d]^\beta \ar[dl]^B   & \\
0 \ar[r] &  \bZ^l  \ar[r]^Q &  \bZ^{r+l} \ar[r] & N\ar[r] & 0
} 
$$

Define the dual group $DG(\beta)$ to be the the cokernel of 
$B^*\oplus Q^*:(\bZ^{r+l})^*\longrightarrow \tM\oplus (\bZ^l)^*$.
The Gale dual of the map $\beta:\tN\to N$ is $\beta^\vee:\tM\to DG(\beta)$. 

$$
\xymatrix{
& 0 & \\
& DG(\beta)\ar[u] & \\
& \tM\oplus (\bZ^l)^* \ar[u] & \tM \ar[l]_{\pr_1^*}\ar[ul]_{\beta^\vee} \\
\bZ^l \ar[ur]^{\pr_2^*} & (\bZ^{r+l})^* \ar[u]^{ B^*\oplus Q^* }\ar[ur]_{B^*}\ar[l]_{Q^*} & \\
}
$$

\subsection{Construction of the toric DM stack} \label{sec:construct}
We follow \cite[Section 3]{BCS}.
Applying $\Hom(-,\bC^*)$ to $\beta^\vee:\tM\to DG(\beta)$, one obtains
$$
\phi: G_{\bSi}:=\Hom(DG(\beta),\bC^*)\to \tT:=\Hom(\tM,\bC^*).
$$
Let $G=\Ker\phi$. Then $G\cong \prod_{j=1}^l \mu_{a_j}$, where $\mu_{a_j}\subset \bC^*$ is 
the group of $a_j$-th roots of unity, which is isomorphic to $\bZ_{a_j}$. 
Let $\cB G$ denote the quotient stack $[\{1\}/G]$.
The algebraic torus $\tT$ acts on $\bC^s$ by 
$$
(\tit_1,\ldots, \tit_s)\cdot(z_1,\ldots, z_s)=(\tit_1 z_1,\ldots,\tit_s z_s),\quad
(\tit_1,\ldots, \tit_s)\in \tT,\quad (z_1,\ldots, z_s)\in \bC^s. 
$$
Let $G_{\bSi}$ act on $\bC^s$ by $g\cdot z:=\phi(g)\cdot z$, where $g\in G_{\bSi}$, $z\in \bC^s$.
Let $\cO(\bC^s) = \bC[z_1,\ldots, z_s]$ be the coordinate ring of $\bC^s$. Let $I_\Si$ be the
ideal of $\cO(\bC^s)$ generated by
$$
\{ \prod_{\rho_i \not\subset \si} z_i : \si\in \Si\}
$$
and let $Z(I_\Si)$ be the closed subscheme of $\bC^s$ defined by $I_\Si$.
Then $U:=\bC^s-Z(I_\Si)$ is a quasi-affine variety over $\bC$. The toric DM stack
associated to the stacky fan $\bSi$ is defined to be the quotient stack
$$
\cX_{\bSi}:= [U/G_{\bSi}].
$$
It is a smooth DM stack whose generic stabilizer is $G$, 
and its coarse moduli space is the toric variety $X_\Si$ defined by 
the simplicial fan $\Si$. There is an open dense immersion
$$
\iota: \cT=[ \tT/G_{\bSi} ] \hookrightarrow \cX_{\bSi} = [U/G_{\bSi}],
$$
where $\cT\cong (\bC^*)^r\times \cB G$ is a DM torus.
The action of $\cT$ on itself extends to an action 
$a:\cT\times \cX_{\bSi}\to \cX_{\bSi}$.

\begin{example}[weighted projective spaces]
Let $w_1,\ldots, w_{r+1}$ be positive integers. The weighted
projective space $\bP[w_1,\ldots, w_{r+1}]$ is defined
to be the quotient stack
$$
[ {( \bC^{r+1}-\{0\})/\bC^*} ],
$$
where $\bC^*$ acts on $\bC^{r+1}-\{0\}$ by 
$$
\la\cdot (z_1,\ldots,z_{r+1})=(\la^{w_1} z_1,\ldots, \la^{w_{r+1}} z_{r+1}).
$$
$\bP[w_1,\ldots, w_{r+1}]$ is a smooth DM stack. It is an orbifold
if and only if $g.c.d.(w_1,\ldots, w_{r+1} )=1$.  We will show
that it is indeed a toric DM stack defined by some stacky
fan $\bSi=(\Si,N,\beta)$.

Let $e= g.c.d.(w_1,\ldots, w_{r+1})\in \bZ_{>0}$, so 
that $(w_1,\ldots, w_{r+1})= e(w_1',\ldots, w_{r+1}')$, where
$w_1',\ldots, w_{r+1}'$ are positive integers such that
$g.c.d.(w_1',\ldots, w_{r+1}')=1$. Define
$$
\tN= \bigoplus_{i=1}^{r+1}\bZ \tb_i \cong \bZ^{r+1}.
$$
Define $\tb_0 := \sum_{i=1}^{r+1} w'_i \tb_i$, which is a primitive
vector in the lattice $\tN$, and define  
$$
\baN= \tN/ \bZ \tb_0 \cong \bZ^r.
$$
Applying $\Hom(-,\bZ)$ to the surjective  map $\tN\to \baN$, we obtain
an injective map 
$$
i: M=\Hom(\baN,\bZ) \to \tM=\Hom(\tN,\bZ)
$$
where $M$ can be identified with the following rank $r$ sublattice of $\tM$:
$$
M=\{ \tm\in \tM\mid \langle \tm,\tb_0\rangle =0\}.
$$

Let $\bab_i \in \baN$ be image of $\tb_i$. 
Define $N=\baN\oplus \bZ/e\bZ$, and let $b_i=(\bab_i,1)$.
Define $\beta:\tN\to N$ by 
$\beta(\tb_i) =b_i$.  
A projective resolution of $N$ is given by 
$$
0\to \bZ \stackrel{Q}{\to}\baN\oplus \bZ\to N =\baN\oplus \bZ/e\bZ\to 0,
$$
where $Q(1)= (0,e)$. The map $\beta:\tN\to N$ can be lifted to
$B:\tN \to \baN\oplus \bZ$, $\tb_i \mapsto (\bab_i,1)$. 
Let $\{ \tb^*_1,\ldots, \tb^*_{r+1}\}$ be the $\bZ$-basis of $\tM$ dual 
to the $\bZ$-basis $\{ \tb_1,\ldots, \tb_{r+1}\}$ of $\tN$. 
The map $B^*\oplus Q^*: M\oplus \bZ\to \tM\oplus \bZ$
is given by
$$
(m,0)\mapsto (i(m),0),\quad (0,1)\mapsto (\sum_{i=1}^{r+1} \tb_i^*,e).
$$
The map  $\tM\oplus \bZ\to DG(\beta)=\bZ$ is given by 
$$
(\tb_i^*,0)\mapsto w_i\quad (0,1)\mapsto \sum_{j=1}^{r+1} w_j'.
$$
Applying $\Hom(-,\bC^*)$ to $[w_1 \  \cdots \ w_{r+1}] :\tM=\bZ^{r+1}\to DG(\beta)=\bZ$,
we obtain
$$
\phi: G_\bSi =\bC^* \to \tT=\Hom(\tM,\bC^*)=(\bC^*)^{r+1},
\quad \la\mapsto (\la^{w_1}, \ldots, \la^{w_{r+1}}).
$$
Therefore,
$$
\cX_{\bSi} =(\bC^{r+1}-\{0\})/G_\bSi = \bP[w_1,\ldots, w_{r+1}].
$$
\end{example}

\begin{example}[complete $1$-dimensional toric orbifolds]
\label{ex:Css}
Suppose that $\bSi=(\Si,N,\beta)$ is a stacky fan which defines
a $1$-dimensional complete toric orbifold $\cX_\bSi$. The coarse
moduli space $X_\Si$ must be $\bP^1$, the unique  $1$-dimensional 
complete simplicial toric variety. So we have
$$
N=\bZ,\quad \tN=\bZ^2,\quad  v_1=1,\quad  v_2=1 \quad  b_1=s_1,\quad  b_2=-s_2,
$$
where $s_1,s_2$ are positive integers. Let $\bSi_{s_1,s_2}$ denote the stacky fan 
$$
(\Si, N=\bZ, \beta=[ \ s_1\  -s_2\ ]), 
$$ 
and let $G_{s_1,s_2}=G_{\bSi_{s_1,s_2}}$.
There is a commutative diagram
\begin{equation}\label{eqn:quasi}
\xymatrix{
1 \ar[r] & G_{s_1,s_2} \ar[r]^{\phi_{s_1,s_2}} \ar[d]^{\hat{p}_{s_1,s_2}} 
& \tT=(\bC^*)^2 \ar[r]^{\pi_{s_1,s_2}} \ar[d]^{\tp_{s_1,s_2}} 
& T=\bC^* \ar[r] \ar[d]^{p} & 1\\
1 \ar[r] & G_\Si =\bC^* \ar[r]^{\phi} 
& \tT=(\bC^*)^2 \ar[r]^{\pi} & T=\bC^* \ar[r] &1
}
\end{equation}
where the rows are short exact sequences of abelian groups.
The arrows are group homomorphisms given explicitly as follows:
$$
\tp_{s_1,s_2}(\tit_1,\tit_2)= (\tit_1^{s_1}, \tit_2^{s_2}),\quad p(t)=t, 
\quad \pi_{s_1,s_2}(\tit_1,\tit_2)= \tit_1^{s_1} \tit_2^{-s_2},\quad \pi(\tit_1,\tit_2)= \tit_1 \tit_2^{-1}.
$$
\begin{eqnarray*}
& G_{s_1,s_2} = \Ker (\pi_{s_1,s_2}) = \{(\tit_1,\tit_2)\in \tT= (\bC^*)^2\mid  \tit_1^{s_1} \tit_2^{-s_2}=1\}\\
& G_\Si = \Ker (\pi) =\{(\tit_1,\tit_2)\in \tT=(\bC^*)^2\mid \tit_1 \tit_2^{-1}=1\}
\end{eqnarray*}

Following \cite{Jo}, let $\cC_{s_1,s_2}$ be the toric orbifold defined by the stacky fan 
$\bSi_{s_1,s_2}$:
$$
\cC_{s_1,s_2}:= \cX_{\bSi_{s_1,s_2}} =\left[(\bC^2-\{(0,0)\})/G_{s_1,s_2} \right].
$$
Note that Example \ref{ex:WPone} is a special case of this: $\bP[2,3]=\cC_{3,2}$.
More generally, when $s_1$ and $s_2$ are relatively prime, 
$G_{s_1,s_2}\cong \bC^*$ and 
$$
\cC_{s_1,s_2}=[ (\bC^2-\{(0,0)\})/\bC^*] =\bP[s_2,s_1]
$$
where $\bC^*$ acts on $\bC^2$ by $\la\cdot (z_1,z_2)= (\la^{s_2} z_1, \la^{s_1} z_2)$.
In general, $G_{s_1,s_2}\cong \bC^*\times \mu_d$, where $d=g.c.d.(s_1,s_2)$ (see
\cite[Example 7.29]{FMN}).

The coarse moduli space of $\cC_{s_1,s_2}$ is the projective line:
$$
X_\Si=(\bC^2-\{(0,0)\})/\bC^*=\bP^1,
$$ 
where $\bC^*$ acts on $\bC^2$ by  $\la\cdot (z_1,z_2)=(\la z_1, \la z_2)$.

We have
$$
\mathcal{I}\cC_{s_1,s_2}= \coprod_{\tiny\begin{array}{c} v\in \bZ\\ -s_2< v <s_1 \end{array}} \cC_{s_1,s_2,v}
$$
where 
$$
\cC_{s_1,s_2,v} =\begin{cases}
\cB\mu_{s_1}, & 1\leq v\leq s_1-1,\\ 
\cC_{s_1,s_2}, & v=0,\\
\cB\mu_{s_2}, & 1-s_2 \leq v\leq  -1,
\end{cases}
$$
and
$$
\Ob(\cC_{s_1,s_2,v})= \begin{cases} 
\{ ((0,1), \zeta_{s_1}^v) \}, & 1 \leq v\leq  s_1-1, \\ 
\{ ((x,y), 1)\mid (x,y)\in \bC^2-\{0\} \}, & v=0, \\
\{ ((1,0), \zeta_{s_2}^{-v}) \}, &  1-s_2\leq v \leq -1. 
\end{cases}
$$

We have
$$
\iota_0:\cC_{s_1,s_2,0} \to \cC_{s_1,s_2,0},
$$
and
$$
\iota_v:\cC_{s_1,s_2,v} \to \begin{cases} \cC_{s_1,s_2,s_1-v}, & 1\leq v\leq s_1-1,\\
\cC_{s_1,s_2, s_2+v}, &  1-s_2 \leq v\leq  -1.
\end{cases}
$$

\end{example}

\subsection{Rigidification} \label{sec:rigidification}

We define the {\em rigidification} of $\bSi=(N,\Si,\beta)$ to be the
stacky fan  $\bSi^\rig:=(\bar{N}, \Si, \bar{\beta})$, where
$\bar{\beta}$ is the composition of $\beta:\tN\to N$ with the projection $N\to \baN$.  
Note that $M$, $\baN_\si$, and $M_\si$ defined in Section 
\ref{sec:stacky} depend only on $\bSi^\rig$.
The generic stabilizer of the toric DM stack $\cX_{\bSi^\rig}$
is trivial because $\baN\cong \bZ^n$ is torsion free. So
$\cX_{\bSi^\rig}$ is a toric orbifold.
There is a morphism of stacky fans $\bSi\to \bSi^\rig$
which induces a morphism of toric DM stacks
$\pi^\rig: \cX_{\bSi}\to \cX_{\bSi^\rig}$. The toric orbifold $\cX_{\bSi^\rig}$ is 
called the {\em rigidification} of the toric DM stack $\cX_{\bSi}$.
The morphism $\pi^\rig:\cX_\bSi\to \cX_{\bSi^\rig}$ makes $\cX_\bSi$ 
a $G$-gerbe over $\cX_{\bSi^\rig}$.
 
$G_{\bSi^\rig}= G_{\bSi}/G$ is a subgroup 
of $\tT$. Let $T:=\tT/G_{\bSi^\rig}\cong (\bC^*)^r$.
There is an open dense immersion
$$
\iota^\rig: T=[ \tT/G_{\bSi^\rig}]\hookrightarrow \cX_{\bSi^\rig} = [U/G_{\bSi^\rig}].
$$

\subsection{Lifting the fan}
\label{sec:lift}

Let $\bSi=(N,\Si,\beta)$ be a stacky fan, where $N\cong \bZ^r$.
Let $U$ be defined as in Section \ref{sec:construct}.
The open embedding $U\hookrightarrow \bC^s$ is $\tT$-equivariant, and can
be viewed as a morphism between smooth toric varieties. More explicitly, consider
the $s$-dimensional cone
$$
\tsi_0 =\Cone(\{ \tb_1,\ldots,\tb_s\})\subset \tN_\bR =\tN\otimes_\bZ \bR,
$$
and let  $\tSi_0 \subset \tN_\bR$ be the fan which consists of all
the faces of $\tsi_0$. Then $\bC^s$ is the smooth toric variety defined by
the fan $\tSi_0$. We define a subfan $\tSi \subset \tSi_0$ as follows.
Given $\si\in \Si(d)$, such that $\si\cap \{ \bar{b}_1,\ldots, \bar{b}_s\}
=\{ \bar{b}_{i_1},\ldots, \bar{b}_{i_d}\}$, let
$$
\tsi = \Cone(\{ \tb_{i_1},\ldots,\tb_{i_d}\}) \subset \tN_\bR.
$$
Then there is a bijection $\Si\to \tSi$ given by $\si\mapsto \tsi$, and $U$ is the smooth toric
variety defined by $\tSi$.

For any $d$-dimensional cone $\tsi\in \tSi$, let $I=\{ i\mid \rho_i\subset \si\}$, and define
\begin{eqnarray*}
U_\tsi&=& \Spec\bC[\tsi^\vee\cap \tM]= \bC^s-\{\prod_{i\notin I}z_i=0\}\\
&=& \{ (z_1,\ldots,z_s)\in \bC^s\mid z_i \neq 0 \textup{ if }i\notin I\} \cong \bC^d\times (\bC^*)^{s-d},\\
O_{\tsi}&=& \{(z_1,\ldots, z_s)\in \bC^s\mid z_i=0 \textup{ iff }i \in I\} \cong (\bC^*)^{s-d}\\
V(\tsi)&=& \{(z_1,\ldots, z_s)\in \bC^s\mid z_i=0 \textup{ if }i\in I\} \cong \bC^{s-d}\\
\tT_{\tsi}&=& \{ (\tit_1,\ldots, \tit_s)\in \tT\mid \tit_i=1 \textup{ for } i\notin  I\} \cong (\bC^*)^d.
\end{eqnarray*}
Then 
\begin{itemize}
\item $U_\tsi$ is a Zariski open subset of $U$.
\item $O_{\tsi}$ is an orbit of the $\tT$-action on $U$. The
stabilizer of the $\tT$-action on $O_{\tsi}$ is 
$\tT_{\tsi}$,  so $O_{\tsi}=\tT/\tT_{\tsi}$.
\item $V(\tsi)$ is a closed subvariety of $U$.
\end{itemize}
Let $G_\si=\phi^{-1}(\tT_\tsi)$ be the stabilizer of 
$G_\bSi$-action on $O_{\tsi}$. Then $G_\si$ is a finite
abelian group. In particular, when $\si=\{0\}$ is the
zero dimensional cone, $G_{\{0\}}=\Ker\phi =G$
is the generic stabilizer.
Note that if $\si\subset \si'$ then
$\tT_{\tsi}\subset \tT_{\tsi'}$ and $G_\si\subset G_{\si'}$.

We have  $\tT$-equivariant open embeddings
$$
\tT \hookrightarrow X_{\tSi}=U \hookrightarrow X_{\tSi_0}=\bC^s.
$$

We define
$$
\cX_\si:= [U_\tsi/G_\bSi],\quad
\bfV(\si)=[V(\tsi)/G_\bSi],\quad
\bfO_\si=[O_{\tsi}/G_\bSi].
$$
Then 
\begin{itemize}
\item $\cX_\si$ is an open substack of $\cX$.
\item $\bfO_\si$ is an orbit of the $\cT$-action on $\cX$.
\item $\bfV(\si)$ is the closure of $\bfO_\si$.
\item $\bfV(\si)\to \bfV(\si)^\rig$ is a $G_\si$-gerbe.
\end{itemize}

The $\tT$-equivariant line bundles on $U_\tsi=\Spec\bC[\tsi^\vee\cap \tM]$
are in one-to-one correspondence with characters in $\Hom(\tT_\tsi,\bC^*)$.
Moreover, we have canonical isomorphisms
$$
\Hom(\tT_{\tsi},\bC^*) \cong \tM/(\tsi^\perp\cap \tM) \cong M_\si.
$$
Given $\chi \in M_\si$, let $\cO_{U_\tsi}(\chi)$  denote the  $\tT$-equivariant
line bundle on $U_\tsi$ associated to $\chi\in M_\si$, and
let $\cO_{\cX_\si}(\chi)$ denote the corresponding $\cT$-equivariant line bundle on $\cX_\si= [U_{\tsi}/G_{\bSi}]$.
Let $\tchi\in\tM$ be any representative of the coset
$\chi\in \tM/(\tsi^\perp\cap \tM) \cong M_\si$. The $T$-weights of
$\Gamma(\cX_\si,\cO_{\cX_\si}(\chi))$ are in one-to-one correspondence with points
in $(\chi+\si^\vee)\cap M$.

More generally, a $\tT$-equivariant coherent sheaf on $U$ descends
to a $\cT$-equivariant coherent sheaf on $\cX =[U/G_\bSi]$; indeed,
we may regard this as the definition of a $\cT$-equivariant
coherent sheaf on $\cX$. Composing the map 
$T\to \cT=[T/G]$ with
the $\cT$-action $a: \cT\times \cX \to \cX$ on the toric DM stack $\cX$,  we obtain a
$T$-action $\bar{a}: T\times \cX\to \cX$ on $\cX$.
Following Kresch \cite{Kr}, we define the $\cT$-equivariant Chow groups of the stack $\cX$ to be the Chow groups of the Artin stack
$[ \cX/ \cT]$:
$$
A_{\cT}^* (\cX): = A^*([\cX/\cT]),\quad A_{\cT}^*(\cX;\bZ) := A^*([\cX/\cT];\bZ).
$$
The identification of stacks 
$$
[\cX/ \cT] = [U/\tT]
$$ 
implies that we may identify these Chow groups with  the $\tT$-equivariant Chow groups of $U$:
$$
A_{\cT}^*(\cX)= A^*_{\tT}(U), \quad  A_{\cT}^*(\cX;\bZ) = A^*_{\tT}(U;\bZ).
$$
Note that we have an isomorphism of rational Chow groups 
$$
A^*_{\cT}(\cX) = A^*_T(\cX).
$$ 
As the following example shows, this isomorphism does not generally hold for integral Chow groups.
\begin{example}
Let $\cX=\bP[w]$ be the zero dimensional weighted projective
space, where $w$ is an integer and $w>1$. 
Then $\cX=\cT=\cB \mu_{w}$ and $T=\{1\}$. 
$$
A^1_{\cT}(\cX;\bZ) = 0, \quad A^1_T(\cX;\bZ) = \bZ/w\bZ.
$$
\end{example}

Let $\cV$ be a $\cT$-equivariant vector bundle over $\cX$. 
Under the identification $A^*_T(\cX) = A^*_\cT(\cX)$
(or equivalently,  $H^*_T(\cX)=H^*_{\cT}(\cX)$),  the
$T$-equivariant Chern classes of $\cV$ are equal
to the $\cT$-equivariant Chern classes of $\cV$: 
$$
c^T_k(\cV) = c^\cT_k(\cV), \quad 0\leq k\leq \rank\cV.
$$

\begin{example}\label{ex:edge-orb} 
Let $\cX= \cC_{s_1,s_2}$ be defined as in Example \ref{ex:Css}. Let
$\fp_1=[0,1]$ and $\fp_2=[1,0]$ be the two $T$-fixed (stacky) points in $\cC_{s_1,s_2}$. 
Then any $T$-equivariant line bundle on $\cC_{s_1,s_2}$ is of the form
$$
\cL_{c_1,c_2}=\cO_\cX(c_1 \fp_1 + c_2 \fp_2),\quad c_1, c_2\in \bZ.
$$
We will compute
$$
ch^T\bigl(H^0(\cX,\cL_{c_1,c_2}) - H^1(\cX,\cL_{c_1,c_2}) \bigr).
$$
We have
$$
N=\bZ, \quad \Sigma =\{ \{0\}, \quad \rho_1 = [0,\infty),\quad \rho_2=(-\infty,0] \} 
$$
Let
$$
\cX_1 =\cX_{\rho_1},\quad \cX_2 =\cX_{\rho_2}, \quad \cX_{12} = \cX_{ \{0\} } = \cX_1\cap \cX_2 =T =\bC^*. 
$$
The cohomology groups $H^0(\cX,\cL_{c_1,c_2})$ and $H^1(\cX,\cL_{c_1,c_2})$ are the kernel and
cokernel of the following \v{C}ech complex:
$$
0\to \Gamma(\cX_1,\cL_{c_1,c_2})\oplus \Gamma(\cX_2,\cL_{c_1,c_2}) \stackrel{\delta}{\to} \Gamma(\cX_{12},\cL_{c_1,c_2})\to 0, 
$$
where $\delta(s_1, s_2) = s_1\bigr|_{\cX_{12}}- s_2 \bigr|_{\cX_{12}}$.  Let $u\in M$ be the dual of the $\bZ$-basis of $v_1\in N$. Then
\begin{eqnarray*}
\ch^T \bigl(\Gamma(\cX_1,\cL_{c_1,c_2}) \bigr) &=& \sum_{m\in \bZ, s_1 m\geq -c_1  } e^{mu}\\
\ch^T \bigl(\Gamma(\cX_2,\cL_{c_1,c_2}) \bigr) &=& \sum_{m\in \bZ, -s_2 m \geq -c_2}e^{mu}\\
\ch^T \bigl(\Gamma(\cX_{12},\cL_{c_1,c_2}) \bigr) &=& \sum_{m\in \bZ} e^{mu}.
\end{eqnarray*}
Therefore, 
\begin{eqnarray*}
ch^T(H^0(\cX,\cL_{c_1,c_2})) &=&\begin{cases}
\displaystyle{ \sum_{ m\in \bZ, -\frac{c_1}{s_1}\leq m \leq \frac{c_2}{s_2} } e^{mu} } , & \frac{c_1}{s_1}+\frac{c_2}{s_2} \geq 0,\\
\quad \quad 0, & \frac{c_1}{s_1}+ \frac{c_2}{s_2}<0,
\end{cases} \\
ch^T(H^1(\cX,\cL_{c_1,c_2})) &=&\begin{cases}
\quad \quad 0, & \frac{c_1}{s_1}+\frac{c_2}{s_2} \geq 0, \\
\displaystyle{ \sum_{ m\in \bZ,  \frac{c_2}{s_2} <  m < -\frac{c_1}{s_1} } e^{mu} }, & \frac{c_1}{s_1}+\frac{c_2}{s_2} < 0. 
\end{cases} 
\end{eqnarray*}

More generally, suppose that a torus $T'$ (of any dimension)  acts on 
the total space of $\cL= \cL_{c_1,c_2}$, such that
$$
c_1^{T'}(T_{\fp_1}\cC_{s_1,s_2}) =\frac{-w_1}{s_1},\quad
c_1^{T'}(T_{\fp_2}\cC_{s_1,s_2}) = \frac{w_1}{s_2},\quad
c_1^{T'}(\cL_{\fp_1}) = w_2,\quad c_1^{T'}(\cL_{\fp_2})=w_3,
$$
where $w_1, w_2, w_3\in H^2(BT';\bQ)$. Then
$$
w_3 = w_2 + a w_1,
$$
where
$$
a=\frac{c_1}{s_1} + \frac{c_2}{s_2} \in \frac{g.c.d.(s_1,s_2)}{s_1s_2} \bZ. 
$$
Let
$$
\ep =\langle \frac{c_2}{s_2}\rangle \in \{0,\frac{1}{s_2},\ldots, \frac{s_2-1}{s_2} \}. 
$$
Then 
\begin{eqnarray*}
ch^T(H^0(\cX,\cL)) &=&\begin{cases}
\displaystyle{ \sum_{m\in \bZ, -\ep\leq m\leq a-\ep }e^{w_3 -( m+\ep) w_1} =\sum_{m=0}^{\lfloor a-\ep\rfloor}e^{w_3-(m+\ep)w_1} } , & a \geq 0,\\
\quad \quad 0, & a<0,
\end{cases} \\
ch^T(H^1(\cX,\cL)) &=&\begin{cases}
\quad \quad 0, & a \geq 0, \\
\displaystyle{ \sum_{m\in \bZ, \ep < m < \ep-a}e^{w_3+ (m -\ep) w_1} =\sum_{m=1}^{\lceil \ep-a-1\rceil} e^{w_3+(m-\ep)+w_1} }, & a < 0. 
\end{cases} 
\end{eqnarray*}

\end{example}

\subsection{Toric graph}\label{sec:toric-graph-orb}
The coarse moduli space of the toric DM stack $\cX=\cX_\bSi$
defined by a stacky fan $\bSi=(N,\Si,\beta)$ is the simplicial
toric variety $X=X_\Si$ defined by the simplicial fan $\Si\subset N_\bR$.
The definitions of the 1-skeleton $X^1$ and the flags
in $\Si$ in Section \ref{sec:toric-graph} for smooth toric varieties also work
for simplicial toric varieties.  The sets $\Si(r)$, $\Si(r-1)$ and 
$F(\Si)$ define a connected graph $\Up$.  
Let $T=(\bC^*)^r$ be the torus acting on the coarse moduli $X$, and let $\cT$ be 
the DM torus acting on $\cX$. Then $\pi:\cX\to X$ restricts to 
$\cT\to T$, and $\cT=T$ if and only if $\cX$ is a toric orbifold.

Given $\si\in\Si(r)$  let $p_\si=V(\si)$ (resp. $\fp_\si=\bfV(\si)$) be the associated
zero dimensional $T$-orbit (resp. $\cT$-orbit) in $X$ (resp. $\cX$). Then
$\fp_\si= [p_\si/G_\si]=\cB G_\si$.
Given $\tau\in \Si(r-1)$, let $\ell_\tau=V(\tau)$ (resp. $\fl_\tau=\bfV(\tau)$) 
be the associated one dimensional $T$-orbit closure (resp. $\cT$-orbit closure)
in $X$ (resp. $\cX$). Then $\fl_\tau$ is a 1-dimensional toric DM stack, 
and $\fl_\tau\to \fl_\tau^\rig$ is a $G_\tau$-gerbe. Define a map
$r:F(\Si)\to \bZ_{>0}$ by 
$$
r(\tau,\si)=\frac{|G_\si|}{|G_\tau|}. 
$$
There there is a short exact sequence of abelian groups
$$
1\to G_\tau \lra G_\si \stackrel{\phi(\tau,\si)}{\lra} \mu_{r(\tau,\si)}\to 1,
$$
where $\phi(\tau,\si): G_\si\to \bC^*$ is the character of 
the irreducible $G_\si$-representation $T_{\fp_\si}\fl_\tau$.

Given $\tau\in \Si(r-1)$, there are two cases:
\begin{enumerate}
\item Suppose that $\tau\in \Si(r-1)_c$. Then
$\tau$ is the intersection of two $r$-dimensional cones $\si$, $\si'$.
We have $\ell_\tau\cong\bP^1$ and $\fl_\tau^\rig\cong \cC_{r(\tau,\si), r(\tau,\si')}$.
\item Suppose that $\tau\notin \Si(r-1)_c$. Then
there is a unique $r$-dimensional cone $\si$ which contains $\tau$. We have
$\ell_\tau\cong \bC$ and $\fl_\tau^\rig \cong [\bC/\mu_{r(\tau,\si)}]$.
\end{enumerate}

Given $(\tau,\si)\in F(\Ga)$, let $\bw(\tau,\si) \in M_\si$ be characterized by
$$
\langle \bw(\tau,\si), b_i\rangle =
\begin{cases} 
0 & \textup{if } \rho_i\subset \tau, \\
1 & \textup{if }\rho_i\subset \si \textup{ and } \rho_i\not \subset \tau.
\end{cases}
$$
This gives rise to a map $\bw:F(\Si)\to M_\bQ$
satisfying the following properties.
\begin{enumerate}
\item $\bw(\tau,\si)$ is the weight of $T$-action 
on $T_{\fp_\si} \fl_\tau$, the tangent
line to $\fl_\tau$ at $\fp_\si$. In other words,
$$
\bw(\tau,\si) = c_1^T(T_{\fp_\si}\fl_\tau) = H^2_T(\fp_\si)=M_\bQ.
$$ 
\item Given any $\si\in \Si(r)$, the
set $\{\bw(\tau,\si)\mid \tau\in E_\si \}$ form
a $\bZ$-basis of $M_\si$.
These are the weights of the $T$-action on the tangent space 
$T_{\fp_\si}\cX$ to $\cX$ at 
the torus fixed (stacky) point $\fp_\si$.
\item Any $\tau\in \Si(r-1)_c$ is contained in two top dimensional
cones $\si,\si'\in \Si(r)$. 
\begin{enumerate}
\item $r(\tau,\si) \bw(\tau,\si) = -  r(\tau,\si') \bw(\tau,\si')\in M$.
\item $\fl_\tau^\rig \cong\cC_{r(\tau,\si), r(\tau,\si')}$. 
\end{enumerate}

\end{enumerate}
Let $\tau$ be as in (2). The normal
bundle of $\fl_\tau$ in $\cX$ is given by
$$
N_{\fl_\tau/\cX}\cong \cL_1\oplus\cdots \oplus  \cL_{r-1}
$$
where $\cL_i$ is a $\cT$-equivariant line bundle
over $\fl_\tau$ such that the weights of the
$T$-actions on the fibers $(\cL_i)_{\fp_\si}$ and $(\cL_i)_{\fp_{\si'} }$
are $\bw(\tau_i,\si) \in M_\si$ and $\bw(\tau_i',\si')\in M_{\si'}$, respectively.  
We have
$$
\bw(\tau'_i,\si')= \bw(\tau_i,\si) -a_i r(\tau,\si) \bw(\tau,\si) 
=\bw(\tau_i,\si_i)+ a_i r(\tau,\si') \bw(\tau,\si')
$$
where 
$$
a_i =\int_{\fl_\tau} c_1(\cL_i) \in \bQ.
$$

\subsection{Cohomology and equivariant cohomology}
In this section, we recall the result of \cite{BCS} on the
Chow ring of toric Deligne-Mumford stacks.
We also state the equivariant version.

Let $\cX=\cX_\bSi$ be the toric DM stack defined by a stacky
fan $\bSi=(N,\Si,\beta)$, and let $X=X_\Si$ be the simplicial
toric variety defined by the simplicial fan $\Si$. We assume that
$X$ is projective.

\begin{definition}
\begin{enumerate}
\item Let $I$ be the ideal in $\bQ[X_1,\ldots,X_s]$ generated by 
the monomials 
$\{ X_{i_1}\cdots X_{i_k}\mid v_{i_1},\ldots,v_{i_k}\textup{ do not generate a cone in $\Si$}\}$.
\item Let $J$ be the ideal in $\bQ[X_1,\ldots,X_s]$ generated by
$\{ \sum_{\alpha=1}^s \langle u,\bar{b}_\alpha\rangle X_\alpha \mid u\in M\}$.
\item Let $I'$ be the ideal in $R_T[X_1,\ldots,X_s]= \bQ[X_1,\ldots,X_s, u_1,\ldots,u_r]$
generated by the monomials
$\{ X_{i_1}\cdots X_{i_k}\mid v_{i_1},\ldots,v_{i_k}\textup{ do not generate a cone in $\Si$}\}$.
\item Let $J'$ be the ideal in $R_T[X_1,\ldots,X_s]= \bQ[X_1,\ldots,X_s, u_1,\ldots,u_r]$
generated by $\{ \sum_{\alpha=1}^s \langle u,\bar{b}_\alpha\rangle X_\alpha-u \mid u\in M\}$.
\item $\deg(X_\alpha)=2$, $\alpha=1,\ldots, s$; $\deg(u_i)=2$, $i=1,\ldots,r$.
\end{enumerate}
\end{definition}

With all the above definitions, the cohomology and equivariant
cohomology rings of $\cX$ can be describe explicitly as follows. 
\begin{theorem} We have the following isomorphisms of graded rings:
$$
\begin{aligned}
& H^*(\cX)\cong\bQ[X_1,\ldots,X_s]/ (I+J).\\
& H^*_T(\cX)= H^*_\cT(\cX)\cong  \bQ[X_1,\ldots,X_s,u_1,\ldots,u_r]/ (I'+J')\cong
\bQ[X_1,\ldots,X_s]/I.
\end{aligned}
$$
The isomorphism is given by $X_\alpha\mapsto c_1(\cO_\cX(\cD_\alpha))$ or 
$c_1^T(\cO_\cX(\cD_\alpha))$.
\end{theorem}

The ring $\bQ[X_1,\ldots,X_s]/I$ is known as the Stanley-Reisner ring.
The ring homomorphism 
$$
i_\cX^*:H^*_T(\cX)=\bQ[X_1,\ldots,X_s, u_1,\ldots,u_r]/(I'+J')\to 
H^*(\cX)=\bQ[X_1,\ldots,X_s]/(I+J)
$$
is surjective. The kernel is the ideal generated by
$u_1,\ldots, u_r$.  We say $\gamma^T\in H_T^*(\cX)$ is 
a $T$-equivariant lift of $\gamma\in H^*(X)$ if
$i_\cX^*(\gamma^T)=\gamma$. 

\begin{example}
Let $\cC_{s_1,s_2}$ be defined as in Example \ref{ex:Css}.
Then 
\begin{eqnarray*}
H^*(\cC_{s_1,s_2}) &\cong& \bQ[X_1, X_2]/\langle s_1 X_1-s_2 X_2, X_1X_2 \rangle \cong \bQ[X_1]/\langle X_1^2 \rangle,\\
H^*_T(\cC_{s_1,s_2})&\cong& \bQ[X_1,X_2]/\langle X_1 X_2\rangle. 
\end{eqnarray*}
\end{example}

\subsection{Orbifold cohomology and equivariant cohomology}
In this section, we recall the results of \cite{BCS} on orbifold
Chow ring. We also state the equivariant version.

For any $\si\in \Si$, define
$$
\mathrm{Box}(\si)=\{ v\in N\mid \bar{v}=\sum_{\rho_i\subset \si} q_i \bar{b}_i,\ 0\leq q_i<1\}
$$
Then there is an bijection between $\mathrm{Box}(\si)$ and
$N(\si)=N/N_\si$.  Define
$$
\mathrm{Box}(\bSi)=\bigcup_{\si\in \Si} \mathrm{Box}(\si).
$$
The inertia stack of $\cX =\cX_{\bSi}$ is 
$$
\cIX =\coprod_{v\in \mathrm{Box}(\bSi)} \cX(\bSi/\si(\bar{v}))
$$
where $\si(\bar{v})$ is the minimal cone in $\bSi$ containing $\bar{v}$.

Given $v =\sum_\alpha q_\alpha \bar{v}_{\alpha=1}^r \in \mathrm{Box}(\bSi)$, define
$$
X^v := \prod_{\alpha=1}^s X_\alpha^{q_\alpha}.
$$
As a $\bQ$-vector spaces, 
$$
H^*_\orb(\cX) = H^*(\cX(\bSi/\si(\bar{v})) [\deg(X^v)].
$$
Let $R_T= \bQ[u_1,\ldots, u_r]$. As an $R_T$-module, 
$$
H^*_{\orb,T}(\cX) = H^*_{\orb,\cT}(\cX) 
=\bigoplus_{v\in \mathrm{Box}(\bSi)} H^*_T(\cX(\bSi/\si(\bar{v})) [\deg(X^v)].
$$

\begin{example}\label{ex:Css-H}
Let $\bSi_{s_1,s_2}$, $\cC_{s_1,s_2}$, and $\{ \cC_{s_1,s_2,v} \}_{v=1-s_2}^{s_1-1}$ 
be defined as in Example \ref{ex:Css}. Then
$$
N=\bar{N}=\bZ,\quad \mathrm{Box}(\bSi_{s_1,s_2})=\{ v\in \bZ\mid 1-s_2 \leq s_1-1\}.
$$
$$
\cX(\bSi/\si(\bar{v})) = \cC_{s_1,s_2,v},\quad  1-s_2 \leq v \leq  s_1-1.
$$
As a $\bQ$-vector space, 
\begin{eqnarray*}
H^*_\orb(\cC_{s_1,s_2}) &=& 
H^*(\cC_{s_1,s_2}) \oplus \bigoplus_{i=1}^{s_1-1} H^*(\cC_{s_1,s_2,i})[\frac{2i}{s_1}] \oplus
\bigoplus_{j=1}^{s_2-1} H^*(\cC_{s_1,s_2,-j})[\frac{2j}{s_2}] \\
&=&\bQ 1 \oplus \bQ H \oplus \bigoplus_{i=1}^{s_1-1} \bQ 1_{\frac{i}{s_1}} \oplus \bigoplus_{j=1}^{s_2-1} \bQ 1'_{\frac{j}{s_2}},
\end{eqnarray*}
where $1_r, 1'_r\in H^{2r}_\orb(\cC_{s_1,s_2})$. 
\end{example}

We next describe the ring structure.  
\begin{definition}
\begin{enumerate}
\item As a $\bQ$-vector space, $\bQ[N]^{\bSi}= \oplus_{c\in N}\bQ y^c$.
\item As a $R_T$-module, $R_T[N]^{\bSi} =\oplus_{c\in N}R_T y^c $.
\item Define the multiplication on $\bQ[N]^{\bSi}$ and $R_T[N]^{\bSi}$ by  
$$
y^{c_1}\cdot y^{c_1} =
\begin{cases}
y^{c_1+c_2}, & \textup{if there is $\si\in \Si$ such that $\bar{c}_1\in \si$ and $\bar{c}_2 \in \si$},\\
0, & \textup{otherwise}.
\end{cases}
$$
\item Given $c\in N$, let $\si$ be the minimal cone in $\Si$ containing $\bar{c}\in \bar{N}$. Then
$\bar{c}=\sum_{\bar{b}_\alpha \in \si} m_\alpha \bar{b}_\alpha$ for some $m_\alpha\in \bQ_{\geq 0}$. Define
$$
\deg(y^c):= 2 \sum_{\bar{b}_\alpha \in \si} m_\alpha.
$$
\item Let $J$ be the ideal of $\bQ[N]^{\bSi}$ generated by 
$\{ \sum_{\alpha =1}^s \langle u_i,\bar{b}_\alpha\rangle y^{b_\alpha}\mid i=1,\ldots,r\}$.
\item Let $J'$ be the ideal of $R_T[N]^{\bSi}$ generated
by $\{ \sum_{\alpha=1}^s \langle u_i, \bar{b}_\alpha\rangle y^{b_\alpha} -u_i \mid i=1,\ldots,r \}$.  
\end{enumerate}
\end{definition}

\begin{theorem}
\begin{enumerate}
\item There is an isomorphisms of $\bQ$-graded rings:
$$
H^*_\orb(\cX) \cong \bQ[N]^{\bSi}/J.
$$
\item There is an isomorphism of $\bQ$-graded $R_T$-modules:
$$
H^*_{\orb,T}(\cX) = H^*_{\orb,\cT}(\cX) \cong R_T[N]^{\bSi}/J'.
$$
\end{enumerate}
\end{theorem}

\begin{example}
Let $\cC_{s_1,s_2}$ be defined as in Example \ref{ex:Css}. Then
\begin{eqnarray*}
H^*_\orb(\cC_{s_1,s_2}) &\cong & \bQ[y_1, y_2]/\langle y_1 y_2, s_1 y_1^{s_1} -s_2 y_2^{s_1} \rangle. 
\end{eqnarray*}
As a $\bQ$-graded $\bQ$-vector space,
\begin{eqnarray*}
H^*_\orb(\cC_{s_1,s_2}) &\cong & \bQ 1 \oplus \bQ H \bigoplus_{i=1}^{s_1-1} \bQ y_1^i \oplus \bigoplus_{i=1}^{s_2-1} \bQ y_2^j,
\end{eqnarray*}
where 
$$
H= s_1 y_1^{s_1} = s_2 y_2^{s_2},\quad  \deg(y_1)=\frac{2}{s_1},
\quad \deg(y_2)=\frac{2}{s_2}. 
$$
Let $1_r$ and $1'_r$ be defined as in Example \ref{ex:Css-H}. Then
$$
y_1^i = 1_{\frac{i}{s_1}} ,\quad y_2^j = 1'_{\frac{j}{s_2}}.
$$
\begin{eqnarray*}
H^*_{\orb,T}(\cC_{s_1,s_2}) &\cong & \bQ[y_1, y_2,u]/\langle y_1 y_2, s_1 y_1^{s_1} -s_2 y_2^{s_1} -u \rangle 
\end{eqnarray*}

\end{example}

%
%

\section{Orbifold Gromov-Witten Invariants of Smooth Toric DM Stacks}
\label{sec:toricGW-orb}

The main reference of this section is P. Johnson's thesis \cite{Jo},
which contains detailed localization computations for one-dimensional toric DM stacks.

Let $\cX$ be a toric DM stack of dimension $r$ defined by 
a stacky fan $\bSi =(N,\Si,\beta)$, and let $s= |\Si(1)|\geq r$.
Let 
$$
\cI \cX = \bigsqcup_{i\in I} \cX_i
$$ 
be the inertia stack of $\cX$, and let $\vi=(i_1,\ldots,i_n)\in I^n$.
The torus $T$ acts on $\cX$, and acts on the moduli stack $\MgXi$ by
$$
t \cdot [f:(\cC, \fx_1,\ldots,\fx_n)\to \cX]\mapsto [t\cdot f:(\cC,\fx_1,\ldots,\fx_n)\to \cX]
$$
where $(t\cdot f)(z)= t\cdot f(z)$, $z\in \cC$. The evaluation maps
$\ev_j:\MgXi\to \cX_{i_j}$ are $T$-equivariant and induce
$\ev_j^*:A^*_T(\cX_{i_j})\to A^*_T(\MgXi)$.

\subsection{Equivariant orbifold Gromov-Witten invariants}
Suppose that $\MgXi$ is proper, so that there are virtual fundamental classes
\begin{eqnarray*}
&& [\MgXi]^\vir \in A_{d^\vir_\vi}(\MgXi)\\
&& [\MgXi]^{\vir,T} \in A_{d^\vir_\vi }^T(\MgXi),
\end{eqnarray*}
where $\vi=(i_1,\ldots,i_n)\in I^n$, and 
$$
d^\vir_\vi= \int_\beta c_1(T\cX)+(r-3)(1-g)+n -\sum_{j=1}^n\age(\cX_{i_j}).
$$
Recall that the weighted virtual fundamental class is given by 
$$
[\MgXi]^w=\Bigl(\prod_{j=1}^n r_{i_j}\Bigr) [\MgXi]^\vir.
$$
Similarly,
$$
[\MgXi]^{w,T} =\Bigl(\prod_{j=1}^n r_{i_j}\Bigr)[\MgXi]^{\vir,T}
$$

Given $\gamma_j\in A^{d_j}(\cX_{i_j})=H^{2d_i}(\cX_{i_j})= H^{2(d_j+\age(\cX_{i_j}))}_\orb(\cX)$ and $a_j\in \bZ_{\geq 0}$, define
$\langle \bar{\tau}_{a_1}(\gamma_1)\cdots \bar{\tau}_{a_n}(\gamma_n)\rangle^\cX_{g,\beta}$ 
as in Section \ref{sec:orbGW}:
\begin{equation}\label{eqn:nonequivariant-orbGW}
\langle \bar{\tau}_{a_1}(\gamma_1)\cdots \bar{\tau}_{a_n}(\gamma_n)\rangle^\cX_{g,\beta}
=\int_{[\MgXi]^w}\prod_{j=1}^n \left(\ev_j^*\gamma_j\cup \bar{\psi}_j^{a_j}\right) \in \bQ.
\end{equation}
By definition, \eqref{eqn:nonequivariant-orbGW} is zero unless
$$
\sum_{j=1}^n d_j = d_\vi^\vir
$$
or equivalently,
$$
\sum_{j=1}^n ( d_j + \age(\cX_{i_j}) ) =\int_\beta c_1(T\cX)+ (r-3)(1-g)+n.
$$

In this case,
\begin{equation}\label{eqn:alphaT-orb}
\langle \bar{\tau}_{a_1}(\gamma_1)\cdots \bar{\tau}_{a_n}(\gamma_n)\rangle^\cX_{g,\beta}
=\int_{[\MgXi]^{w,T}}\prod_{j=1}^n \left(\ev_j^*\gamma_j^T\cup (\bar{\psi}_j^T)^{a_j}\right)
\end{equation}
where $\gamma_j^T\in A^{d_j}_T(X)$ is any $T$-equivariant lift of
$\gamma_j\in A^{d_j}(X)$,
and 
$$
\bar{\psi}_j^T\in A^1_T(\MgXi)
$$ 
is any $T$-equivariant lift of $\bar{\psi}_j\in A^1(\MgXi)$.

Given $\gamma_j^T\in A_T^{d_j}(\cX_{i_j})$, 
we define $T$-equivariant orbifold Gromov-Witten invariants
\begin{equation}\label{eqn:equivariant-orbGW}
\begin{aligned}
\langle\bar{\tau}_{a_1}(\gamma_1^T),\cdots,\bar{\tau}_{a_n}(\gamma_n^T)\rangle_{g,\beta}^{\cX_T}
& :=\int_{[\MgXi]^{w,T}} \prod_{j=1}^n \left(\ev_i^*\gamma_i^T (\bar{\psi}_i^T)^{a_i}\right)\\
&\in \bQ[u_1,\ldots,u_l](\sum_{j=1}^n d_j-d^\vir_\vi).
\end{aligned}
\end{equation}
where $\bQ[u_1,\ldots,u_l](k)$ is the space of degree $k$ homogeneous polynomials in 
$u_1,\ldots,u_l$ with rational coefficients. In particular, 
$$
\langle\bar{\tau}_{a_1}(\gamma_1^T),\cdots,\bar{\tau}_{a_n}(\gamma_n^T)\rangle_{g,\beta}^{\cX_T}
=\begin{cases}
0, & \sum_{i=1}^n d_i <d^\vir_\vi,\\
\langle\bar{\tau}_{a_1}(\gamma_1),\cdots,\bar{\tau}_{a_n}(\gamma_n)\rangle_{g,\beta}^\cX\in \bQ, &
\sum_{j=1}^n d_j = d^\vir_\vi.
\end{cases}
$$
where $\gamma_j=i_{\cX_{i_j}}^*\gamma_j^T\in A^{d_j}(\cX_{i_j})$.

In this section, we will compute the 
$T$-equivariant orbifold Gromov-Witten invariants \eqref{eqn:equivariant-orbGW}
by localization. 
Let $\MgXi^T\subset \MgXi$ be the substack of $T$ fixed points, 
and let $i:\MgXi^T\to \MgXi$ be the inclusion.
Let $N^\vir$ be the virtual normal bundle of substack 
$\MgXi^T$ in $\MgXi$; in general, $N^\vir$ has different
ranks on different connected components of $\MgXi^T$.  
By virtual localization,
\begin{equation}
\begin{aligned}
& \int_{[\MgXi]^{w,T}} \prod_{j=1}^n\left(\ev_j^*\gamma_j^T\cup (\bar{\psi}_j^T)^{a_j}\right) \\
= &\int_{[\MgXi^T]^{w,T}}\frac{i^*\prod_{j=1}^n\left(\ev_j^*\gamma_j^T\cup (\bar{\psi}_j^T)^{a_j}\right)}
{e^T(N^\vir)}.
\end{aligned}
\end{equation}

Indeed, we will see that $\MgXi^T$ is proper even when $\MgXi$ is not. When $\MgXi$ is not proper, we
{\em define} 
\begin{equation}\label{eqn:residue-orb}
\begin{aligned}
\langle\bar{\tau}_{a_1}(\gamma_1^T),\ldots,\bar{\tau}_{a_n}(\gamma_n^T)\rangle^\cX_{g,\beta}
= &\int_{[\MgXi^T]^{w,T}}
\frac{i^*\prod_{j=1}^n\left(\ev_j^*\gamma_j^T\cup (\bar{\psi}_j^T)^{a_j}\right)}{e^T(N^\vir)}\\
& \in  \bQ(u_1,\ldots,u_r).
\end{aligned}
\end{equation}
When $\MgXi$ is not proper, the right hand side of \eqref{eqn:residue-orb} is a rational function
(instead of a polynomial) in $u_1,\ldots,u_r$. It can be nonzero when
$\sum_{j=1}^n d_j< d^\vir_\vi$, and does not have a nonequivariant limit (obtained by setting $u_i=0$) 
in general.

\subsection{Torus fixed points and graph notation}\label{sec:graph-notation-orb}
In this subsection, we describe the $T$-fixed points in 
$\MgXi$. Given
a twisted stable map $f:(\cC,\fx_1,\ldots,\fx_n)\to \cX$ such that
$$
[f:(\cC,\fx_1,\ldots,\fx_n)\to \cX] \in \MgXi^T,
$$
we will associate a decorated graph  $\vGa$.  We first give a formal definition.
\begin{definition}\label{df:GgX-orb}
A decorated graph $\vGa=(\Ga, \vf, \vd, \vg, \vs, \vk)$
for $n$-pointed, genus $g$, degree $\beta$ 
stable maps to $\cX$  consists of the following data.

\begin{enumerate}
\item $\Ga$ is a compact, connected 1 dimensional CW complex. 
We denote the set of vertices (resp. edges) in $\Ga$ 
by $V(\Ga)$ (resp. $E(\Ga)$). The set of flags of $\Gamma$ is
defined to be
$$
F(\Ga)=\{(e,v)\in E(\Ga)\times V(\Ga)\mid v\in e\}.
$$

\item The {\em label map} $\vf: V(\Ga)\cup E(\Ga)\to \Si(r)\cup \Si(r-1)_c$
sends a vertex $v\in V(\Ga)$ to 
a top dimensional cone $\si_v \in \Si(r)$, and 
sends an edge $e\in E(\Ga)$ to
an $(r-1)$-dimensional cone $\tau_e \in \Si(r-1)_c$. 
Moreover, $\vf$ defines a map from the graph $\Ga$
to the graph $\Up$: if $(e,v)\in F(\Ga)$ 
then $(\tau_e,\si_v)\in F(\Si)$.

\item The {\em degree map} $\vd:E(\Ga)\to \bZ_{>0}$
sends an edge $e\in E(\Ga)$ to a positive
integer $d_e$.

\item The {\em genus map} $\vg:V(\Ga)\to \bZ_{\geq 0}$
sends a vertex $v\in V(\Ga)$ to a nonnegative
integer $g_v$.

\item The {\em marking map} $\vs: \{1,2,\ldots,n\}\to V(\Ga)$
is defined if $n>0$.

\item The {\em twisting map} $\vk$ sends an edge $e\in E(\Ga)$ to
an element $k_e\in G_e:= G_{\tau_e}$, a flag $(e,v)$ to an element
$k_{(e,v)}\in G_v:= G_{\si_v}$, a marking $j\in \{1,\ldots,n\}$ to 
an element $k_j \in G_v$ if $\vi(j)=v$. 
\end{enumerate} 

The above maps satisfy the following two constraints:
\begin{enumerate}
\item[(i)] (topology of the domain)
$\displaystyle{\sum_{v\in V(\Ga)} g_v + |E(\Ga)| - |V(\Ga)| +1 = g}$.
\item[(ii)] (topology of the map)
$\displaystyle{ \sum_{e\in E(\Ga)} d_e[\ell_{\tau_e}] =\beta}$.
\item[(iii)] (compatibility along an edge)
Given any edge $e\in E(\Ga)$, let $v, v'\in V(\Ga)$ be its two ends.
Then $k_{(e,v)} \in G_v$ and $k_{(e,v')} \in G_{v'}$ are determined by 
$d_e \in \bZ_{>0}$ and $k_e\in G_e$ \cite[Lemma II.13]{Jo}.    

\item[(iv)] (compatibility at a vertex)
Given $v\in V(\Ga)$, let  $E_v$ and $S_v$ be defined as in Definition \ref{df:unstable}.
Then
$$
\prod_{e\in E_v} k_{(e,v)}^{-1} \prod_{j\in S_v} k_j =1. 
$$
In particular, if $(e,v)\in F(\Ga)$ and $v\in V^1(\Ga)$ then
$k_{(e,v)}=1 \in G_v$. 

\item[(v)] (compatibility with $\vi=(i_1,\ldots, i_n)$)
Given $j\in \{1,\ldots,n\}$, if $\vs(j)= v$, then
the pair $(p_{\si_v}, k_j)$ represent a point in $\cX_{i_j}$, the connected
component of $\cI\cX$ labelled by $i_j$. 
\end{enumerate}

Let $\GgXi$ be the set of all decorated graphs
$\vGa=(\Ga,\vf, \vd, \vg,\vs, \vk)$ satisfying the above
constraints.
\end{definition}

Let $f:(\cC,\fx_1,\ldots,\fx_n)\to \cX$ be a twisted stable map which represents
a $T$ fixed point in $\MgXi$. Let $\bar{f}:(C,x_1,\ldots, x_n)\to X$ be the corresponding
stable map between coarse moduli spaces. Then $\bar{f}:(C,x_1,\ldots,x_n)\to X$
represents a $T$ fixed point in $\MgX$, so we may define, as in Section
\ref{sec:graph-notation}, $\Gamma, \vf, \vd,\vg, \vs$, $C_v$ for each vertex $v\in V(\Ga)$, and $C_e$ 
for each edge $e\in E(\Ga)$. It remains to define the twisting map $\vk$. Let $\cC_v$ (resp. $\cC_e$) be
the preimage of $C_v$ (resp. $C_e$) under the projection $\cC\to C$. 
\begin{itemize}
\item Given an edge $e\in E(\Gamma)$, the map $f_e:=f|_{\cC_e}:\cC_e\to \fl_\tau$  is determined by 
the degree $d_e$ of the map $\bar{f}_e:=\bar{f}|_{C_e}:C_e =\bP^1 \to \ell_\tau =\bP^1$ and
$k_e\in G_{\tau_e}$. Define $\vk(e)=k_e$. 
\item Given $(e,v)\in F(\Ga)$, let $\fy(e,v) =\cC_e \cap \cC_v$. 
Define $\vk(e,v) =k_{(e,v)}\in G_v$ to
be the image of the generator of the stabilizer of the stacky point
$\fy(e,v)$ in the orbicurve $\cC_e$.
\item Under the evaluation map $\ev_j$, the $j$-th marked point $\fx_j$ is mapped to 
$(\fp_\si, k)$ in the inertial stack $\cI\cX$, where $\si\in V(\Sigma)$ and $k\in G_\si$. 
Then $\vf\circ \vs(j)=\si$. Define $\vk(j)=k_j= k$. 
\end{itemize}
Define
\begin{equation}\label{eqn:rev}
r_{(e,v)}= |\langle k_{(e,v)}\rangle|.
\end{equation}
where $\langle k_{(e,v)}\rangle$ is the subgroup of $G_v$ generated by $k_{(e,v)}$. 
Suppose that $v, v'\in V(\Ga)$ are the two end points of the edge $e\in E(\Ga)$.
Then 
$$
\fl_\tau^\rig \cong \cC_{r(\tau_e, \si_v), r(\tau_e, \si_{v'})},
\quad
\cC_e \cong \cC_{r_{(e,v)}, r_{(e,v')}}.
$$

To summarize, we  have a map from $\MgXi^T$ to the discrete set $\GgXi$. 
Let $\cF_\vGa\subset \MgXi^T$ denote the preimage of $\vGa$.
Then
$$
\MgXi^T=\bigsqcup_{\vGa\in \GgXi}\cF_\vGa
$$
where the right hand side is a disjoint union of connected components.

We now describe the fixed locus $\cF_\vGa$ associated to
each decorated graph $\vGa\in \GgXi$. 
Given an edge $e\in E(\Ga)$, the map $f_e: \cC_e\to \fl_\tau$,
where $\tau=\vf(e)$, is determined by $\vGa$ up to isomorphism. 
The automorphism group of $f_e$ is $G_e\times \bZ_{\vd(e)}$. 
The moduli space of $f_e$ is
$$
\cM_e = \cB (G_e \times \bZ_{\vd(e)}). 
$$
Given a stable vertex $v\in V^S(\Gamma)$,  
the map $f_v:=f|_{\cC_v}: \cC_v\to \fp_\si =\cB G_v$, where $\si=\vf(v)$, represents a point
in $\Mbar_{g_v, E_v\cup S_v}(\fp_\si)$, where $E_v$ and $S_v$ are defined as in Definition
\ref{df:unstable}.
For each $e\in E_v\subset E(\Ga)$, there is an evaluation map
$$
\ev_{(e,v)}:\Mbar_{g_v,E_v\cup S_v}(\fp_\si) \to \cI \fp_{\si_v}. 
$$
For each $j\in S_v \subset \{1,\ldots,n\}$, there is an evaluation map
$$
\ev_j:\Mbar_{g_v,E_v\cup S_v}(\fp_\si)\to \cI \fp_{\si_v}.
$$
We have 
$$
\cI\fp_{\si_v} \cong \cI \cB G_v =\bigsqcup_{k\in G_v} (\cB G_v)_k,
$$
where $(\cB G_v)_k$ are connected components of $\cI\cB G_v$ (see Example \ref{ex:BG}). 
The moduli space of $f_v$ is 
$$
\Mbar_{g_v,\vi_v}(\cB G_v): = \bigcap_{e\in E_v} \ev_{(e,v)}^{-1}( (\cB G_v)_{k_{(e,v)}^{-1}})\cap \bigcap_{j\in S_v} \ev_j^{-1}( (\cB G_v)_{k_j}). 
$$

To obtain a $T$ fixed point $[f:(\cC,\fx_1,\ldots, \fx_n)\to \cX]$, we glue
the the above maps $f_v$ and $f_e$ along the nodes. 
Let $V^2(\Ga)$ and $F^S(\Ga)$ be defined as in Definition \ref{df:unstable}.
The nodes of $\cC$ are 
$$
\{ \fy_{(e,v)}=\cC_e\cap \cC_v\mid (e,v)\in F^S(\Ga) \} \cup \{ \fy_v = \cC_v \mid v\in V^2(\Ga), E_v=\{e_1,e_2\}\}.
$$
We define $\widetilde{\cM}_\vGa$ by the following 2-cartesian diagram
$$
\begin{CD}
\widetilde{\cM}_\vGa @>{f_E}>>  &  \prod_{e\in E(\Ga)} \cM_e\\
@V{f_V}VV &  @V{\ev_E}VV \\
\prod_{v\in V^S(\Ga)} \Mbar_{g_v, \vi_v}(\cB G_v)  @>{\ev_V}>> & \prod_{(e,v)\in F^S(\Ga)} \overline{\cI} \cB G_v \times \prod_{v\in V^2(\Ga)} \overline{\cI} \cB G_v
\end{CD}
$$
where $\ev_V$ and $\ev_E$ are given by evaluation at nodes, 
and $\overline{\cI}\cB G_v$ is the rigidified inertia stack.
More precisely:
\begin{itemize}
\item  For every stable flag $(e,v)\in F^S(\Ga)$,  
let $\ev_{(e,v)}$ be the evaluation map at the node $\fy_{(e,v)}$,
and let $\overline{\ev}_{(e,v)} =\iota\circ \ev_{(e,v)}$, where
$\iota$ is the involution on $\cI\cB G_v$.
\item  For each $v\in V^2(\Ga)$,  let $E_V=\{e_1,e_2\}$ (we
pick some ordering of the two edges in $E_v$), 
let $\ev_{(e_1,v)}$ be the evaluation map at the node $\fy_v$, and
let $\ev_{(e_2,v)} =\iota \circ \ev_{(e_2,v)}$. 
\item  Define
\begin{eqnarray*}
\ev_V &=& \prod_{ (e,v)\in F^S(\Ga)} \ev_{(e,v)} \\
\ev_E &=& \prod_{(e,v)\in F^S(\Ga)} \overline{\ev}_{(e,v)} \times \prod_{\tiny \begin{array}{c} v\in V^2(\Ga)\\ (e,v)\in F(\Ga) \end{array}} \ev_{(e,v)}. 
\end{eqnarray*}
\end{itemize}
The fixed locus associated to the decorated graph $\vGa$ is
$$
\cF_\Ga = \widetilde{\cM}_\vGa/\Aut(\vGa).
$$

From the above definitions, up to some finite morphism, $\cF_\Ga$ can be identified with
$$
\cM_{\vGa} :=  \prod_{v\in V^S (\Ga)} \Mbar_{g_v,\vi_v}(\cB G_v),  
$$
and 
$$
[\cF_\vGa] =  c_\vGa [\cM_{\vGa}]\in A_*(\cM_{\vGa})
$$
where
\begin{equation}\label{eqn:cGa}
c_\vGa =  \frac{1}{|\Aut(\vGa)|\prod_{e\in E(\Ga)}(d_e|G_e|)  } \cdot
\prod_{(e,v)\in F^S(\Ga)} \frac{|G_v|}{r_{(e,v)}}\cdot \prod_{v\in V^2(\Ga)} \frac{|G_v|}{r_v}. 
\end{equation}
In the above equation:
\begin{itemize}
\item $\displaystyle{ \frac{|G_v|}{r_{(e,v)}} = |G_v/\langle k_{(e,v)}\rangle| }$,
where $G_v/\langle k_{(e,v)}\rangle$ is the automorpshim
group of $k_{(e,v)}^{-1}$ in the rigidified inertial stack $\overline{\cI} \cB G_v$.
\item  If $v\in V^2(\Ga)$ and $E_v=\{ e_1, e_2\}$, we define $r_v = r(e_1,v)= r(e_2,v)$. 
\end{itemize}
\subsection{Virtual tangent and normal bundles}
Given a decorated graph $\vGa\in \GgXi$ and
a twisted stable map $f:(\cC,\fx_1,\ldots,\fx_n)\to \cX$
which represents a point in the fixed locus
$\cF_{\vGa}$ associated to $\vGa$, let
\begin{eqnarray*}
&& B_1 =  \Hom(\Omega_\cC(\fx_1+\cdots+\fx_n),\cO_\cC),\quad B_2 =   H^0(\cC,f^*T\cX)\\
&& B_4 = \Ext^1(\Omega_\cC(\fx_1+\cdots+ \fx_n),\cO_\cC),\quad B_5= H^1(\cC,f^*T\cX)
\end{eqnarray*}
$T$ acts on $B_1, B_2, B_3, B_4$. 
Let $B_i^m$ and $B_i^f$ be the moving and fixed parts
of $B_i$, respectively. Then
\begin{equation}
0\to B_1^f\to B_2^f\to T^{1,f}\to B_4^f\to B_5^f\to T^{2,f}\to 0
\end{equation}
\begin{equation}
 0\to B_1^m\to B_2^m\to T^{1,m}\to B_4^m\to B_5^m\to T^{2,m}\to 0
\end{equation}

The irreducible components of $\cC$ are
$$
\{ \cC_v\mid v\in V^S(\Ga)\}\cup\{\cC_e\mid e\in  E(\Ga) \}.
$$
Recall that the nodes of $\cC$ are
$$
\{\fy(e,v) =\cC_e\cap \cC_v \mid (e,v)\in F^S(\Ga)\}\cup \{ \fy_v=\cC_v \mid v\in V^2(\Ga) \}.
$$

\subsubsection{Automorphisms of the domain} \label{sec:aut-orb}
\begin{eqnarray*}
B_1^f &=&\bigoplus_\edge \Hom(\Omega_{\cC_e}(\fy(e,v)+\fy(e,v')),\cO_{\cC_e})\\
&=& \bigoplus_\edge H^0(\cC_e, T\cC_e(-\fy(e,v)-\fy(e,v'))\\
B_1^m&=& \bigoplus_\vone T_{ \fy(e,v) }\cC_e 
\end{eqnarray*}
We define
$$
w_{(e,v)} :=  e^T(T_{\fy(e,v) }\cC_e)=\frac{r(\tau_e,\si_v)\bw(\tau_e,\si_v)}{r_{(e,v)}d_e} \in 
H_T^2(\fy(e,v))= M_\bQ.
$$

\subsubsection{Deformations of the domain} \label{sec:deform-orb}
Given any $v\in V^S(\Ga)$, define
a divisor $\bx_v$ of $\cC_v$ by
$$
\bx_v=\sum_{i\in S_v} \fx_i + \sum_{e\in E_v} \fy(e,v).
$$
Then
\begin{eqnarray*}
B_4^f&=& \bigoplus_{v\in V^S(\Ga)} \Ext^1(\Omega_{\cC_v}(\bx_v),\cO_\cC) = \bigoplus_{v\in V^S(\Ga)}
T\Mbar_{g_v, \vi_v}(\cB G_v)\\
B_4^m &= & \bigoplus_{v\in V^2(\Ga), E_v=\{e,e'\} }
T_{\fy_v}\cC_e\otimes T_{\fy_v} \cC_{e'} \oplus \bigoplus_{(e,v)\in F^S(\Ga)} 
T_{ \fy(e,v) }\cC_v\otimes T_{ \fy(e,v) } \cC_e
\end{eqnarray*}
where
\begin{eqnarray*}
e^T (T_{\fy_v}\cC_e \otimes T_{\fy_v} \cC_{e'})&=& w_{(e,v)}+w_{(e',v)},\quad v\in V^2(\Ga)\\
e^T (T_{\fy(e,v) }\cC_v \otimes T_{\fy(e,v) } \cC_e) &=& w_{(e,v)}-\frac{\bar{\psi}_{(e,v)}}{r_{(e,v)}},\quad v\in V^S(\Ga)
\end{eqnarray*}

\subsubsection{Unifying stable and unstable vertices}
From the discussion in Section \ref{sec:aut-orb} and Section \ref{sec:deform-orb},
\begin{equation} \label{eqn:Bonefour-orb}
\begin{aligned}
\frac{e^T(B_1^m)}{e^T(B_4^m)}=& 
\prod_\vone w_{(e,v)} 
\prod_{v\in V^2(\Ga), E_v=\{e,e'\} }
\frac{1}{w_{(e,v)}+ w_{(e',v)} }\\
&\cdot \prod_{v\in V^S(\Ga)}\frac{1}{\prod_{e\in E_v}(w_{(e,v)}-\bar{\psi}_{(e,v)}/r_{(e,v)}) }.
\end{aligned}
\end{equation}

Recall that
$$
\cM_\vGa =\prod_{v\in V^S(\Ga)} \Mbar_{g_v,\vi_v}(\cB G_v).
$$
$$
c_\vGa =\frac{1}{|\Aut(\vGa)\prod_{e\in E(\Ga)} (d_e|G_e|)}\prod_{(e,v)\in F^S(\Ga)}\frac{|G_v|}{r_{(e,v)} }
\prod_{v\in V^2(\Ga)} \frac{|G_v|}{r_v}. 
$$

To unify the stable and unstable vertices, we use the following
convention for the empty sets $\Mbar_{0,(1)}(\cB G)$ and $\Mbar_{0, (c,c^{-1})}(\cB G)$,
where $1\in G$ is the identity element, and $c\in G$. 
Let $G$ be a finite abelian group. Let $w_1, w_2$ be formal variables.
\begin{itemize}
\item   $\Mbar_{0,(1)}(\cB G)$ is a $-2$ dimensional space, and
\begin{equation}\label{eqn:one-orb}
\int_{\Mbar_{0,(1)}(\cB G)}\frac{1}{w_1-\bar{\psi}_1}=\frac{w_1}{ |G| } 
\end{equation}
\item  $\Mbar_{0,(c,c^{-1})}(\cB G)$ is a $-1$ dimensional space, and
\begin{equation}\label{eqn:two-orb}
\int_{\Mbar_{0,(c, c^{-1})} (\cB G)}\frac{1}{(w_1-\bar{\psi}_1)(w_2-\bar{\psi}_2)}= \frac{1}{(w_1+w_2)\cdot |G| }
\end{equation}
\begin{equation}\label{eqn:one-one-orb}
\int_{\Mbar_{0,(c,c^{-1})}(\cB G)}\frac{1}{w_1-\bar{\psi}_1} =\frac{1}{ |G| }
\end{equation}
\end{itemize}
From \eqref{eqn:one-orb}, \eqref{eqn:two-orb}, \eqref{eqn:one-one-orb}, we obtain the following identities for non-stable vertices:
\begin{itemize}
\item[(i)] If $v\in V^1(\Ga)$ and $(e,v)\in F(\Ga)$, then $r_{(e,v)}=1$, and 
$$
|G_v| \int_{\Mbar_{0,(1)} (\cB G_v) } \frac{1}{w_{(e,v)}-\bar{\psi}_{(e,v)}} = w_{(e,v)}.
$$
\item[(ii)] If $v\in V^2(\Ga)$ and $E_v=\{ e,e'\}$, let $c= \rho(e,v)=\rho(e',v)^{-1}\in G_v$, then
\begin{eqnarray*}
&&  \frac{|G_v|}{r_v}\cdot \frac{|G_v|}{r_v} \cdot \int_{\Mbar_{0, (c,c^{-1})}(\cB G_v)  }
\frac{1}{ (w_{(e,v)} -\bar{\psi}_{(e,v)}/r_v)(w_{(e',v)}-\bar{\psi}_{(e',v)}/r_v) }\\
&=&\frac{|G_v|}{r_v} \cdot \frac{1}{w_{(e,v)} + w_{(e',v)}}.
\end{eqnarray*}
\item[(iii)]If $v\in V^{1,1}(\Ga)$ and $(e,v)\in F(\Ga)$, then
$$
\frac{|G_v|}{r_{(e,v)}} \int_{\Mbar_{0,(c,c^{-1})}(\cB G_v)}\frac{1}{w_{(e,v)}- \bar{\psi}_1/r_{(e,v)} } =1.
$$
\end{itemize}
We then redefine $\cM_\vGa$ and $c_\vGa$ as follows:
\begin{equation}
\cM_\vGa =\prod_{v\in V(\Ga)} \Mbar_{g_v, \vi_v}(\cB G_v),\quad  [\cF_\vGa]=c_\vGa [\cM_\vGa],
\end{equation}
\begin{equation}\label{eqn:unified-c}
c_\vGa = \frac{1}{|\Aut(\vGa)|\prod_{e\in E(\Ga)} (d_e|G_e|)}  \prod_{(e,v)\in F(\Ga)}\frac{|G_v|}{r_{(e,v)}}.  
\end{equation}

With the above conventions \eqref{eqn:one-orb}--\eqref{eqn:unified-c}, we may rewrite \eqref{eqn:Bonefour-orb} as
\begin{equation}
\frac{e^T(B_1^m)}{e^T(B_4^m)}=
\prod_{v\in V(\Ga)} \frac{1}{ \prod_{e\in E_v}(w_{(e,v)}-\bar{\psi}_{(e,v)}/r_{(e,v)})  }.
\end{equation}

The following lemma shows that the conventions \eqref{eqn:one-orb}, \eqref{eqn:two-orb}, and 
\eqref{eqn:one-one-orb} are consistent with
the stable case $\Mbar_{0,(c_1,\ldots, c_n)}(\cB G)$, $n\geq 3$.
\begin{lemma} Let $G$ be a finite abelian group. Let $\vc=(c_1,\ldots, c_n)\in G^n$,
where $c_1 \cdots c_n =1$. Let $w_1,\ldots,w_n$ be formal variables. Then 
\begin{enumerate}
\item[(a)]$\displaystyle{ \int_{\Mbar_{0,\vc}(\cB G)}\frac{1}{\prod_{i=1}^n(w_i-\bar{\psi}_i)}
=\frac{1}{|G|\cdot  w_1\cdots w_n}(\frac{1}{w_1}+\cdots \frac{1}{w_n})^{n-3} }$.
\item[(b)]$\displaystyle{\int_{\Mbar_{0,\vc}(\cB G)}\frac{1}{w_1-\bar{\psi}_1} =\frac{ w_1^{2-n}}{|G|} }$.
\end{enumerate}
\end{lemma}
\begin{proof}
The unstable cases $n=1$ and $n=2$ follow from the definitions
\eqref{eqn:one-orb} and \eqref{eqn:two-orb}, respectively.
The stable case ($n\geq 3$) follows from Corollary \ref{cor:abelian-psi} and Lemma \ref{lemma:psi}.
\end{proof}

\subsubsection{Deformation of the map}
We first introduce some notation. Given $\si\in \Si(r)$ and $ k \in G_\si$, 
let $\bigl(T_{\fp_\si}\cX\bigr)^k$ denote the subspace which is invariant
under the action of $k$ on $T_{\fp_\si}\cX$. Then
$$
\bigl(T_{\fp_\si}\cX\bigr)^k =\bigl(T_{\fp_\si}\cX\bigr)^{k^{-1}}. 
$$

Consider the normalization sequence
\begin{equation}\label{eqn:normalize-orb}
\begin{aligned}
0 &\to \cO_\cC\to \bigoplus_{v\in V^S(\Ga)} \cO_{\cC_v} \oplus \bigoplus_{e\in E(\Ga)} \cO_{\cC_e}\\
& \to \bigoplus_{v\in V^2(\Ga)} \cO_{\fy_v}
\oplus \bigoplus_{(e,v)\in F^S(\Ga) } \cO_{ \fy(e,v) }\to 0.
\end{aligned}
\end{equation}
We twist the above short exact sequence of sheaves
by $f^*T\cX$. The resulting short exact sequence gives
rise a long exact sequence of cohomology groups
\begin{eqnarray*}
0&\to& B_2 \to \bigoplus_{v\in V^S(\Ga)} H^0(\cC_v)\oplus
\bigoplus_{e\in E(\Ga)}H^0(\cC_e) \\
&\to& \bigoplus_{\tiny \begin{array}{c} v\in V^2(\Ga) \\ E_v=\{e,e'\} \end{array}} (T_{f(\fy_v)}\cX)^{k_{(e,v)}}
\oplus \bigoplus_{(e,v)\in F^S(\Ga)} \bigl(T_{f(\fy(e,v))}\cX\bigr)^{k_{(e,v)}} \\ 
&\to& B_5 \to \bigoplus_{v\in V^S(\Ga)} H^1(\cC_v)\oplus
\bigoplus_{e\in E(\Ga)}H^1(\cC_e) \to 0.
\end{eqnarray*}
where
$$
H^i(\cC_v) = H^i(\cC_v, f_v^*T\cX),\quad 
H^i(\cC_e) = H^i(\cC_e, f_e^*T\cX)
$$
for $i=0,1$. 

$f(\fy_v)=\fp_{\si_v}= f(\fy(e,v))$. Given $(e,v)\in F(\Gamma)$, define
\begin{equation}\label{eqn:hev}
\bh(e,v) =e^T(\bigl(T_{\fp_\si} \cX\bigr)^{k_{(e,v)}}) 
=\prod_{(\tau,\si_v)\in F(\Si), \langle k_{(e,v)}\rangle  \subset G_\tau} \bw(\tau,\si_v). 
\end{equation}
The map $B_1\to B_2$  sends
$H^0(\cC_e, T\cC_e(-\fy(e,v)-\fy(e',v)))$ isomorphically
to $H^0(\cC_e, f_e^*T\fl_{\tau_e})^f$, the 
fixed part of $H^0(\cC_e, f_e^*T\fl_{\tau_e})$.

It remains to compute
$$
\bh(v) := \frac{ e^T(H^1(\cC_v, f_v^*T\cX)^m) }{e^T(H^0(\cC_v, f_v^*T\cX)^m)} ,\quad
\bh(e) := \frac{e^T(H^1(\cC_e, f_e^*T\cX)^m)}{e^T(H^0(\cC_e, f_e^* T\cX)^m)}
$$
We first introduce some notation.
\begin{itemize}
\item If $v\in V^S(\Ga)$, then there is a cartesian diagram
$$
\begin{CD}
\tcC_v @>{\tf_v}>> \mathrm{pt}\\
@VVV @VVV\\
\cC_v @>{f_v }>> \cB G_v.
\end{CD} 
$$
Let $\hG_v$ denote the subgroup of $G_v$ generated by the monodromies
of the $G_v$-cover $\tcC_v\to \cC_v$. Then 
the number of connected components of 
$\tcC_v$ is $|G_v/\hG_v|$, and each connected component
is a $\hG_v$-cover of $\cC_v$.

\item Given $(\tau,\si)\in F(\Si)$, let  $\phi(\tau,\si)\in G_\si^*$ be the irreducible character which
corresponds to the 1-dimensional $G_\si$-representation $T_{p_\si}\fl_\tau$.
\item Given an irreducible character $\phi$ of $G_v$, let $\bC_\phi$ denote
the 1-dimensional $G_v$-representation associated to $\phi$. Define
$$
\Lambda_\phi^\vee(u)=\sum_{i=0}^{\rank \bE_\phi} (-1)^i \lambda_i^\phi u^{\rank \bE_\phi-i},
$$
where $\lambda_i^\phi \in A^i(\Mbar_{g_v, \vi_v}(\cB G_v))$ are Hurwitz-Hodge classes associated
to $\phi\in G_v^*$. Here $\rank\bE_\rho$ is the rank of $\bE_\rho\to\Mbar_{g_v,\vi_v}(\cB G_v)$.
The rank of a Hurwitz-Hodge bundle $\bE_\rho \to \Mbar_{g,\vc}(\cB G)$, where
$G$ is any finite group and $\rho\in G^*$,  is given in Section \ref{sec:hurwitz-hodge}.  
\item Given a $G_v$ representation $V$, let $V^{G_v}$ denote the subspace
on which $G_v$ acts trivially. 
\end{itemize}
\begin{lemma}
Suppose that $v\in V^S(\Ga)$ and $\vf(v)=\si\in \Si(r)$. Then
\begin{equation} \label{eqn:hv}
\bh(v) =  \frac{ \displaystyle{ \prod_{(\si,\tau)\in E(\Ga)} \Lambda^\vee_{\phi(\tau,\si)}(\bw(\tau,\si)) } }{ 
\displaystyle{ \prod_{(\si,\tau)\in E(\Ga), \hG_v \subset G_\tau} \bw(\tau,\si)} }
\end{equation}
\end{lemma}

\begin{proof} 
We have
\begin{eqnarray*}
H^i(\cC_v, f_v^*T\cX) &=& \left(H^i(\tcC_v,\cO_{\tcC_v})\otimes T_\si \cX \right)^{G_v} \\
&\cong & \bigoplus_{(\tau,\si)\in F(\Ga)} \left(H^i(\tcC_v, \cO_{\tcC_v})\otimes \bC_{\phi(\tau,\sigma)} \right)^{G_v}.
\end{eqnarray*}
The group homomorphism $G_v\to G_v/\hG_v$ induces an inclusion
$(G_v/\hG_v)^* \to G_v^*$ of sets of irreducible characters, so 
$(G_v/\hG_v)^*$ can be viewed as a subset of $G_v^*$.
$H^0(\tcC_v,\cO_{\tcC_v})$ is the regular representation of $G_v/\hG_v$, so 
$$
H^0(\tcC_v,\cO_{\tcC_v}) = \bigoplus_{\phi\in (G_v/\hG_v)^*} \bC_\phi.
$$
$\phi(\tau,\si) \in (G_v/\hG_v)^*$ iff $\hG_v\subset G_\tau$, so 
$$
e_T\Big( \big(H^0(\tcC_v,\cO_{\tcC_v}) \otimes \bC_{\phi(\tau,\sigma)}\big)^{G_v}\Big)
=\begin{cases} \bw(\tau,\si), & \hG_v\subset G_\tau,\\
1, & \hG_v \not \subset G_\tau.
\end{cases}
$$
Therefore,
\begin{equation}\label{eqn:denominator}
e_T(H^0(\cC_v, f_v^*T\cX)^m)=e_T(H^0(\cC_v, f_v^*T\cX)) = \prod_{(\tau,\si)\in F(\Ga), \hG_v \subset G_\tau} \bw(\tau,\si)
\end{equation}
$$
\bigl(H^1(\tcC_v,\cO_{\cC_v})\otimes \bC_{\phi(\tau,\si)}\bigr)^{G_v} = \bE_{\phi(\tau,\si)}^\vee,
$$
so 
\begin{equation}\label{eqn:numerator}
e_T(H^1(\cC_v, f_v^*T\cX)^m)= e_T(H^1(\cC_v, f_v^*T\cX)) = \prod_{(\tau,\si)\in F(\Ga)} \Lambda^\vee_{\phi(\tau,\si)} (\bw(\tau,\si)).
\end{equation}
Equation \eqref{eqn:hv} follows from \eqref{eqn:denominator} and \eqref{eqn:numerator}. 
\end{proof}

\begin{lemma}
Suppose that $e\in E(\Ga)$. Let $d=d_e\in \bZ_{>0}$, and let $\tau=\vf(e)\in \Si(r-1)_c$.
Define $\si, \si', \tau_i, \tau_i', a_i$ as in Section \ref{sec:toric-graph}.
Suppose that $(e,v), (e,v')\in F(\Ga)$, $\vf(v)= \si$, $\vf(v')=\si'$. 
Then $k_{(e,v)}\in G_\si$ acts on $T_{\fp_\si}\fl_\tau$ by multiplication
by $e^{2\pi\sqrt{-1}\langle d/r(\tau,\si)\rangle }$, and acts on $T_{\fp_\si}\fl_{\tau_i}$ by 
$e^{2\pi\sqrt{-1} \ep_j}$, where
$$
\langle \frac{d}{r(\tau,\si)}\rangle , \ep_1,\ldots, \ep_{r-1} 
\in \{ 0, \frac{1}{r_{(e,v)} },\ldots, \frac{r_{(e,v)}-1 }{r_{(e,v)}} \}.  
$$
Define
$$
\uu=r(\tau,\si)\bw(\tau,\si) = -r(\tau,\si')\bw(\tau,\si').
$$
Then
\begin{equation}\label{eqn:he}
\bh(e)=\frac{ (\frac{d}{\uu})^{\lfloor\frac{d}{r(\tau,\si)}\rfloor } }
{\lfloor\frac{d}{r(\tau,\si)}\rfloor !}
\frac{ (-\frac{d}{\uu})^{\lfloor\frac{d}{r(\tau,\si')}\rfloor } }
{\lfloor\frac{d}{r(\tau,\si')}\rfloor !}\prod_{i=1}^{r-1} \mathbf{b}_i
\end{equation}
where
\begin{equation}\label{eqn:b-orb}
\mathbf{b}_i=\begin{cases}
\displaystyle{\prod_{j=0}^{\lfloor da_i-\ep_i \rfloor} (\bw(\tau_i,\si) -(j+\ep_i) \frac{\uu}{d})^{-1} }, 
& a_i\geq 0,\\
\displaystyle{\prod_{j=1}^{\lceil \ep_i -da_i -1 \rceil } (\bw(\tau_i,\si)+ (j-\ep_i) \frac{\uu}{d}) },  
& a_i <0.
\end{cases}
\end{equation}
\end{lemma}

\begin{proof}
Let
$$
\bw_i =\bw(\tau_i,\si),\quad i=1,\ldots,r-1.
$$
We have
$$
N_{\fl_\tau/\cX}=\cL_1\oplus \cdots \oplus \cL_{r-1}.
$$
\begin{itemize} 
\item The weights of $T$-actions on $(\cL_i)_{\fp_\si}$
and $(\cL_i)_{\fp_{\si'}}$ are 
$\bw_i$ and $\bw_i- a_i \uu $,
respectively. 
\item The weights of $T$-action on $T_{\fp_\si}\fl_\tau$
and $T_{\fp_{\si'}}\fl_\tau$ are 
$\displaystyle{ \frac{\uu}{r(\tau,\si)} }$ and 
$\displaystyle{ \frac{-\uu}{r(\tau,\si')} }$, respectively.
\item  Let $\fp_v =f_e^{-1}(\fp_\si), \fp_{v'}=f_e^{-1}(\fp_{\si'})$
be the two torus fixed points in $\cC_e$. Then
the weights of $T$-action on $T_{\fp_v}\cC_e$ and
$T_{\fp_{v'}}\cC_e$ are
$\displaystyle{ \frac{\uu}{d r_{(e,v)} } }$ and 
$\displaystyle{ \frac{-\uu}{d r_{(e,v')} } }$, respectively.
\end{itemize}
By Example \ref{ex:edge-orb},  
$$
\ch_T(H^1(\cC_e, f_e^*\cL_i)-H^0(\cC_e, f_e^*\cL_i))
= \begin{cases}
- \displaystyle{ \sum_{j=0}^{ \lfloor da_i-\ep_i\rfloor } e^{\bw_i - (j+\ep_i) \frac{\uu}{d} } }, & a_i\geq 0,\\
\displaystyle{ \sum_{j=1}^{ \lceil \ep_i -da_i-1\rceil } e^{\bw_i+(j-\ep_i) \frac{\uu}{d} }  }, & a_i<0.
\end{cases}
$$
Note that $\bw_i - (j+\ep_i) \uu$ and $\bw_i + (j-\ep_i)\uu$  are nonzero for any $j\in \bZ$ since
$\bw_i$ and $\uu$ are linearly independent for $i=1,\ldots,r-1$. So
$$
\frac{e^T\left(H^1(\cC_e,f_e^*\cL_i)^m\right)}{e^T\left(H^0(\cC_e,f_e^*\cL_i)^m\right)}
=\frac{e^T\left(H^1(\cC_e,f_e^*\cL_i)\right)}{e^T\left(H^0(\cC_e,f_e^* \cL_i)\right)} =\mathbf{b}_i
$$
where $\mathbf{b}_i$ is defined by \eqref{eqn:b-orb}.

By Example \ref{ex:edge-orb} again, 
\begin{eqnarray*}
&& \ch_T(H^1(\cC_e, f_e^*T\fl_\tau)-H^0(\cC_e, f_e^*T\fl_\tau))\\
&=& \sum_{j \in \bZ,  -\langle \frac{d}{r(\tau,\si)} \rangle \leq j \leq \frac{d}{r(\tau,\si)}+ \frac{d}{r(\tau,\si')} - \langle \frac{d}{r(\tau,\si)} \rangle}
e^{\frac{\uu}{r(\tau,\si)}- (j+ \langle \frac{d}{r(\tau,\si)}\rangle)\frac{\uu}{d} } \\
&=&  = 1+\sum_{j=1}^{\lfloor \frac{d}{r(\tau,\si)} \rfloor} e^{j\frac{\uu}{d}}
+ \sum_{j=1}^{\lfloor \frac{d}{r(\tau,\si')}\rfloor} e^{-j\frac{\uu}{d}}.
\end{eqnarray*}
So
\begin{eqnarray*}
\frac{e^T(H^1(\cC_e,f_e^*T\fl_\tau)^m)}{e^T(H^0(\cC_e,f_e^*T\fl_\tau)^m)}
&=& \prod_{j=1}^{\lfloor \frac{d}{r(\tau,\si)}\rfloor}\frac{1}{j\frac{\uu}{d} }\prod_{j=1}^{\lfloor \frac{d}{r(\tau,\si')} \rfloor} \frac{1}{-j\frac{\uu}{d}}\\
&=& \frac{ (\frac{d}{\uu})^{\lfloor\frac{d}{r(\tau,\si)}\rfloor } }{\lfloor\frac{d}{r(\tau,\si)}\rfloor !}
\frac{ (-\frac{d}{\uu})^{\lfloor\frac{d}{r(\tau,\si')}\rfloor } }
{\lfloor\frac{d}{r(\tau,\si')}\rfloor !} 
\end{eqnarray*}
Therefore,
\begin{eqnarray*}
 \frac{e^T(H^1(\cC_e, f_e^* T\cX )^m )}
{e^T (H^0(\cC_e, f_e^*T\cX)^m) }
&=&\frac{e^T(H^1(\cC_e ,f_e^*T\fl_\tau)^m)}{e^T(H^0(\cC_e,f_e^*T\fl_\tau)^m)}
\cdot\prod_{i=1}^{r-1}\frac{e^T(H^1(\cC_e,f_e^*\cL_i)^m)}{e^T(H^0(\cC_e,f_e^*\cL_i)^m)}\\
&=&  \frac{ (\frac{d}{\uu})^{\lfloor\frac{d}{r(\tau,\si)}\rfloor } }
{\lfloor\frac{d}{r(\tau,\si)}\rfloor !}
\frac{ (-\frac{d}{\uu})^{\lfloor\frac{d}{r(\tau,\si')}\rfloor } }
{\lfloor\frac{d}{r(\tau,\si')}\rfloor !} \prod_{i=1}^{r-1} \mathbf{b}_i
\end{eqnarray*}
\end{proof}

From the above discussion, we conclude that
$$
\frac{e^T(B_5^m)}{e^T(B_2^m)} = \prod_{v\in V^2(\Ga), E_v= \{e,e'\} } \bh(e,v) 
\cdot \prod_{(e,v)\in F^S(\Ga)} \bh(e,v) \cdot \prod_{v\in V^S(\Gamma)}\bh(v)\cdot\prod_{e\in E(\Ga)}\bh(e)
$$
where $\bh(e,v)$, $\bh(v)$, and $\bh(e)$ are defined by \eqref{eqn:hev}, \eqref{eqn:hv}, \eqref{eqn:he}, respectively.
To unify the stable and unstable vertices, we define
$$
\bh(v):=\begin{cases}
\displaystyle{ \frac{1}{\bh(e,v)} }, & v\in V^1(\Ga)\cup V^{1,1}(\Ga), \quad E_v=\{e\},\\
\displaystyle{ \frac{1}{\bh(e,v)} =\frac{1}{\bh(e',v)} }, & v\in V^2(\Ga),\quad  E_v=\{e,e'\}.
\end{cases}
$$
Then
$$
\frac{e^T(B_5^m)}{e^T(B_2^m)} 
= \prod_{v\in V(\Ga)} \bh(v)\cdot \prod_{(e,v)\in F(\Ga)} \bh(e,v) \cdot \prod_{e\in E(\Ga)}\bh(e).
$$

\subsection{Contribution from each graph} \label{sec:each-graph-orb}
\subsubsection{Virtual tangent bundle} We have $B_1^f=B_2^f$, $B_5^f=0$. So
$$
T^{1,f}=B_4^f =\bigoplus_{v\in V^S(\Ga)}T\Mbar_{g_v, \vi_v}(\cB G_v),\quad
T^{2,f}=0.
$$
We conclude that
$$
[\prod_{v\in V^S(\Ga)}\Mbar_{g_v, \vi_v}(\cB G_v)]^\vir
=\prod_{v\in V^S(\Ga)}[\Mbar_{g_v, \vi_v}(\cB G_v) ].
$$
\subsubsection{Virtual normal bundle} 
Let $N^\vir_{\vGa}$ be the virtual bundle on $\cM_\vGa$ which corresponds to the virtual normal bundle
of $\cF_{\vGa}$ in $\MgXi$. Then 
\begin{eqnarray*}
\frac{1}{e_T(N^\vir_\vGa)} &=&\frac{e^T(B_1^m)e^T(B_5^m)}{e^T(B_2^m)e^T(B_4^m)}\\
&=& \prod_{v\in V(\Ga)}\frac{\bh(v)}{\prod_{e\in E_v}(w_{(e,v)}-\bar{\psi}_{(e,v)}/r_{(e,v)})} 
\prod_{(e,v)\in F(\Ga)}\bh(e,v) \cdot \prod_{e\in E(\Ga)}\bh(e)
\end{eqnarray*}

\subsubsection{Integrand} 
Given $\si\in \Si(r)$, let
$$
i_\si^*: A^*_T(\cX)\to A^*_T(\fp_\si) =\bQ[u_1,\ldots,u_r]
$$
be induced by the inclusion $i_\si:\fp_\si \to \cX$.  Given $\vGa\in \GgXi$, let
$$
i_\vGa^* :A^*_T(\MgXi)\to A^*_T(\cF_\vGa) \cong A^*_T(\cM_\vGa)
$$
be induced by the inclusion $i_\vGa:\cF_\vGa \to \MgXi$. Then
\begin{equation}\label{eqn:integrand-orb}
\begin{aligned}
&i_\vGa^*\prod_{i=1}^n\left(\ev_i^*\gamma_i^T \cup (\bar{\psi}_i^T)^{a_i}\right)\\
=&\prod_{\tiny \begin{array}{c} v\in V^{1,1}(E)\\ S_v=\{i\}, E_v=\{e\}\end{array}} i^*_{\si_v}\gamma_i^T (-w_{(e,v)})^{a_i} 
\cdot \prod_{v\in V^S(\Ga)}\Bigl(\prod_{i\in S_v} i^*_{\sigma_v}\gamma_i^T\prod_{e\in E_v}\bar{\psi}_{(e,v)}^{a_i}\Bigr)
\end{aligned}
\end{equation}
To unify the stable vertices in $V^S(\Ga)$ and the unstable vertices in $V^{1,1}(\Ga)$ , 
we use the following convention: for $a\in \bZ_{\geq 0}$,
\begin{equation}\label{eqn:one-one-a-orb}
\int_{\Mbar_{0,(c,c^{-1})}(\cB G)}\frac{\bar{\psi}_2^a}{w_1-\bar{\psi}_1}= \frac{(-w_1)^a}{|G|}.
\end{equation}
In particular, \eqref{eqn:one-one-orb} is obtained by setting $a=0$. With the 
convention \eqref{eqn:one-one-a-orb}, we may rewrite \eqref{eqn:integrand-orb} as
\begin{equation}
i_\vGa^*\prod_{i=1}^n\left(\ev_i^*\gamma_i^T \cup (\bar{\psi}_i^T)^{a_i}\right)=
\prod_{v\in V(\Ga)}\Bigl(\prod_{i\in S_v} i^*_{\sigma_v}\gamma_i^T\prod_{e\in E_v}\bar{\psi}_{(e,v)}^{a_i}\Bigr).
\end{equation}

The following lemma shows that the convention \eqref{eqn:one-one-a-orb} is consistent with
the stable case $\Mbar_{0,(c_1,\dots,c_n)}(\cB G)$, $n\geq 3$.
\begin{lemma}  Let $n,a$ be integers, $n\geq 2$, $a\geq 0$. Let $\vc=(c_1,\ldots, c_n)\in G^n$, where
$c_1 \cdots c_n=1$. Then
$$
\int_{\Mbar_{0,\vc}(\cB G)} \frac{\bar{\psi}_2^a}{w_1-\bar{\psi}_1}=
\begin{cases} \displaystyle{ \frac{\prod_{i=0}^{a-1}(n-3-i)}{a! |G| }w_1^{a+2-n} }, & n=2 \textup{ or } 0\leq a\leq n-3.\\
0, & \textup{ otherwise.}
\end{cases} 
$$
\end{lemma}
\begin{proof} The case $n=2$ follows from \eqref{eqn:one-one-a-orb}.
For $n\geq 3$, 
\begin{eqnarray*}
&& \int_{\Mbar_{0,\vc}(\cB G)}\frac{\psi_2^a}{w_1-\bar{\psi}_1}
=\frac{1}{w_1} \int_{\Mbar_{0,\vc}(\cB G)}\frac{\bar{\psi}_2^a}{1-\frac{\bar{\psi}_1}{w_1}}
=w_1^{a+2-n} \int_{\Mbar_{0,\vc}(\cB G)} \bar{\psi}_1^{n-3-a}\bar{\psi}_2^a\\
&& = w_1^{a+2-n} \cdot \frac{1}{|G|}\cdot \frac{(n-3)!}{(n-3-a)! a_!}=\frac{\prod_{i=0}^{a-1}(n-3-i)}{a!|G|} w_1^{a+2-n}.
\end{eqnarray*}
\end{proof}

\subsubsection{Integral}\label{sec:vGa-integral-orb}
Let
$$
i^* :A^*_T(\MgXi)\to A^*_T(\MgXi^T) 
$$
be induced by the inclusion $i:\MgXi^T\to \MgXi$. 
The contribution of 
$$
\int_{[\MgXi^T]^{\vir,T}} \frac{i^*\prod_{i=1}^n (\ev_i^*\gamma_i^T\cup (\bar{\psi}_i^T)^{a_i})}{e^T(N^\vir)}
$$
from the fixed locus $\cF_\vGa$ is given by
\begin{eqnarray*}
&& c_\vGa \prod_{e\in E(\Ga)}\bh(e)\prod_{(e,v)\in F(\Ga)} \bh(e,v) 
\prod_{v\in V(\Ga)}\Bigl(\prod_{i\in S_v}i_{\si_v}^*\gamma_i^T\Bigr)\\
&& \quad\quad \cdot \prod_{v\in V(\Ga)}\int_{\Mbar_{g_v,\vi_v}(\cB G_v)}
\frac{\bh(v)\cdot \prod_{e\in E_v}\bar{\psi}^{a_i}_{(e,v)} }{\prod_{e\in E_v}(w_{(e,v)}-\bar{\psi}_{(e,v)}/r_{(e,v)})} 
\end{eqnarray*}
where $c_\vGa\in \bQ$ is defined by \eqref{eqn:unified-c}.

\subsection{Sum over graphs} 
Summing over the contribution from each graph $\vGa$ given
in Section \ref{sec:vGa-integral-orb} above, we obtain the following
formula.
\begin{theorem}\label{main-orb}
\begin{equation}\label{eqn:sum-orb}
\begin{aligned}
& \langle \bar{\tau}_{a_1}(\gamma_1^T)\cdots \bar{\tau}_{a_n}(\gamma_n^T)\rangle^{\cX_T}_{g,\beta}\\
=& \sum_{\vGa\in G_{g,\vi}(\cX,\beta)} c_\vGa
\prod_{e\in E(\Ga)} \bh(e) \prod_{(e,v)\in F(\Ga)} \bh(e,v)
\prod_{v\in V(\Ga)}\Bigl(\prod_{i\in S_v} i_{\si_v}^* \gamma_i^T \Bigr) \\
&\cdot \prod_{v\in V(\Ga)}
\int_{\Mbar_{g,\vi_v} (\cB G_v) }
\frac{\bh(v)\prod_{i\in S_v} \bar{\psi}_i^{a_i}}{\prod_{e \in E_v} (w_{(e,v)}-\bar{\psi}_{(e,v)}/r_{(e,v)} )}.
\end{aligned}
\end{equation}
where $\bh(e)$, $\bh(e,v)$, $\bh(v)$ are given by
\eqref{eqn:he}, \eqref{eqn:hev}, \eqref{eqn:hv}, respectively, and we have
the following convention for the $v\notin V^S(\Ga)$:
\begin{eqnarray*}
&&\int_{\Mbar_{0,(1)}(\cB G)}\frac{1}{w_1-\bar{\psi}_2}= \frac{w_1}{|G|},\quad
\int_{\Mbar_{0,(c,c^{-1})}(\cB G)}\frac{1}{(w_1-\bar{\psi}_1)(w_2-\bar{\psi}_2)}=\frac{1}{|G|\cdot(w_1+w_2)},\\
&& \int_{\Mbar_{0,(c,c^{-1})}(\cB G)}\frac{\bar{\psi}_2^a}{w_1-\bar{\psi}_1}=\frac{(-w_1)^a}{|G|},\quad a\in \bZ_{\geq 0}.
\end{eqnarray*}
\end{theorem}

\end{document}